\documentclass[10pt]{amsart}

\makeatother

\usepackage[utf8]{inputenc}

\usepackage{xypic}
\usepackage{bm}

\usepackage{setspace}
\usepackage{fullpage}
\usepackage{braket}
\usepackage{verbatim}
\usepackage{lscape}
\usepackage{tabularx}
\usepackage{tikz-cd}
\usepackage{amsmath}
\usepackage{pgfplots}
\usepackage{enumitem}
\usepackage{stmaryrd}
\usepackage{amsthm,thmtools}
\usepackage{extarrows}

\usepackage{graphicx}
\usepackage{float}
\usepackage{graphics,amssymb}
\usepackage{amsfonts}
\usepackage{latexsym}
\usepackage{epsf}
\usepackage{mathrsfs}
\usepackage{bbm}

\usepackage{color}
\usepackage[framemethod=tikz]{mdframed}
\usepackage{ulem}
\usepackage{cancel}
\usepackage{tikz}
\usetikzlibrary{positioning,calc,intersections,through,backgrounds,matrix,arrows}
\usepackage{mathtools}
\usepackage{url}
\usepackage{fancybox}
\usepackage{dsfont}
\usepackage{bbding}
\usepackage{todonotes}
\usepackage[page,title]{appendix}
\usepackage{setspace}
\usetikzlibrary{fit}
\usepackage[low-sup]{subdepth}
\usepackage[square,sort,comma, numbers]{natbib}
\usepackage[T1]{fontenc} % Font package
\usepackage{libertine}
\usepackage[utf8]{inputenc}
\usepackage{arydshln}
\usepackage[none]{hyphenat}

\theoremstyle{plain}

\newtheorem{theorem}{Theorem}[section]

\newtheorem{lemma}[theorem]{Lemma}
\newtheorem{proposition}[theorem]{Proposition}
\newtheorem{corollary}[theorem]{Corollary}
\newtheorem{conjecture}[theorem]{Conjecture}
\newtheorem{claim}{Claim}

\newtheoremstyle{boldremark}
    {\dimexpr\topsep/2\relax} % space above
    {\dimexpr\topsep/2\relax} % space below
    {}          % body font
    {}          % indent amount
    {\bfseries} % theorem head font
    {.}         % punctuation after theorem head
    {.5em}      % space after theorem head
    {}          % theorem hed spec. (empty = "normal")

\theoremstyle{definition}
\newtheorem{remark}[theorem]{Remark}

\newtheorem{definition}[theorem]{Definition}

\newtheoremstyle{boldhypothesis}
    {\dimexpr\topsep/2\relax} % space above
    {\dimexpr\topsep/2\relax} % space below
    {}          % body font
    {}          % indent amount
    {\bfseries} % theorem head font
    {.}         % punctuation after theorem head
    {0.5em}      % space after theorem head
    {}          % theorem hed spec. (empty = "normal")
\theoremstyle{boldhypothesis}
\newtheorem{hypothesis}[theorem]{Hypothesis}

\newcommand{\Gal}[2]{\mathrm{Gal}(#1/#2)}
\newcommand{\hotimes}{\hat{\otimes}}

\newcommand{\Sel}{\mathrm{Sel}}
\newcommand{\Frac}{\mathrm{Frac}}

\newcommand{\dominant}[1]{\mathbf{\underline{#1}}}
\newcommand{\m}{\mathfrak{m}}
\newcommand{\frakm}{\mathfrak{m}}
\newcommand{\R}{\mathbb{R}}
\newcommand{\C}{\mathbb{C}}
\newcommand{\p}{\mathfrak{p}}

\newcommand{\Z}{\mathbb{Z}}
\newcommand{\Q}{\mathbb{Q}}

\newcommand{\cont}{\mathrm{cont}}
\newcommand{\val}{\mathrm{val}}
\newcommand{\Hom}{\mathrm{Hom}}

\newcommand{\Cyc}{\mathrm{cyc}}
\newcommand{\QQ}{\mathcal{Q}}

\newcommand{\stab}{\mathrm{stab}}

\newcommand{\AAA}{\mathbb{A}}
\newcommand{\Frob}{\mathrm{Frob}}
\newcommand{\GGG}{\Gamma}

\newcommand{\II}{{\mathbb{I}}}
\newcommand{\OO}{{\mathcal O}}
\newcommand{\Div}{\mathrm{Div}}

\newcommand{\LL}{\mathrm{L}}
\newcommand{\Gl}{\mathrm{GL}}

\newcommand{\weight}{\mathrm{wt}}

\newcommand{\TTT}{\bbT}
\newcommand{\BBB}{\mathcal{B}}
\newcommand{\CCC}{\mathcal{C}}
\newcommand{\YYY}{\mathcal{Y}}
\newcommand{\hhh}{\mathbb{H}}
\newcommand{\III}{\mathcal{I}}

\newcommand{\JJJ}{\mathcal{J}}

\newcommand{\GSp}{\mathrm{GSp}}
\newcommand{\DDD}{\mathscr{D}}

\newcommand{\Spec}{\mathrm{Spec}}

\newcommand{\mmm}{\mathfrak{m}}

\newcommand{\Iw}{\mathrm{Iw}}
\newcommand{\Fil}{\mathrm{Fil}}
\newcommand{\LLLL}{\mathfrak{L}}
\newcommand{\SLLL}{\mathscr{L}}
\newcommand{\pmbtwo}[2]{\pmb{#1},#2}
\newcommand{\rhopmb}[2]{\rho_{\pmbtwo{#1}{#2}}}
\newcommand{\RRR}{\mathcal{R}}
\newcommand{\Fitt}{\mathrm{Fitt}}
\newcommand{\charcomp}{\vartheta}
\makeatletter
\usepackage{hyperref}

\def\ulk{{\underline{k}}}

\newcommand{\powerseries}[1]{\llbracket{#1}\rrbracket}
\def\iso{\simeq}
\def\ord{{\rm ord}}
\def\bbT{\mathbb T}
\def\wtd{\widetilde}
\def\cH{\mathcal H}

\def\End{{\rm End}}
\def\cO{\mathcal O}
\def\cP{\mathcal P}
\def\cQ{\mathcal Q}

\def\hF{\mathcal F}
\def\hG{\mathcal G}
\def\cI{\mathrm Iw}
\def\pcT{\mathrm T}

\def\bfS{\mathbf S}

\newcommand{\pMX}[4]{\left(\begin{array}{rr}
#1&#2\\
#3&#4
\end{array}\right)}
\def\GL{{\rm GL}}
\def\frakg{\mathfrak g}
\def\bfone{{\mathbf 1}}
\newcommand{\beq}{\begin{equation}}
\newcommand{\eeq}{\end{equation}}

\def\frakX{\mathfrak X}

\def\Gm{\mathbb G_m}

\newcommand\dG[3]{\big[\,#1,\,#2;\,#3\big]}

\def\sL{\mathscr L}
\def\al{\alpha}

\def\wh{\widehat}

\def\Mat{\rm M}
\def\cyc{\varepsilon}

\def\onto{\twoheadrightarrow}
\newcommand{\norm}[1]{\|#1\|}
\usepackage[capitalize]{cleveref}

\def\onehalf{\frac{1}{2}}
\begin{document}
\title{On the congruence ideal associated to $p$-adic families of Yoshida lifts}
%\author{Ming-Lun Hsieh and Bharathwaj Palvannan}

\author{Ming-Lun Hsieh}
\address{Department of Mathematics, Astro-Math Building Room 412, No. 1, Sec. 4, Roosevelt Road, Taipei 10617, Taiwan}
\email{mlhsieh@math.ntu.edu.tw}

\author{Bharathwaj Palvannan}
\address{Department of Mathematics, Indian Institute of Science, CV Raman Road, Bengaluru 560012, India}
\email{bharathwaj@iisc.ac.in}

\begin{abstract}  
We study congruences involving $p$-adic families of Hecke eigensystems of Yoshida lifts associated with two Hida families (say $\hF,\hG$) of elliptic cusp forms. With appropriate hypotheses, we show that if a Hida family of genus two Siegel cusp forms admits a Yoshida lift at an appropriately chosen classical specialization, then all classical specializations are Yoshida lifts. Moreover, we prove that the characteristic ideal of the non-primitive Selmer group of (a self-dual twist of) the Rankin--Selberg convolution of $\hF$ and $\hG$ is divisible by the congruence ideal of the Yoshida lift associated with $\hF$ and $\hG$. Under an additional assumption inspired by pseudo-nullity conjectures in higher codimension Iwasawa theory, we establish the pseudo-cyclicity of the dual of the primitive Selmer group over the cyclotomic $\Z_p$-extension.
\end{abstract}
\maketitle

\tableofcontents

\section{Introduction}
\subsection{Motivation}
The idea to study congruences between Eisenstein series of $\Gl(2)$ over $\Q$ and cuspidal eigenforms dates back to Ribet's proof \cite{MR419403} of the converse to Herbrand's theorem. The idea can in fact be traced to earlier unpublished calculations of Greenberg and Monsky. Generalizations of this idea have been applied in Iwasawa theory with tremendous success, most notably first by Mazur--Wiles \cite{MR742853} to prove the Iwasawa main conjecture over $\Q$. The space of Yoshida lifts of a pair of primitive Hida families  $(\hF,\hG)$ of elliptic modular forms constitutes the appropriate analog of the Eisenstein series in our study of $2$-variable Iwasawa conjecture involving the Rankin--Selberg product of two Hida families $\hF$ and $\hG$. The genesis of the construction of Yoshida lifts of cuspforms goes back to Yoshida \cite{MR586427,MR758701}. There are generalizations of this construction in works of B\"{o}cherer--Schulze-Pillot \cite{MR1096467,MR1475167}.  We are interested in studying the $p$-adic family of Hecke eigensystems corresponding to the space of Yoshida lifts of $\hF$, $\hG$. We rely on works of Hida \cite{MR1954939} and Pilloni \cite{MR3059119} on Hida theory for $\GSp_4$ over $\Q$ involving coherent cohomology. \\

Our first result is a type of rigidity result proved under (among others) an assumption that amounts to the finiteness of the Rankin-Selberg Selmer group (as predicted by the Bloch--Kato conjecture) for suitably chosen classical specializations $(f_0,g_0)$ of the pair of primitive Hida families $(\hF,\hG)$. We prove that a $p$-adic family of non-primitive Hecke eigensystems of $\GSp(4)$ over $\Q$ that contains the Hecke eigensystem corresponding to the space of Yoshida lifts of $(f_0,g_0)$ is unique and, in fact, corresponds to the $p$-adic family of Hecke eigensystems associated to the space of Yoshida lifts of the Hida families $(\hF,\hG)$. Furthermore, we prove that the Galois representation associated to a branch of the $\GSp_4$ ordinary Hecke algebra is irreducible if the branch does not correspond to the Hecke eigensystem of Yoshida lifts of Hida families. These results also crucially use the modularity results of Diamond \cite{MR1405946} and Skinner--Wiles \cite{MR1928993}. See Theorem~\ref{thm:yoshidafamily} for the precise statement. \\

The idea to study congruences between cuspidal eigenforms of $\Gl(2)$ over $\Q$, on the other hand, dates back to an unpublished manuscript of Doi--Hida. This topic was then studied in a series of articles by Hida in the early 1980s.  Hida \cite{MR868300} introduced the general notion of a \textit{congruence module} --- which is a module over the Hecke algebra ---  to systematically keep track of these congruences.  When applied to study congruences between $p$-adic families of Eisenstein series of $\Gl(2)$ over $\Q$ and cuspidal eigenforms, the congruence module inherits an additional structure over Iwasawa algebra. It turns out to be isomorphic to a quotient of the Iwasawa algebra by the \textit{congruence ideal} in the Iwasawa algebra. Under some assumptions, Emerton \cite{MR1710074} (in level one) and Ohta \cite{MR1980312} (for a general level) have showed that the Kubota--Leopoldt $p$-adic $L$-function generates the congruence ideal. The interplay between the Hecke-module structure and the Iwasawa-module structure of the congruence module along with the subsequent link on the analytic side between the congruence ideal and the $p$-adic $L$-function is an important motivation for us to pursue this generalization and study the congruence ideal and the congruence module associated to $p$-adic families of Yoshida lifts and their applications to the Iwasawa--Greenberg main conjecture for Rankin--Selberg product of Hida families. See Theorem  \ref{thm:mainconj2} and Corollary \ref{cor:mainconj3} for the precise statements.  \vspace{1mm}

Our application of congruence ideals to main conjectures is inspired by works of Hida--Tilouine \cite{MR1209708,MR1269427}  towards anti-cyclotomic main conjectures for CM fields,  where they study congruences involving CM forms. The idea to use Yoshida lifts to support the Bloch--Kato conjecture for the Rankin--Selberg product of cusp forms dates back to works of Agarwal--Klosin \cite{MR3045201} and B\"{o}cherer--Dummigan--Schulze-Pillot \cite{MR2998926}.  Our third main source of inspiration is work of Urban on the Eisenstein-Klingen ideal \cite{MR1813234}. \\

Ohta \cite{MR1800769}, generalizing works of Harder--Pink \cite{MR1237103} and Kurihara \cite{MR1247385},  was able to provide an alternative proof of the Iwasawa main conjecture over $\Q$. Ohta's proof allowed him to obtain finer information about the structure of the classical Iwasawa module over both the Iwasawa algebra and the Hecke algebra. Ohta was also able to formulate hypotheses (see for example \cite[Theorem 2]{MR2164625}) under which the classical Iwasawa module can be shown to be a cyclic module over the Iwasawa algebra. However, Wake \cite{MR3317761} has shown that these hypotheses need not always hold and that the Iwasawa module need not always be cyclic over the Iwasawa algebra. One is then naturally led to ask whether there is a natural set of hypotheses under which the classical Iwasawa module can be shown to be  pseudo-cyclic over the Iwasawa algebra. Wake \cite{MR3317761} obtained such results on pseudo-cyclicity of the Iwasawa module under hypotheses linked to pseudo-nullity conjectures of Greenberg for totally real abelian number fields. We prove pseudo-cyclicity of the dual of the primitive Selmer group associated to the cyclotomic deformation of the Rankin--Selberg product of Hida families under an assumption inspired by pseudo-nullity conjectures in higher codimension Iwasawa theory. See Theorem \ref{thm:selmstructure} for the precise statement.

\subsection{Hida theory for elliptic modular forms} \label{subsec:HidafamilyGL2}

Let $p$ be an odd prime. Let $\OO$ be the ring of integers in a finite suitably large extension of $\Q_p$. We denote by $\mathbb{F}_q$ the residue field of $\OO$, where $q$ is a power of $p$. We fix once and for all an isomorphism $\overline{\Q}_p\cong \mathbb{C}$ and embeddings $\overline{\Q} \hookrightarrow \overline{\Q}_p\cong\mathbb{C}$.
Denote by $\omega:(\Z/p\Z)^\times\to \mu_{p-1}$ the mod-$p$ character and also view $\omega$ as a Dirichlet character modulo $p$ via $\overline{\Q}\hookrightarrow \mathbb{C}$. Let \[\hF=\sum_{n>0}a_n(\hF)q^n\in \II_\hF[[q]];\quad \hG=\sum_{n>0}a_n(\hG)q^n\in \II_\hG[[q]]\] be two  Hida families of ordinary elliptic modular forms with the tame conductors $N_{\hF}$ and $N_{\hG}$. The coefficient rings $\II_\hF$ and $\II_\hG$ over which the $q$-expansions are defined will be introduced shortly. 
Throughout this manuscript, we shall impose the following conditions:
 \begin{hypothesis}\label{Hpy1}\noindent
\begin{enumerate}[style=sameline, align=left,label=(\scshape{Irr}), ref=\scshape{Irr},partopsep=0pt,parsep=0pt]
  \item\label{lab:irr} As $G_\Q$-representations, the residual representations associated to $\hF$ and $\hG$ respectively (say $\overline{\rho}_\hF$ and $\overline{\rho}_\hG$) are irreducible.
\end{enumerate}

\begin{enumerate}[style=sameline, align=left,label=(\scshape{$p$-Dist}), ref=\scshape{$p$-Dist},partopsep=0pt,parsep=0pt]
  \item\label{lab:pdist} The restrictions to $\Gal{\overline{\Q}_p}{\Q_p}$ of the semi-simplifications of the residual representations $\overline{\rho}_\hF$ and $\overline{\rho}_\hG$ are non-scalar. 
\end{enumerate}

\begin{enumerate}[style=sameline, align=left,label=(\scshape{Mult-Free}), ref=\scshape{Mult-Free},partopsep=0pt,parsep=0pt]
  \item\label{lab:multfree} As $\Gal{\overline{\Q}}{\Q}$-representations, the residual representations $\overline{\rho}_\hF$ and $\overline{\rho}_\hG$ are non-isomorphic.
\end{enumerate}

\end{hypothesis}
Let $S$ denote the set of primes in $\Q$ containing $p$, $\infty$ and the primes dividing $N_{\hF}$ and $N_{\hG}$.  Let $\Sigma_0 = S \setminus \{p\}$. Following Greenberg's convention, we use the symbol $\Sigma_0$ to denote non-primitive Selmer groups indicating that we drop the local conditions at primes in $\Sigma_0$. Let $\Q_S$ denote the maximal extension of $\Q$ unramified outside $S$. Let $\GGG_\Cyc$ denote the Galois group of the cyclotomic $\Z_p$-extension $\Q_\Cyc$ over $\Q$.  Let $\mathbb{T}_{\hF,N_{\hF}}$ and $\mathbb{T}_{\hG,N_{\hG}}$ denote Hida's ordinary Hecke algebra passing through $\hF$ and $\hG$ with tame levels $N_{\hF}$ and $N_{\hG}$ respectively. Let $\II_{\hF}$ and $\II_{\hG}$ denote the integral closures in the fraction fields of the irreducible components passing through $\hF$ and $\hG$ respectively. The Hecke algebras are generated by  the Hecke operators $T_\ell$ and the normalized diamond operators $  \langle \ell\rangle$ at  primes $\ell \notin S$ along with the $U_p$ operator. The rings  $\II_{\hF}$ and $\II_{\hG}$ are finite extensions of the one variable power series rings $\Z_p[\![x_\hF]\!]$ and $\Z_p[\![x_\hG]\!]$ respectively, where $x_\hF$ and $x_\hG$ are the weight variables for the Hida families. A height one prime ideal of $\II_{\hF}$ corresponds to a classical specialization of weight $k$ if it contains $(1+x_\hF)^n-(1+p)^{nk}$ for some integer $n \geq 1$; similarly for $\II_{\hG}$ with $x_\hG$ in place of $x_\hF$.   We will assume that the integral closures of $\Z_p$ in $\II_{\hF}$ and $\II_{\hG}$ are equal (to $\OO$, extending scalars if necessary).  We have tautological characters \[\kappa_\hF:G_\Q \twoheadrightarrow \GGG_\Cyc \hookrightarrow \OO[\![\GGG_\Cyc]\!]^\times \cong \OO[\![x_\hF]\!]^\times, \qquad \kappa_\hG:G_\Q \twoheadrightarrow \GGG_\Cyc \hookrightarrow  \OO[\![\GGG_\Cyc]\!]^\times \cong  \OO[\![x_\hG]\!]^\times,\] where the (non-canonical) isomorphisms $ \OO[\![\GGG_\Cyc]\!] \cong \OO[\![x_\hF]\!]$ and $\OO[\![\GGG_\Cyc]\!] \cong  \OO[\![x_\hG]\!]$ are obtained by identifying a topological generator $\gamma_0$ of $\GGG_\Cyc$  with $1+x_\hF$ and $1+x_\hG$ respectively. We let
 \[\rho_\hF:\Gal{\Q_S}{\Q} \rightarrow \Gl_2(\II_{\hF}), \quad \rho_\hG:\Gal{\Q_S}{\Q} \rightarrow \Gl_2(\II_{\hG})\] 
 denote the Galois representations constructed by Hida \cite{MR848685} associated to $\hF$ and $\hG$ respectively. We let $\LL_\hF$ and $\LL_\hG$ denote the corresponding free modules, associated to these Galois representations, of rank two over $\II_{\hF}$ and $\II_{\hG}$ respectively. It is possible to factor the characters \[\det(\rho_\hF): G_\Q \rightarrow \OO[\![x_\hF]\!]^\times, \quad \det(\rho_\hG): G_\Q \rightarrow \OO[\![x_\hG]\!]^\times\] as the products
\begin{align}\label{eq:dettaut}
  \det(\rho_\hF) = \chi_\hF\cyc^{-1}\kappa_\hF , \qquad \det(\rho_\hG) = \chi_\hG \cyc^{-1} \kappa_\hG ,
\end{align}
where  $\chi_\hF:G_\Q \rightarrow \OO^\times$ and $\chi_\hG:G_\Q \rightarrow \OO^\times$ are tame characters of Hida families $\hF$ and $\hG$ respectively, and the character $\cyc:G_\Q \rightarrow \Z_p^\times$ is  the $p$-adic cyclotomic character. %The integers $k_{f_0}$ and $k_{g_0}$ are the weights of specializations $f_0$ and $g_0$, that are chosen in the next subsection. whereas $\omega:G_\Q \rightarrow \Z_p^\times$ is the $p$-adic Teichm\"uller character\\

Let $\II_{\hF,\hG}$ denote the normalization of the completed tensor product $\II_{\hF} \hat{\otimes} \II_{\hG}$ over $\OO$. Note that we have natural inclusions $i_F: \II_{\hF} \hookrightarrow \II_{\hF,\hG}$ and $i_G: \II_{\hG} \hookrightarrow \II_{\hF,\hG}$.  Let $\II_{\hF,\hG}[\![\GGG_\Cyc]\!]$ be the usual Iwasawa algebra over $\II_{\hF,\hG}$. This topological $\II_{\hF,\hG}$-algebra is isomorphic to the power series ring $\II_{\hF,\hG}[\![x]\!]$, the isomorphism being (non-canonically) defined by sending a chosen topological generator $\gamma_0$ of $\GGG_\Cyc$ to $x+1$. Consider the tautological character 
\[\widetilde{\kappa}: G_S \twoheadrightarrow \Gamma_\Cyc \hookrightarrow \OO[\![\Gamma_\Cyc]\!]^\times.\]

Let $\rhopmb{4}{3}:\Gal{\Q_S}{\Q} \rightarrow \Gl_4\left({\II_{\hF,\hG}\powerseries{\Gamma_\Cyc}}\right)$ denote the Galois representation obtained by the action of $G_\Q$ on the following $\II_{\hF,\hG}[\![\Gamma_\Cyc]\!]$-module of rank four:
\[\sL_{\hF,\hG}:=\Hom_{\II_{\hF,\hG}}\bigg( \LL_\hF \otimes_{\II_{\hF}} \II_{\hF,\hG} , \  \LL_\hG \otimes_{\II_{\hG}} \II_{\hF,\hG} \bigg) \otimes_{\II_{\hF,\hG}} \II_{\hF,\hG}[\![\Gamma_\Cyc]\!](\widetilde{\kappa}^{-1}).\]
Here, the subscript $4$ denotes the dimension of the Galois representation, while the subscript $3$ represents the number of variables in the corresponding $p$-adic $L$-function, which is also one less than the Krull dimension of the ring over which $\rhopmb{4}{3}$ is defined. 
\begin{remark}
We impose the condition \ref{lab:irr} that the residual representations associated to $\hF$ and $\hG$ are irreducible along with the hypothesis \ref{lab:pdist} so that the Galois representations $\rhopmb{4}{3}$ and $\rhopmb{4}{2}$ (introduced in equation \ref{eq:lattice3var}) satisfy Greenberg's Panchishkin condition which in turn allows us to formulate the Iwasawa--Greenberg main conjecture in the Rankin--Selberg setting. Furthermore the hypothesis \ref{lab:irr} allows us to apply the $R^{\mathrm{ord}}=T$ theorem of Skinner--Wiles \cite{MR1928993} (which is a strengthening of a theorem of Diamond \cite{MR1405946}). The hypothesis \ref{lab:multfree} plays a key role while using results of Bella\"iche--Chenevier \cite{MR2656025} on pseudocharacters. 

\end{remark}

\subsection{Iwasawa--Greenberg main conjectures in the Rankin--Selberg setting}

We consider Conjectures \ref{conj:mainconj3} and \ref{conj:mainconj2}  that fall under the umbrella of Iwasawa--Greenberg main conjectures. There are two choices of Greenberg's Panchishkin condition for the Galois representation $\rhopmb{4}{3}$. We say that $\hF$ is the dominant Hida family if the weights of the specializations of $\hF$ are larger than the weights of the specializations of $\hG$ amongst the critical specializations determined by a chosen Panchishkin condition. One has a similar notion for when $\hG$ is the dominant Hida family. Depending on the choice of Panchishkin condition, the corresponding Selmer groups and $p$-adic $L$-functions differ. We will usually denote our choice of dominant Hida family using an underlined font. Without loss of generality, let us consider the case when the dominant Hida family is $\hF$.  One can, in fact, consider a non-primitive $3$-variable Rankin--Selberg $p$-adic $L$-function $\theta^{\Sigma_0}_{\pmbtwo{4}{3}}$ -- also constructed by Hida \cite{MR774534,MR976685,MR1216135} --- in the fraction field of $\II_{\hF,\hG}[\![\GGG_\Cyc]\!]$. See work of Dasgupta \cite{MR3514962} or work of the first author and Chen \cite{chenprimitive} for a construction of the primitive $p$-adic $L$-function. For these $3$-variable $p$-adic $L$-functions, in addition to the weight variables of $\hF$ and $\hG$, one can vary the cyclotomic variable $s$. On the algebraic side, we have a discrete non-primitive Selmer group $\Sel^{\Sigma_0}_{\rhopmb{4}{3}}(\Q)$, whose Pontryagin dual is a finitely generated $\II_{\hF,\hG}[\![\GGG_\Cyc]\!]$-module. For the Selmer group, the term \textit{non-primitive} indicates that we drop the local condition at primes in $\Sigma_0$ defining the corresponding Selmer group. For the $p$-adic $L$-function, the term \textit{non-primitive} is used to indicate that we omit the Euler factors at those same primes from the corresponding $L$-function. 

\begin{conjecture}[Conjecture 4.1 in \cite{MR1265554}]\label{conj:mainconj3}
  The $\II_{\hF,\hG}[\![\GGG_\Cyc]\!]$-module $\Sel^{\Sigma_0}_{\dominant{\hF},\hG}(\Q)^\vee$ is torsion. Furthermore, we have the following equality of divisors of $\II_{\hF,\hG}[\![\GGG_\Cyc]\!]$:
  \begin{align}
    \Div(\Sel^{\Sigma_0}_{\rhopmb{4}{3}}(\Q)^\vee) = \Div(\theta^{\Sigma_0}_{{\pmbtwo{4}{3}}}).
  \end{align}
\end{conjecture}
Here, the divisor $\Div(M)$ of a finitely generated torsion module $M$ 
over an integrally closed Noetherian domain $A$ is defined as $\sum_{P}\mathrm{len}_{A_P}(M_P)\cdot[P]$, 
where the sum runs over height-one primes $P$ of $A$; 
see \cite{MR1727221}. 

One of our primary goals in this article is to obtain a partial result towards a two-variable version of the inequality \[    \Div(\Sel^{\Sigma_0}_{\rhopmb{4}{3}}(\Q)^\vee) \geq  \Div(\theta^{\Sigma_0}_{{\pmbtwo{4}{3}}}),\] by studying the congruence ideals of Hida families of Yoshida lifts. Throughout we shall make the following fundamental hypothesis, which is necessary when one wants to consider Yoshida lifts of cusp forms:

\begin{enumerate}[style=sameline,  align=left,label=(\scshape{Yos-Neb}), ref=\scshape{Yos-Neb}]
  \item\label{lab:yos-neb}  $\chi_{\hF} = \chi_{\hG}$.
\end{enumerate}

    \begin{enumerate}[style=sameline,  align=left,label=(\scshape{Yos-disc}), ref=\scshape{Yos-disc}]
  \item\label{lab:yos} There exists a finite prime $\ell_0 \neq p$ such that for some pair $(f,g)$  of classical specializations of $(\hF,\hG)$, the irreducible admissible representations $\pi_{f,\ell_0}$ and $\pi_{g,\ell_0}$ of $\Gl_2(\Q_{\ell_0})$ (in the sense of \cite[Definition 6.1]{MR0379375}) are both discrete series representations (in the sense of \cite[Section 17.4]{MR2234120}). In particular, this implies that $\ell_0$ divides both tame conductors $N_\hF$ and $N_\hG$.
\end{enumerate}
We remark that by the rigidity of automorphy types of Hida families \cite[Lemma 2.14]{FO12}, if a prime $\ell_0$ satisfies \ref{lab:yos} for a single classical specialization of $(\hF,\hG)$, then $\ell_0$ also satisfies \ref{lab:yos} for any classical specialization of $(\hF,\hG)$.

%Lei--Loeffler--Zerbes \cite{MR3224721} and Kings--Loeffler--Zerbes \cite{MR4060872,MR3637653} are able to utilize the Euler systems of Beilinson--Flach elements to deduce, under suitable assumptions, the finiteness of the Rankin--Selberg Selmer groups. This motivates the hypothesis that we introduce below --- labeled \ref{lab:tor}.
We further assume the following technical hypotheses:
\begin{hypothesis}\label{Hyp2}\noindent
There exist classical specializations $f_0$, $g_0$ of the Hida families $\hF$ and $\hG$, with level prime to $p$, satisfying the following conditions:

\begin{enumerate}[style=sameline, align=left,label=(\scshape{Tor}), ref=\scshape{Tor}]
  \item\label{lab:tor}  The Rankin--Selberg Selmer group associated to $f_0$, $g_0$ over $\Q$  --- defined in equation \eqref{eq:defRS} --- is finite.       
  \end{enumerate}
 \begin{enumerate}[style=sameline,  align=left,label=(\scshape{CT}), ref=\scshape{CT}]   
          \item\label{lab:uni} $\mathrm{weight}(f_0) >\mathrm{weight}(g_0)$.
  \end{enumerate}
   
	\begin{enumerate}[style=sameline,  align=left,label=(\scshape{Yos-Det}), ref=\scshape{Yos-Det}]       
	\item\label{lab:yosdet} $\mathrm{weight}(f_0) \equiv \mathrm{weight}(g_0) \ (\mathrm{mod} \ 2(p-1))$.
	\end{enumerate}	

\end{hypothesis}

\begin{remark} All of the hypotheses \ref{lab:yos-neb},  \ref{lab:yosdet} and \ref{lab:yos} are required for the construction of Yoshida lifts. See Section \ref{sec:yoshidalifts}.  The local representation $\pi_{f,\ell_0}$ participates in the local Jacquet--Langlands correspondence for $\Gl(2)$  precisely when $\pi_{f,\ell_0}$ is a discrete series \cite[Section 56]{MR2234120}. The hypothesis \ref{lab:yos} is equivalent to the fact that there exists at least one finite prime $\ell_0$ such that both $\pi_{f_0,\ell_0}$ and $\pi_{g_0,\ell_0}$ participate in the local Jacquet--Langlands correspondence for $\Gl(2)$, which is then required when applying Arthur's multiplicity formula to show the representation theoretic existence of Yoshida lifts of $f_0$ and $g_0$.
\end{remark}

\begin{remark}
The assertion of hypothesis (\ref{lab:tor}) would be predicted by the Bloch--Kato conjectures, which would relate the finiteness of these Selmer groups to non-vanishing of the Rankin--Selberg $L$-function $L\left(f_0\otimes g_0, \chi_{f_0},\frac{k_{f_0}+k_{g_0}}{2}-1\right)$ associated to $f_0$ and $g_0$. By combining well-known results of Ramakrishnan \cite{RD00} and Jacquet-Shalika \cite{JS76}, one can deduce that $L\left(f_0\otimes g_0,\chi_{f_0}, \frac{k_{f_0}+k_{g_0}}{2}-1\right)\neq 0$. If $f_0$ and $g_0$ are CM forms that have complex multiplication by the same imaginary quadratic field $K$, the specialization of the Galois representation $\rhopmb{4}{3}$ at $f_0,g_0$ splits as a direct sum of two induced characters over $K$. One can then use Rubin's results \cite{MR1079839} towards the Iwasawa main conjectures over $K$ to deduce \ref{lab:tor} from the non-vanishing of the corresponding $L$-values. 

In the case when $f_0$ and $g_0$ are both non-CM forms and are not Galois-conjugate twists, one can apply results of \cite{MR3637653}. When the product of the Nebentypen $\chi_\hF^2$ is non-trivial and the hypothesis labeled Hyp(BI) of \cite{MR3637653} is satisfied, then \cite[Theorem 11.6.6]{MR3637653} can be applied to the critical character $\tau = \chi_\hF \epsilon^{\frac{k_{f_0}+ k_{g_0}}{2}-2}$ along with \cite[Proposition 11.2.9]{MR3637653} to establish \ref{lab:tor}. To establish the hypothesis (HypBI), one should then appeal to the results \cite[Proposition 4.2.1, Remark 4.2.2]{MR3576325}. There are two theoretical limitations to producing a large class of interesting examples in an abstract manner. By interesting, we mean cases where $p$ divides the appropriate Rankin-Selberg $L$-value so as to require congruences between Yoshida lifts and stable forms. Firstly, the results establishing Hyp(BI) in \cite[Proposition 4.2.1, Remark 4.2.2]{MR3576325} are asymptotic as they hold for all but finitely many primes. See the conclusion of \cite[Theorem 3.2.2]{MR3576325}. Secondly, it is only known conjecturally that 100\% of the primes are \textit{ordinary} for general non-CM modular forms. We thus leave it to the interested computational expert to produce numerical examples satisfying Hyp(BI) (thereby guaranteeing \ref{lab:tor}) along with our hypothesis \ref{lab:pdist-gsp4} given below (and where $p$ divides the appropriate Rankin--Selberg $L$-value).
\end{remark}

The hypothesis \ref{lab:yos-neb} allows us to consider a twist of a specialization of $\rhopmb{4}{3}$, which is essentially self-dual. Let $\rhopmb{4}{2}:\Gal{\Q_S}{\Q} \rightarrow \Gl_4(\II_{\hF,\hG})$ denote the Galois representation obtained by the action of $G_\Q$ on the following $\II_{\hF,\hG}$-module of rank four:
\begin{align}\label{eq:lattice3var}
\Hom_{\II_{\hF,\hG}}\bigg( \LL_\hF \otimes_{\II_{\hF}} \II_{\hF,\hG} , \  \LL_\hG \otimes_{\II_{\hG}} \II_{\hF,\hG} \bigg) \otimes_{\II_{\hF,\hG}} \II_{\hF,\hG}\left(\sqrt{\kappa_\hF\kappa_\hG^{-1}}\right).
\end{align}
Here, the subscript $4$ denotes the dimension of the Galois representation, while the subscript $2$ represents the number of variables in the corresponding $p$-adic $L$-function, which is also one less than the Krull dimension of the ring over which $\rhopmb{4}{2}$ is defined. 

Greenberg's methodology produces a discrete non-primitive Selmer group $\Sel^{\Sigma_0}_{\rhopmb{4}{2}}(\Q)$, whose Pontryagin dual is conjectured to be a finitely generated torsion $\II_{\hF,\hG}$-module. Hida \cite{MR774534,MR976685,MR1216135} has constructed a non-primitive $2$-variable Rankin--Selberg $p$-adic $L$-function $\theta^{\Sigma_0}_{\pmbtwo{4}{2}}$ in the fraction field of $\II_{\hF,\hG}$. For these $p$-adic $L$-functions, one can vary the weights of $\hF$ and $\hG$. We recall the non-primitive variation of the Iwasawa--Greenberg main conjecture associated to $\rhopmb{4}{2}$ (\cite[Conjecture 4.1]{MR1265554}).

\begin{conjecture}\label{conj:mainconj2}
  The $\II_{\hF,\hG}$-module $\Sel^{\Sigma_0}_{\rhopmb{4}{2}}(\Q)^\vee$ is torsion. Furthermore, we have the following equality of divisors of $\II_{\hF,\hG}$:
  \begin{align}
    \Div(\Sel^{\Sigma_0}_{\rhopmb{4}{2}}(\Q)^\vee) = \Div(\theta^{\Sigma_0}_{\pmbtwo{4}{2}}).
  \end{align}
\end{conjecture}

More precisely, we introduce an ideal ${\CCC}$ in $\II_{\hF,\hG}$, which we call the \textit{congruence ideal}, and that is expected to satisfy, under a suitable $p$-distinguished hypothesis, the following inequalities of divisors of $\II_{\hF,\hG}$:
\begin{align}\label{eq:expected}
  \Div\left(\Sel^{\Sigma_0}_{\rhopmb{4}{2}}(\Q)^\vee\right) \stackrel{?}{\geq} \Div\left({\CCC}\right) \stackrel{?}{\geq} \Div\left(\theta^{\Sigma_0}_{\pmbtwo{4}{2}}\right).
\end{align}

One of our main theorems (Theorem \ref{thm:mainconj2}) establishes the following inequality, under suitable assumptions: \[\Div(\Sel^{\Sigma_0}_{\rhopmb{4}{2}}(\Q)^\vee) \geq\Div({\CCC}).\]

\begin{remark} 
The formulation of the Iwasawa--Greenberg main conjectures in \cite{MR1265554} relates the primitive Selmer group to the primitive $p$-adic $L$-function, while the formulations in Conjectures \ref{conj:mainconj2} and \ref{conj:mainconj3} relate the non-primitive Selmer groups to the non-primitive $p$-adic $L$-functions. The idea that non-primitive versions of the main conjectures should be equivalent to their primitive versions is well-known and goes back to works of Greenberg \cite{MR0444614}, Greenberg--Vatsal \cite{MR1784796} etc. For reasons of space, we do not discuss this equivalence in our paper. However, for some technical reasons, it will be advantageous for us to work with the non-primitive versions. For instance, the theory of Hecke operators at bad primes is not well developed for $\GSp_4$. Although there might be candidates for these Hecke operators at bad places in works of Andrianov \cite{MR884891} and Roberts--Schmidt \cite{MR2344630}, we prefer to avoid this latter complication altogether. It is for these reasons that we simply concentrate on the \textit{non-primitive main conjecture} and work with the reduced Hecke algebra, thereby not including Hecke operators at bad primes. 
\end{remark}

\subsection{Hida Theory for Siegel modular forms of genus two}\label{subsec:HidafamilyGSP4}

Our starting point is the construction of an automorphic representation, say $\varPi(f_0,g_0)$, in the space of Yoshida lifts of $f_0$ and $g_0$. As seen in Section \ref{sec:Galoisreps}, one can observe that this automorphic representation is $p$-ordinary and has weight
\begin{align}\label{eq:siegelweight}
(k_1,k_2) = \left(\dfrac{\weight(f_0)+\weight(g_0)}{2},\dfrac{\weight(f_0)-\weight(g_0)}{2}+2\right).
\end{align}

We shall impose the following $p$-distinguished hypothesis for $\GSp_4$, which implies and is stronger than the hypothesis \ref{lab:pdist}.
\begin{enumerate}[style=sameline, align=left,label=(\scshape{$p$-Dist-GSP4}), ref=\scshape{$p$-Dist-GSp4},partopsep=0pt,parsep=0pt]
  \item\label{lab:pdist-gsp4} The four residual characters appearing in the restriction to $\Gal{\overline{\Q}_p}{\Q_p}$ of the semi-simplification of the residual representation associated to the Yoshida lift $\varPi(f_0,g_0)$ are distinct. 
\end{enumerate}

\begin{remark}
The hypothesis \ref{lab:pdist-gsp4} is stronger than \ref{lab:pdist} and we require the stronger hypothesis for two important reasons. Firstly, it is required to show that the change of basis matrix, used in Theorem \ref{thm:pseudorepBC}  to establish ordinarity of the local Galois representation, has the required integral structure. Secondly, this hypothesis allows us to work with the strict Selmer group and establish precise control theorems in Section \ref{sec:proofthm2}, as we avoid the phenomenon of trivial zeroes. 	
\end{remark}

We will work within the framework of Hida theory for Siegel modular forms developed by Hida \cite{MR1954939} and Pilloni \cite{MR3059119}. Let $\Lambda:=\Z_p[\![x_1,x_2]\!]$ be the two-variable power series ring with $x_1$ and $x_2$ weight variables. Let $N$ be a positive integer prime to $p$. Let $\hhh(N)$ be the big Hecke algebra acting on the space of ordinary $\Lambda$-adic cuspidal Siegel modular forms of principal level subgroup $\Gamma(N)\subset \GSp_4(\AAA_f)$. The Hecke algebra is a $\Lambda$-algebra generated by the classical Hecke operators $T_\ell$, $R_\ell$, $S_\ell$ (and its inverse $S_\ell^{-1}$) operators for the primes $\ell\notin S$, along with the $U_p$-operators $U_\cP$ and $U_\cQ$. It is known that the ring $\hhh(N)$ is a reduced, torsion-free and finite $\Lambda$-algebra. Choose $N$ such that $\varPi(f_0,g_0)$ has a non-zero $\Gamma(N)$-fixed vector. Then the Hecke eigensystem associated with $\varPi(f_0,g_0)$ determines a ring homomorphism 
\[\lambda_{f_0,g_0}: \hhh(N)  \rightarrow \overline{\Q}_p.\] Let $\m$ denote the maximal ideal containing $\ker\lambda_{f_0,g_0}$. Let $\bbT:=\hhh(N)_\m$ denote the local component of $\hhh(N)$ passing through the Yoshida lift $\varPi(f_0,g_0)$. Let $\{\eta_i\}_{1 \leq i \leq m}$ denote the set of minimal prime ideals of $\bbT$. %Observe that the intersection of the minimal prime ideals $\displaystyle \bigcap \limits_{i=1}^m \eta_i$ equals $(0)$. 
Since $\bbT$ is a reduced (Henselian) local ring, the natural map $\displaystyle  \bbT \hookrightarrow \prod_{i=1}^m \bbT_{\eta_i}$ is an injection and the localizations $\bbT_{\eta_i}$ are fields. %Since $\Z_p[\![X_1,X_2]\!] \hookrightarrow \bbT$ is a finite integral extension,
We have the following decomposition of $\mathrm{Frac}(\Lambda)$-algebras:
\begin{align} \label{eq:artindecomp}
\bbT\otimes_{\Lambda}\mathrm{Frac}(\Lambda)\cong
  \prod_{i=1}^m \bbT_{\eta_i}.
\end{align}

Associated to each minimal prime ideal $\eta_i$ (also called \textit{branch} or a \textit{Hida family}), one can use the theory of pseudocharacters (for example, see the book by Bella\"iche--Chenevier \cite{MR2656025}) to construct a four-dimensional continuous semi-simple Galois representation $\rho_{\eta_i}:G_S \rightarrow \Gl_4(\bbT_{\eta_i})$ associated to it. The trace of $\rho_{\eta_i}$ interpolates (in a suitable sense) traces of Galois representations constructed by Laumon  \cite{MR2234859}, Taylor \cite{MR1240640} and  Weissauer \cite{MR2234860} associated to $p$-ordinary $\GSp_4$ cuspforms.   

\begin{comment}
 Let $Q_\ell(X)\in \hhh(N)[X]$ be the Hecke polynomial at $\ell$ defined by 
\beq\label{E:1}X^4-T_\ell X^3+(\ell R_\ell +\ell(\ell^2+1)S_\ell)X^2-\ell^3T_\ell S_\ell X+\ell^6S_\ell^2.\eeq
We define the Hecke eigensystem $\phi:\hhh\to \II\otimes \II$ of Yoshida lifts by 
\[\phi(Q_\ell(X)=(X^2-a_\ell(\hF)X+\psi_1(\ell)\ell^{-1}\kappa_{X_1}(\ell))(T^2-a_\ell(\hG)\kappa_{X_1}^\onehalf \kappa_{X_2}^{-\onehalf}(\ell)\chi(\ell)X+\ell^{-1}\psi_1(\ell)\kappa_{X_1}(\ell))\]
\end{comment}

\begin{definition}\label{def:Yoshidalifthidafamilies}
We say that the natural map $\phi_i: \bbT \rightarrow {\bbT}_{\eta_i}$ is the \textit{$p$-adic family of Hecke eigensystems associated to the space of Yoshida lifts of $\hF$, $\hG$} if there exists injective ring maps $\iota_F: \II_{\hF} \hookrightarrow {\bbT}_{\eta_i}$ and $\iota_G: \II_{\hG} \hookrightarrow {\bbT}_{\eta_i}$ such that the following equality inside ${\bbT}_{\eta_i}$ holds  for all primes $\ell$ not dividing $Np$:
\[\mathrm{Trace}\left(\rho_{\eta_i}(\Frob_\ell)\right) = \mathrm{Trace}\left(\left(\rho_\hG \left(\sqrt{\kappa_\hF\kappa_\hG^{-1}}\right) \oplus \rho_\hF \right)(\Frob_\ell)\right).\]

\end{definition}

Let $\eta$ denote a(ny) minimal prime ideal of $\bbT$ contained in  $\ker(\lambda_{f_0,g_0})$. Without loss of generality, we let the minimal prime ideal $\eta$ equal $\eta_1$. Expanding on the decomposition in equation (\ref{eq:artindecomp}), we write
\begin{align}\label{eq:GSp4decomphecke}
	\bbT\otimes_{\Lambda}\mathrm{Frac}(\Lambda) \cong  {\bbT}_{\eta_1} \times \underbrace{{\bbT}_{\eta_2} \times \cdots \times {\bbT}_{\eta_r}}_{\substack{\text{$p$-adic families} \\
	\text{of Hecke eigensystems} \\ \text{not associated to Yoshida lifts} }} \times \underbrace{{\bbT}_{\eta_{r+1}} \times {\bbT}_{\eta_m}}_{\substack{\text{Yoshida lifts of} \\ \text{ Hida families  $\hF',\hG'$  } \\ \text{ where } (\hF',\hG') \neq (\hF,\hG)}}.
\end{align}

\begin{enumerate}[leftmargin=0pt,label=(\roman*)]
\item We let $\TTT^{\rm st}$ denote the image of $\bbT \rightarrow \underbrace{{\bbT}_{\eta_2} \times \cdots \times {\bbT}_{\eta_r}}_{\substack{\text{$p$-adic families} \\
	\text{of Hecke eigensystems} \\ \text{not associated to Yoshida lifts} }}$. We call $\bbT^{\rm st}$ the stable component of $\bbT$. 
  \item\label{item:yoshidaideal}   We let $\TTT_\YYY$ denote the image of the map $\TTT\rightarrow \TTT_{\eta}$. The ring $\TTT_\YYY$ is a domain. Let $\widetilde{\TTT}_\YYY$ denote the integral closure of $\TTT_\YYY$ in its fraction field. %Let $\JJJ$ denote the kernel of the map $\TTT \rightarrow \TTT_{\eta}$. The image of $\JJJ$ in $\bbT^{\rm st}$ is called the \textit{Yoshida ideal} (following Agarwal--Klosin \cite{MR3045201}).
   % \[\JJJ \coloneqq \ker\left(\TTT \rightarrow \TTT_{\eta}\right).\]  
  %\item We let $\TTT^{\rm st}$ denote the image of the map \[\TTT \rightarrow \prod \limits_{i \neq 1}\TTT_{\eta_i}.\] The set $\{\eta_i\}_{2 \leq i \leq m}$ could be empty, in which case $\TTT^{\rm st}$ would denote the ``zero ring''. It is not hard to see that the Yoshida ideal $\JJJ$ can naturally be viewed as an ideal of $\TTT^{\rm st}$.
\end{enumerate}

%\begin{remark}
%For our proofs in Section \ref{sec:proofs}, we will  need to choose an auxiliary $\Gl_2$-tame level $N_2$ divisible by both $N_{\hF}$ and $N_{\hG}$, as given below:
%\begin{align}\label{eq:GL2tamelevel}
%	N_2 \coloneqq \prod_{\ell \mid N_{\hF}N_{\hG}} \ell^{\max\left\{2,\val_\ell(N_{\hF}) +  2\val_p(\ell-1),\val_\ell(N_{\hG}) +  2\val_p(\ell-1)\right\}}.  
%\end{align}
% We will also consider Hida's Hecke algebras $\bbT_{\hF,N_2}$ and $\bbT_{\hG,N_2}$, defined analogous to $\mathbb{T}_{\hF,N_{\hF}}$ and $\mathbb{T}_{\hG,N_{\hG}}$, but with tame level $N_2$ instead. One could have primes $\ell \neq p$ dividing $N_{\hG}$ but not $N_{\hF}$.	 Such primes could contribute to introducing new irreducible components for the ring $\bbT_{\hF,N_2}$ which were not present in the ring $\mathbb{T}_{\hF,N_{\hF}}$. For example, one could have primes $\ell \neq p$ participating in Ribet's level raising congruences (even if we work with squarefree levels). For the tame level chosen in equation (\ref{eq:GL2tamelevel}), one could also twist the Hida family $\hF$ by characters with conductor $\ell$ and order $p$, if a prime $\ell$ dividing $N_{\hG}$ is also congruent to $1$ modulo $p$. For these reasons, one could truly have new Yoshida lifts of Hida families appearing in the decomposition given in equation (\ref{eq:GSp4decomphecke}). 
%\end{remark}

\subsection{Rigidity results -- Existence and Uniqueness: Theorem \ref{thm:yoshidafamily}} There are two aspects to our first theorem.  Theorem \ref{thm:yoshidafamily}\ref{thmcond:yoshidafamily} and \ref{thm:yoshidafamily} \ref{thmcond:unique}  are existence and uniqueness statements: they establish both the uniqueness of the $\GSp_4$ Hida family passing through the Yoshida lift $\varPi(f_0,g_0)$ and that every classical specialization of the $\GSp_4$ Hida family passing through $\varPi(f_0,g_0)$ is also a Yoshida lift; in fact it is obtained as the Yoshida lift of the classical specializations of $\Gl_2$ Hida families $\hF$ and $\hG$ respectively. %Theorem \ref{thm:yoshidafamily}\ref{thmcond:otheryoshida} establishes that the $4$-dimensional Galois representation associated to the Hida families of $\TTT^{\rm st}$ is irreducible. 

\begin{restatable}{Theorem}{theoremone}\label{thm:yoshidafamily}
  Suppose that the hypotheses \ref{lab:irr}, \ref{lab:pdist-gsp4}, \ref{lab:multfree},  \ref{lab:tor}, \ref{lab:uni}, \ref{lab:yosdet},   \ref{lab:yos} and \ref{lab:yos-neb} hold. Then, the following statements hold:
  \begin{enumerate}[label=(\roman*)]
    \item\label{thmcond:yoshidafamily} There is a natural isomorphism $\II_{\hF,\hG} \cong \widetilde{\TTT}_\YYY$ of integrally closed domains. With respect to the injections $$ \II_{\hF} \stackrel{i_F}\hookrightarrow \II_{\hF,\hG} \cong \widetilde{\TTT}_\YYY \hookrightarrow \TTT_{\eta}, \qquad \II_{\hG} \stackrel{i_G}\hookrightarrow \II_{\hF,\hG} \cong \widetilde{\TTT}_\YYY \hookrightarrow \TTT_{\eta},$$
          the natural map $\TTT \rightarrow \TTT_{\eta}$ is the $p$-adic family of Hecke eigensystems associated to the space of Yoshida lifts of $\hF$ and $\hG$. In other words, there \emph{exists
          a $p$-adic family of Hecke eigensystems associated to  $\hF$, $\hG$}.
    \item\label{thmcond:unique} $\eta$ is the unique minimal prime ideal contained in $\ker(\lambda_{f_0,g_0})$. In other words, there exists a \emph{unique} $\GSp_4$-Hida family passing through the Yoshida lift of $f_0$, $g_0$.
 %\item \label{thmcond:otheryoshida} The Galois representation attached to any $\GSp_4$-Hida family in the component $\TTT^{\rm st}$ is irreducible. 
  \end{enumerate}
\end{restatable}

As motivation for the rigidity result of Theorem \ref{thm:yoshidafamily}, one can consider the following question: 
\begin{itemize}
\item  Suppose that a $\Gl_2$ Hida family $F'$ passes through a classical CM form (resp. Eisenstein  series) with weight $k \geq 2$. Are all classical specializations of $F'$ CM forms (resp.~Eisenstein series)?
\end{itemize}  

This question has an affirmative answer due to Hida. There are two ingredients: (a) Hida \cite{MR848685} has proved uniqueness of Hida families passing through a classical specialization with classical weight $\geq 2$, and (b)  Hida \cite{MR1216135} has constructed $p$-adic families of Eisenstein series and cuspidal CM forms. On the one hand, ingredient (a) in the $\Gl_2$ setting relies both on the vertical control theorem of Hida  and duality between the space of modular forms and Hecke operators. On the other hand, ingredient (b) uses the fact that one can bound the tame level over a $p$-adic family of CM forms (resp.~Eisenstein series) and that one has an explicit description of their Fourier coefficients. Note that these questions need not have an affirmative answer, if one instead considers forms with weight $k=1$.  \\

Since in the $\Gl_2$-setting, the Hecke eigensystem determines the cuspform, one can ask whether there is a purely Galois theoretic approach (that is, without using explicit descriptions of Fourier coefficients) to the above question.
\begin{itemize}
\item  Suppose that the Galois representation associated to a $p$-distinguished $\Gl_2$ Hida family $F'$ at a classical specialization with weight $k \geq 2$ is induced from an imaginary quadratic field (respectively reducible). Are the Galois representations associated to all classical specializations of $F'$ induced from an imaginary quadratic field (respectively reducible)?
\end{itemize}  
 Indeed, the idea to study a purely Galois theoretic approach to the above question in the context of Eisenstein series can be traced back to Bella\"iche--Chenevier \cite{MR2225045}. The finiteness of class groups comes into play.  One can similarly answer this Galois theoretic question in the context of CM forms, when the associated adjoint representation is $p$-distinguished: one has to study an (appropriately twisted) anti-cyclotomic class group of an imaginary quadratic field. One of the novelties in our current work is to establish a similar Galois-theoretic approach in the $\GSp_4$-setting, without using explicit constructions of $p$-adic  families of Yoshida lifts. This is the main motivation for Theorem \ref{thm:yoshidafamily}. The analogous hypothesis in our result, labelled \ref{lab:tor}, is the finiteness of the Rankin--Selberg Selmer group.  \\
 
One may ask whether there is a proof simply for the \textit{existence} of the $p$-adic family of Hecke eigensystems associated to the Yoshida lifts of the two $\Gl_2$ Hida families $\hF$ and $\hG$ without assuming the finiteness of Selmer groups as in hypothesis \ref{lab:tor}.  In light of  the connection between global theta lifts and $L$-values via the Rallis inner product formula (see \cite{MR3858470}), it is expected that an explicit construction of $p$-adic families of the global theta lifts may play an important role in establishing the inequality $\Div({\CCC}) \geq  \Div(\theta^{\Sigma_0}_{\pmbtwo{4}{2}})$ of equation (\ref{eq:expected}). This is part of ongoing work of the first author and Liu. This is related to establishing the hypothesis \ref{lab:stab} --- introduced in Theorem \ref{thm:mainconj2} --- contingent on the $p$-divisibility of a special value of the normalized Rankin--Selberg complex $L$-function. Nevertheless without a direct automorphic construction of $p$-adic families of Yoshida lifts, one can construct the $p$-adic family of Hecke eigensystems associated to the space of Yoshida lifts of $\hF$ and $\hG$ if the following premise holds: one has to be able to bound the tame Siegel congruence level for the Yoshida lifts of $\Gl_2$ eigenforms $f_i$ and $g_j$, as one  varies over the modular forms $f_i$ and $g_j$ of the $\Gl_2$ Hida families $\hF$ and $\hG$, chosen to satisfy the hypotheses given at the beginning of Section \ref{sec:yoshidalifts}. Proposition \ref{prop:LRYoshidafam} lists two hypotheses, labelled \ref{lab:oddl} and \ref{lab:hypLR}, which allows such a bounded choice. For the hypothesis \ref{lab:oddl}, one requires the tame level to be odd. This hypothesis is related to work of Ganapathy \cite{Radhika15} on the preservation of \textit{depth} for  primes $\ell  \geq 3$ under the local Langlands correspondence for $\GSp_4$ along with the relationship between depth and fixed vectors for the Iwahori level. See also Remark \ref{rem:auxtamelevel}.  A bounded choice of tame level would be guaranteed  when the tame levels $N_{\hF}$ and $N_{\hG}$ are square free and the Atkin--Lehner eigenvalues coincide at all the primes dividing both $N_{\hF}$ and $N_{\hG}$, as in the works of the first author and Namikawa \cite{MR3623733} along with Saha--Schmidt \cite{MR3092267}. This is hypothesis \ref{lab:hypLR}.  In this case, one can let the $\GSp_4$ tame level $N$ equal $\mathrm{lcm}(N_{\hF},N_{\hG})$. While Saha--Schmidt analyse the local representations at ``bad primes''  following Brooks--Schmidt \cite{MR2344630}, the first author and Namikawa \cite{MR3623733} produce an explicit construction of Yoshida lifts using well-chosen test functions at primes $\ell$ dividing $N_{\hF}$ and $N_{\hG}$ under these assumptions. \\ 

Although vertical control theorems of Hida and Pilloni are available in the $\GSp_4$ setting for $\Lambda$-adic families of modular forms, the absence of duality between the space of modular forms and the Hecke algebra in the $\GSp_4$ setting (in contrast with the $\Gl_2$ setting as in Hida's work \cite{MR868300}) prevents us from using these control theorems for modular forms to directly establish \textit{uniqueness} of $p$-adic family of Hecke eigensystems.

\subsection{Congruence ideals and Selmer groups}

\begin{definition}\label{def:congruenceideal}
We have natural projection maps: 
\[\pi_{\YYY}:\bbT \rightarrow \bbT_{\YYY}, \qquad \pi_{\rm st}:\bbT \rightarrow \bbT_{\rm st}.\] We define $\CCC$ to be $\pi_{\YYY}(\ker(\pi_{\rm st}))$. By abuse of notation, we will also denote by  ${\CCC}$, the corresponding ideal generated inside the integral closure $\widetilde{\TTT}_{\YYY}$ of $\TTT_{\YYY}$. The ideal $\CCC$ is called the \textit{congruence ideal}.
%be the image under the $\bbT_{\YYY}$ The \textit{congruence module} $C(\bbT)$ is defined by 
%\[C(\bbT):=\frac{\TTT_{\YYY}\bigoplus \TTT^{\rm st}}{\bbT}.\]

%  The pushout in the category of $\TTT$-algebras of the following diagram is called 
%  \begin{align}\label{eq:GSp4congruencemoduledef}
 %   \xymatrix{\TTT  \ar[d] \ar[r] & \TTT_{\YYY} \ar[d]  \\
  %  \TTT^{\rm st} \ar[r]                                     & \frac{\TTT_{\YYY}\bigoplus \TTT^{\rm st}}{\bbT}
   % }
  %\end{align}
%As a $\TTT_{\YYY}$-module, the canonical map $\bbT_{\YYY}\to \bbT_{\YYY}\oplus \bbT^{\rm st}$ induces an isomorphism $\TTT_{\YYY}/\CCC\simeq C(\bbT)$ for some ideal $\CCC$ inside $\TTT_{\YYY}$. 

  %Despite the terminology ``module'', the congruence module is in fact a ring.
  %\begin{enumerate}
%    \item As a $\TTT^{\rm st}$-module, the congruence module is isomorphic to $\TTT^{\rm st}/\JJJ$.
 %   \item As a $\TTT_{\YYY}$-module, the congruence module is isomorphic to $\TTT_{\YYY}/\CCC$, for some ideal $\CCC$ inside $\TTT_{\YYY}$. By abuse of notation, we will also denote by  ${\CCC}$, the corresponding ideal generated inside the integral closure $\widetilde{\TTT}_{\YYY}$ of $\TTT_{\YYY}$. The ideal $\CCC$ is called the \textit{congruence ideal}.
 % \end{enumerate}

\end{definition}

\begin{restatable}{Theorem}{theoremtwo}\label{thm:mainconj2}
 Suppose that the hypotheses \ref{lab:irr}, \ref{lab:pdist-gsp4}, \ref{lab:multfree}, \ref{lab:yos} and \ref{lab:yos-neb} hold.  In addition, suppose that the following hypotheses hold:
 
 \begin{enumerate}[style=sameline,  align=left,label=\scshape{(Yos-Exist)}, ref=\scshape{Yos-Exist}]
    \item\label{lab:exist} The conclusion of Theorem \ref{thm:yoshidafamily} \ref{thmcond:yoshidafamily} holds.
  \end{enumerate}

  \begin{enumerate}[style=sameline,align=left,label=\scshape{(Stab)}, ref=\scshape{Stab}]
    \item\label{lab:stab} $\TTT^{\rm st} \neq 0$.
  \end{enumerate}
  Then we have the following 
  \begin{enumerate}
 \item \label{thmcond:otheryoshida} The Galois representation attached to any $\GSp_4$-Hida family in the component $\TTT^{\rm st}$ is irreducible.   
  
\item \label{part:divinequality} We have the following inequality of divisors of  $\II_{\hF,\hG}$:
  \begin{align} \label{eq:main1}
\Div\left(\Sel^{\Sigma_0}_{\rhopmb{4}{2}}(\Q)^\vee\right) \geq \Div\left({\CCC}\right).
  \end{align}
  \end{enumerate}
\end{restatable}

Under the additional hypotheses \ref{lab:tor} and \ref{lab:uni}, note that Theorem \ref{thm:yoshidafamily} would establish the hypothesis \ref{lab:exist} of Theorem  \ref{thm:mainconj2}. We prefer working with the hypothesis \ref{lab:exist} in Theorem  \ref{thm:mainconj2} since as stated earlier and illustrated in Proposition \ref{prop:LRYoshidafam} under the hypothesis \ref{lab:hypLR}, it is also possible to construct $p$-adic families of Hecke eigensystems of Yoshida lifts.\\

 The Euler system machinery  is expected to yield the following (opposite) inequality over $\II_{\hF,\hG}[\![\GGG_\Cyc]\!]$ for the $3$-variable main conjecture.   See works of Lei--Loeffler--Zerbes \cite{MR3224721}, Kings--Loeffler--Zerbes \cite{MR4060872,MR3637653} and  B\"{u}y\"{u}kboduk--Ochiai \cite{MR4603000} for various hypotheses under which such an inequality is available.
   \[\Div\left(\Sel^{\Sigma_0}_{\rhopmb{4}{3}}(\Q)^\vee\right) \stackrel{?}{\leq}  \Div\left(\theta^{\Sigma_0}_{\pmbtwo{4}{3}}\right).\]
 One can use the fact that the direction of the inequality afforded by the Euler system machinery is opposite to the one given by the method of congruences to deduce both Conjectures \ref{conj:mainconj2} and \ref{conj:mainconj3}. These ideas fall under the topic of specialization of Selmer groups and are dealt with in an axiomatic manner in the second author's thesis \cite{MR3919711}. A crucial ingredient in the current setup to these specialization techniques is a result of the second author and Lei \cite{Lei_2020}, where they establish that if the $\II_{\hF,\hG}[\![\GGG_\Cyc]\!]$-module $\Sel^{\Sigma_0}_{\rhopmb{4}{3}}(\Q)^\vee$ is torsion, then the divisor $\Div\left(\Sel^{\Sigma_0}_{\rhopmb{4}{3}}(\Q)^\vee\right)$ in $\II_{\hF,\hG}[\![\GGG_\Cyc]\!]$ is principal. 

\begin{restatable}{corollary}{corollarythirteen}\label{cor:mainconj3}
  In addition to the hypotheses of Theorem \ref{thm:mainconj2}, suppose also that the following hypotheses hold:
  \begin{enumerate}
  \item We have the following inequality of divisors in $\II_{\hF,\hG}[\![\Gamma_{\Cyc}]\!]$:
  \begin{align}\label{eq:ES}
    \Div\left(\Sel^{\Sigma_0}_{\rhopmb{4}{3}}(\Q)^\vee\right) \leq  \Div(\theta^{\Sigma_0}_{\pmbtwo{4}{3}}),
  \end{align}	

 \item We have the following inequality of divisors in $\II_{\hF,\hG}$:
  \begin{align}
    \Div({\CCC}) {\geq} \Div(\theta^{\Sigma_0}_{\pmbtwo{4}{2}}).
  \end{align}
    \end{enumerate}
Then, Conjectures \ref{conj:mainconj2} and \ref{conj:mainconj3} hold. That is, we obtain the following equality of divisors in  $\II_{\hF,\hG}[\![\Gamma_{\Cyc}]\!]$:
  \begin{align}
    \Div\left(\Sel^{\Sigma_0}_{\rhopmb{4}{3}}(\Q)^\vee\right) =  \Div\left(\theta^{\Sigma_0}_{\pmbtwo{4}{3}}\right),
  \end{align}
  along with the following equality of divisors in  $\II_{\hF,\hG}$:
  \begin{align}
    \Div\left(\Sel^{\Sigma_0}_{\rhopmb{4}{2}}(\Q)^\vee\right) = \Div\left({\CCC}\right) = \Div\left(\theta^{\Sigma_0}_{\pmbtwo{4}{2}}\right).
  \end{align}
\end{restatable}

\subsection{Pseudo-cyclicity of the  dual of the primitive Selmer group: Theorem \ref{thm:selmstructure}} \label{subsec:twopanchishkin}

As we mentioned earlier, the Panchishkin condition for the Galois representation $\rhopmb{4}{3}$ depends on the choice of a dominant Hida family. We will need to emphasize this point (that there is a dominant Hida family) only for Theorem \ref{thm:selmstructure}. Consequently, to state our result in Theorem \ref{thm:selmstructure},  we adopt a notation for Selmer groups that is slightly different from our earlier results. We let $\Sel_{\rhopmb{4}{3},\dominant{\hF}}(\Q)$ and $\Sel_{\rhopmb{4}{3},\dominant{\hG}}(\Q)$ denote the discrete \textit{primitive} Selmer group associated to $\rhopmb{4}{3}$ when $\hF$ and $\hG$ respectively are chosen to be dominant. Let $\vartheta_{\pmbtwo{4}{3},\dominant{\hF}}$ and $\vartheta_{\pmbtwo{4}{3},\dominant{\hG}}$ respectively denote the principal divisors in $\II_{\hF,\hG}[\![\GGG_\Cyc]\!]$ associated to the Pontryagin duals of the primitive Selmer groups. The fact that these divisors are principal is also shown in \cite{Lei_2020}.   We will also consider the following hypothesis:

 \begin{enumerate}[style=sameline,  align=left,label=(\scshape{GCD}), ref=\scshape{GCD}]
    \item\label{lab:gcd} The height of the ideal $(\vartheta_{\pmbtwo{4}{3},\dominant{\hF}}, \vartheta_{\pmbtwo{4}{3},\dominant{\hG}})$ in $\II_{\hF,\hG}[\![\GGG_\Cyc]\!]$  is at least two.
  \end{enumerate}

\begin{restatable}{Theorem}{theoremthree}\label{thm:selmstructure}
  In addition to the hypotheses of Theorem \ref{thm:mainconj2}, suppose that the hypothesis \ref{lab:gcd} holds. Then, the $\II_{\hF,\hG}[\![\GGG_\Cyc]\!]$-modules $\Sel_{\rhopmb{4}{3},\dominant{\hF}}(\Q)^\vee$  and  $\Sel_{\rhopmb{4}{3},\dominant{\hG}}(\Q)^\vee$ are pseudo-cyclic. That is, for every height one prime ideal $\mathfrak{p}$ in $\II_{\hF,\hG}[\![\GGG_\Cyc]\!]$, we have the following surjections of $(\II_{\hF,\hG}[\![\GGG_\Cyc]\!])_\mathfrak{p}$-modules:
  \begin{align*}
    (\II_{\hF,\hG}[\![\GGG_\Cyc]\!])_\p \twoheadrightarrow \left(\Sel_{\rhopmb{4}{3},\dominant{\hF}}(\Q)^\vee\right)_\p, \qquad  (\II_{\hF,\hG}[\![\GGG_\Cyc]\!])_\p \twoheadrightarrow \left(\Sel_{\rhopmb{4}{3},\dominant{\hG}}(\Q)^\vee\right)_\p.
  \end{align*}

\end{restatable}

\begin{remark}
Analogs of the hypothesis \ref{lab:gcd} have been recently considered in the topic of higher codimension Iwasawa theory \cite{MR4084165,MR3993809} and are related to pseudo-nullity conjectures, following Greenberg \cite[Conjecture 3.5]{MR1846466}. See the discussion in \cite[Section 1.3]{Lei_2020}.  We do not consider a similar hypothesis for the pair of non-primitive Selmer groups since they may have a common factor corresponding to the local Euler factors away from $p$.	
\end{remark}

\subsection*{Notations}
We shall use the following notation. Let $J=\pMX{0_2}{\bfone_2}{-\bfone_2}{0_2}\in \GL_4(\Z)$. Let $\GSp_4$ be the symplectic group of degree $4$. For any commutative ring $R$, \begin{align*}
\GSp_4(R) := \left\{g \in \Gl_4(R)\mid gJ g^t = \nu(g)J,\quad \nu(g)\in R^\times \right\}.
\end{align*}
The map $\nu: \GSp_4(R) \rightarrow R^\times$ is called the similitude character. Denote by $B$ the standard Borel subgroup of $\GSp_4$ consisting of matrices of the form
\[\begin{pmatrix}*&*&*&*\\
&*&*&*\\
&&*&\\
&&*&*
\end{pmatrix}\] 

For $(a,b,c)\in \Gm^3$, put
\[\dG{a}{b}{c}=\begin{pmatrix}a&&&\\
&b&&\\
&&a^{-1}c&\\
&&&b^{-1}c\end{pmatrix}\in \GSp_4.\]

Let $\AAA$ and $\AAA_f$ denote respectively the adeles and the finite adeles of $\Q$. 
Let $\cI_p$ be the Iwahori subgroup of $\GSp_4(\Z_p)$, consisting of matrices in $\GSp_4(\Z_p)$ whose reduction modulo $p$ belongs to the Borel subgroup $B(\Z/p\Z)$. \\ 

Define a morphism $h:\C^\times \to G(\R)$ by 
\[h(x+\sqrt{-1}y)=\pMX{x1_2}{y1_2}{-y1_2}{x1_2}.\]
Let $K_\infty$ denote the centralizer of $h$ in $\GSp_4(\R)$ and let $\frakg$ be the Lie algebra of $\GSp_4(\R)$.

\section{Hecke algebras for $\GSp_4$}\label{sec:Hecke}

\subsection{The universal Hecke algebra}
Let $A$ be a commutative ring. For each prime $\ell$ and an open-compact subgroup $K\subset \GSp_4(\Q_\ell)$, we let $\cH(K,A)$ be the ring (multiplication is via convolution), consisting of $K$ bi-invariant locally constant functions $\GSp_4(\Q_\ell) \rightarrow A$ with compact support. If $K$ equals $\GSp_4(\Z_\ell)$, it is well-known that the algebra $\cH(K,A)$ is commutative and is equal to the polynomial ring \[
A[T_\ell, R_\ell, S_\ell, S_\ell^{-1}],\]
where $T_\ell$, $R_\ell$ and $S_\ell$ are respectively the characteristic functions of the following double cosets:
\begin{align*}
\GSp_4(\Z_\ell)\dG{\ell}{\ell}{\ell}\GSp_4(\Z_\ell),  \quad  \GSp_4(\Z_\ell)\dG{\ell^2}{\ell}{\ell^2} \GSp_4(\Z_\ell),  \quad  \GSp_4(\Z_\ell)\dG{\ell}{\ell}{\ell^2}\GSp_4(\Z_\ell)
\end{align*}
(see \cite[Section 5.1.3]{MR4105535}). If $N$ is a positive integer, define the universal anemic Hecke algebra $\cH(N,A)$ of level $N$ by \[\cH(N,A)=\bigotimes_{\ell\nmid N}'\cH(\GSp_4(\Z_\ell),A).\]
Let  $A[U_\cP,U_\cQ]$ be the subalgebra of $\cH(\cI_p,A)$, where $U_\cP$ and $U_\cQ$ are the characteristic functions of the following double cosets:
\begin{align*}
\cI_p\dG{p}{p}{p}\cI_p,  \qquad  \cI_p\dG{p^2}{p}{p^2}\cI_p.
\end{align*}
Define the universal Hecke algebra $\cH(N,\cI_p,A)$ of the Iwahori level at $p$ by  
\[\cH(N,\cI_p,A):= A[U_\cP,U_\cQ]\bigotimes'_{\ell\nmid Np}\cH(\GSp_4(\Z_\ell),A)\]
It is isomorphic to the polynomial ring 
$A\left[\left\{U_\cP,U_\cQ\right\} \cup \left\{T_\ell,R_\ell,S_\ell,S_\ell^{-1}\right\}_{\ell \nmid Np}\right]$ over the countable set $\left\{U_\cP,U_\cQ\right\} \cup \left\{T_\ell,R_\ell,S_\ell,S_\ell^{-1}\right\}$, as $\ell$ variables over primes not dividing $Np$.

\subsection{Hida theory for Siegel modular forms} \label{sec:hidatheoryGSP4}
Let $T^1$ denote the diagonal torus of ${\rm Sp}_4$ defined for any commutative ring $R$ as follows: 
\[T^1(R)=\{\dG{t_1}{t_2}{1}\mid t_1,t_2\in R^\times\}.\]
Let $\OO$ denote the ring of integers in a finite extension of $\Q_p$. Consider the completed group $\cO$-algebra: \[\wtd\Lambda:=\cO\powerseries{T^1(\Z_p)}.\] For each \textit{integer weight} $\ulk=(k_1,k_2) \in \Z^2$, the character
\[
	T^1(\Z_p) \rightarrow \Z_p^\times,\quad {\rm diag}(t_1,t_2,t_1^{-1},t_2^{-1}) \mapsto  t_1^{k_1} t_2^{k_2}.
\]
determines a ring homomorphism of completed $\Z_p$-algebras:
\begin{align}\label{eq:ringhomk1k2}
	\varphi_\ulk: \wtd\Lambda \rightarrow \Z_p.
\end{align}
For each weight $\ulk=(k_1,k_2)$, let $P_\ulk$ denote the prime ideal $\ker\varphi_\ulk$ of $\wtd\Lambda$. 
Let $\frakX^{\rm cls}(\wtd\Lambda)$ and $\frakX^{\rm temp}(\wtd\Lambda)$  denote the set of classical points and tempered weights (with $k_2>3$):

\[\frakX^{\rm cls}(\wtd\Lambda):=\{P_\ulk\mid k_1\geq k_2>3\}, \qquad \frakX^{\rm temp}(\wtd\Lambda):=\{P_\ulk \mid k_1> k_2>3\}\]

For any open-compact subgroup $U\subset \GSp_4(\AAA_f)$, denote by $S_\ulk(U,\cO)$ the space of cuspidal Siegel modular forms of weight $\ulk$ and level $U$ with Fourier coefficients in $\cO$. Let $N$ be a positive integer prime to $p$. Consider the principal level subgroup $\Gamma(N)$ of $\GSp_4(\AAA_f)$ defined by 
\[\Gamma(N):=\{g\in \GSp_4(\wh \Z)\mid g\equiv I_4\pmod{N\wh\Z}\}.\]
Let $K$ be an open-compact subgroup of $\GSp_4(\AAA_f)$ containing $K(N)$. In particular, $K_\ell$ equals $\GSp_4(\Z_\ell)$ for all primes $\ell\nmid N$. Consider
\[K_0(p) \coloneqq \{g\in K\mid g_p\in \Iw_p\}.\]
 Then $S_\ulk(K_0(p),\cO)$ is a natural $\cH(N,\cI_p,\cO)$-module. 
Let $e_\ord:=\displaystyle \lim_{n\to\infty}(U_\cP U_\cQ)^{n!}$ denote the ordinary projector on $S_\ulk(K_0(p),\cO)$. Let $S^\ord_\ulk(K_0(p),\cO)$ denote $e_\ord \cdot  S_\ulk(K_0(p),\cO)$. \\ 

 We have the following \textit{vertical control} theorem proved by Hida \cite[Theorem 1.1]{MR1954939} and extended (to arbitrary standard parabolic subgroups) by Pilloni \cite[Theorem 7.1]{MR3059119}, which later allowed Pilloni to prove a more precise control theorem \cite[Th\'eor\`eme 1.2]{MR2783930}.  
\begin{theorem}[Control Theorem]\label{T:control}
There exists a $\cH(N,\cI_p,\wtd\Lambda)$-module $\bfS^{\ord}(K)$ such that 
\begin{enumerate}
\item $\bfS^{\ord}(K)$ is a free $\wtd\Lambda$-module with finite rank,
\item For every classical weight $P_\ulk\in\frakX^{\rm cls}(\wtd\Lambda)$, we have an $\cH(N,\cI_p,\cO)$-module isomorphism
\[\bfS^\ord(K) \otimes_{\wtd\Lambda} \wtd\Lambda/P_\ulk \iso S_\ulk^\ord(K_0(p),\cO).\]
\end{enumerate}
\end{theorem}
\begin{proof}

Let $\mathcal V_{\rm cusp}:=\mathcal{V}^{\mathrm{ord}-B,\star}_{\mathrm{cusp}}$ be the $\wtd\Lambda$-module introduced in the first paragraph of \cite[\S 7, p.379]{MR3059119}. By construction, $\mathcal V_{\rm cusp}$ is equipped with a natural Hecke action $\cH(N,\cI_p,\wtd\Lambda)\to \End_{\wtd\Lambda}(\mathcal V_{\rm cusp})$. Let
\[ \bfS^\ord(K):=\Hom_{\wtd\Lambda}(\mathcal{V}_{\mathrm{cusp}},\wtd\Lambda).\]
By \cite[Theorem 7.1(4,6)]{MR3059119}, $\bfS^\ord(K)$ is a free $\wtd\Lambda$-module and for any weight $\ulk$, 
\[\bfS^\ord(K)/P_\ulk \bfS^\ord(K_0(p))\iso \widehat{S}_\ulk^\ord(K_0(p),\cO),\]
where $\widehat{S}_\ulk^\ord(K_0(p),\cO)$ is the space of ordinary $p$-adic modular forms of weight $\ulk=(k_1,k_2)$ and level $K_0(p)$. By \cite[Th\'eor\`eme 1.2]{MR2783930} (see also  \cite[Proposition 5.4 and Th\'eor\`eme A.3]{MR3059119} for general results), ordinary $p$-adic modular forms of weight $(k_1,k_2)$ with $k_1\geq k_2$  are overconvergent and by \cite[Th\'eor\`eme and Remarque, p.977]{BPS16},  any overconvergent modular form of weight $(k_1,k_2)$ with $k_1\geq k_2>3$ is classical. We find that 
\[\widehat{S}_\ulk^\ord(K_0(p),\cO)=S_\ulk^\ord(K_0(p),\cO)\text{ if }k_1\geq k_2>3.\]
This completes the proof. 
\end{proof}

\begin{definition}Let ${\rm H}_{\ulk}^{\mathrm{ord}}(N)$ be the $\cO$-subalgebra of $\End_{\cO}(S^{\ord}_\ulk(K_0(p),\cO))$ generated by the image of $\cH(N,\cI_p,\cO)$.  We define the big cuspidal Hecke algebra $\hhh(N)$ to be the $\wtd\Lambda$-subalgebra of $\End_{\wtd\Lambda}(\bfS^\ord(K))$ generated by the image of $\cH(N,\cI_p,\wtd\Lambda)$. Since $\End_{\wtd\Lambda}(\bfS^\ord(K))$ is finitely generated as a $\wtd\Lambda$-module, $\hhh(N)$ is also finitely generated as a $\wtd\Lambda$-module. \end{definition}

\begin{proposition}\label{prop:kerspknilpotent}
For every $P_\ulk\in \frak{X}^{\rm cls}(\wtd\Lambda)$, the specialization map 
\[{\rm sp}_\ulk: \hhh(N)\otimes_{\wtd\Lambda} \wtd\Lambda/P_{\ulk} \onto{\rm H}_{\ulk}^{\mathrm{ord}}(K),\quad t\mapsto t|_{S^{\ord}_\ulk(K_0(p),\cO)} \]
is surjective and the kernel of ${\rm sp}_\ulk$ is contained in the radical of $\hhh(N)\otimes \wtd\Lambda/P_{\ulk}$.
\end{proposition}
\begin{proof} 
The surjectivity is clear. Let $t\in \hhh(N)$. Since $\bfS^\ord(K)$ is a free $\wtd\Lambda$-module of finite rank (say $n$), we can consider the characteristic polynomial $Q(X)\in \wtd\Lambda[X]$ of $t$. The reduction $Q(X)\pmod{P_{\ulk}}$ is the characteristic polynomial of its image under ${\rm sp}_{\ulk(t)}$. If $t$ belongs to $\ker{\rm sp}_\ulk$, then $Q(X)\pmod{P_{\ulk}}=X^n$. We can thus write $Q(X)=X^n+\beta(X)$, for some element $\beta(X) \in P_{\ulk}[X]$. It follows by using the Cayley-Hamilton theorem that $t^n+\beta(t)$, viewed as an element of $\hhh(N)$, equals $0$. Since $\beta(t)$ belongs to $P_{\ulk}\hhh(N)$, we can conclude that $t$ is nilpotent in $\hhh(N)/P_{\ulk}$.
\end{proof}

As an immediate corollary to Proposition \ref{prop:kerspknilpotent}, we have:
\begin{corollary}\label{cor:everyspecclassical}
Let $P_\ulk$ be an element of $\frakX^{\rm cls}(\wtd\Lambda)$. Then, every ring homomorphism $\hhh(N) \rightarrow \overline{\Q}_p$ whose kernel contains the prime ideal $P_\ulk$ of $\wtd\Lambda$ corresponds to a $p$-adic Hecke eigensystem with weight $\ulk$.
\end{corollary}

Let $\Delta$ be the torsion subgroup of $T^1(\Z_p)$ and let $\wh\Delta=\Hom(\Delta,\Z_p^\times)$ be the set of characters of $\Delta$. For each $\charcomp\in \wh\Delta$, define the idempotent 
\[e_{\charcomp}:=\frac{1}{\sharp\Delta}\sum_{\delta\in \Delta}\charcomp(\delta^{-1})[\delta]\in \wtd\Lambda.\]
Define
\[\bfS^\ord(K,\charcomp):=e_\charcomp\cdot\bfS^\ord(K);\quad \hhh(N,\charcomp)=e_\charcomp\cdot\hhh(N).\]

Let $\Gamma\subset T^1(\Z_p)$ be the subgroup defined by  
\[\Gamma:=\{\dG{t_1}{t_2}{1}\mid  t_1,t_2\in 1+p\Z_p\}\subset T^1(\Z_p).\]
Let $\Lambda$ denote the complete group algebra $\cO\powerseries{\Gamma}$ of $\Gamma$ over $\cO$. Then, $\Lambda$ turns out to be an Iwasawa algebra, being (non-canonically) isomorphic to the two-variable power series $\cO\powerseries{X_1,X_2}$. We also have the following decomposition:  
\begin{align}\label{eq:lambdadec}\wtd\Lambda=\prod_{\charcomp \in \wh\Delta}\Lambda e_{\charcomp}.
\end{align}
If the character component $\charcomp$ is clear from the context, we simply write $\Lambda$ instead of $\Lambda e_\charcomp$. Each maximal ideal $\frakm$ of $\hhh(N)$ determines a unique $\charcomp_{\frakm}\in \wh\Delta$ such that 
\[\hhh(N)_\frakm=e_{\charcomp_{\frakm}} \hhh(N)=\hhh(N,\charcomp_{\frakm}),\]
and there exists $(a,b)\in (\Z/(p-1)\Z)^2$ such that $\charcomp_{\frakm}(\dG{t_1}{t_2}{1})=t_1^at_2^b$. Put
\[\frakX^{\rm cls}_{\charcomp}(\Lambda):=\{P_\ulk\in\frakX^{\rm cls}\mid k_1\equiv a,\,k_2\equiv b\pmod{p -1}\}, \qquad \frakX^{\rm temp}_{\charcomp}:=\frakX^{\rm cls}_{\charcomp} \cap \frakX^{\rm temp}.\]
 Using Theorem \ref{T:control}, one can conclude that $\bfS^\ord(K)_{\frakm}$ is free of finite rank over $\Lambda$ and  
\[\bfS^\ord(K)_{\frakm}\otimes \Lambda/P_{\ulk}\iso \bfS_\ulk^\ord(K,\cO),\quad \ulk\in \frakX^{\rm cls}_{\frakm}.\]

\begin{proposition}\label{P:reduced} The Hecke algebra $\hhh(N)_{\frak{m}}$ is reduced local $\Lambda$-algebra. 
\end{proposition}
\begin{proof}Let $t\in \hhh(N)_{\frakm}\subset \End_{\Lambda}(\bfS^\ord(K))\iso \Mat_n(\Lambda)$ be a nilpotent element with $t^m=0$. We have the semisimplicity of the ordinary Hecke algebra ${\rm H}_{\ulk}^{\mathrm{ord}}(N)$ for a fixed weight $\ulk\in\frakX_\frakm^{\rm temp}$. See \cite[Corollary 5.4]{MR2055355}. It thus follows that the image of $t$ in $ \Mat_n(\Lambda/P_{\ulk})$ equals zero for all tempered weights $\ulk\in\frakX_\frakm^{\rm temp}$. It follows that $t=0$ by the Zariski density of $\frakX_{\frakm}^{\rm temp}$. See Lemma \ref{lem:density}.
\end{proof}

\section{Automorphic representations on $\GSp_4$ and associated Galois representations} \label{sec:Galoisreps}
\subsection{Automorphic representations}
Let $\varPi$ be an irreducible cuspidal automorphic representation of $\GSp_4(\AAA)$ with the central character $\omega_\varPi$. We decompose $\varPi$ into the restricted tensor product of a $(\frakg,K_\infty)$-module $\varPi_\infty$ and irreducible admissible representations of $\GSp_4(\Q_\ell)$:
\begin{align*}
\varPi \cong \varPi_\infty \otimes \left(\bigotimes\limits_{\ell\text{ finite}} \varPi_\ell\right)
\end{align*}
For any place $\ell$ of $\Q$, the representation $\varPi_\ell$ is called the local component of $\varPi$ at $\ell$. Let $K \coloneqq \prod_\ell K_\ell\subset \GSp(\AAA_f)$ be an open-compact subgroup such that $K_\ell=\GSp_4(\Z_\ell)$ for all primes $\ell\nmid N$. Suppose that the finite dimensional $\C$-vector space $\varPi^K\neq 0$.  Then
$\varPi^K$ is a $\cH(N,\C)$-module. In particular, we obtain a ring homomorphism:
\begin{equation}\label{eq:localhecke}
\phi_{\varPi}:\cH(N,\Z) \rightarrow \mathbb{C}
\end{equation}
such that $Tv=\phi_{\varPi}(T)v$ for every element $T$ in $\cH(N,\Z)$ and every vector $v$ in $\varPi^K$. This homomorphism $\phi_{\varPi}$ is called the Hecke eigensystem associated with $\varPi$. 

Let $(k_1,k_2)$ be integers with \beq\label{E:regular}k_1 \geq k_2 \geq 3.\eeq
We say that the archimedean component $\varPi_\infty$ is a holomorphic discrete series of weight $(k_1,k_2)$ if $\varPi_\infty\otimes\norm{\cdot}^\frac{k_1+k_2-3}{2} $ is the unitary holomorphic discrete series representation of weight $(k_1,k_2)$, i.e. $\varPi_\infty$ has Harish-Chandra parameter $(k_1-1,k_2-2)$.
Then $\varPi_\infty$ is cohomological and $\omega_\varPi \norm{\cdot}^{k_1+k_2-3} $ is unitary. It is well-known (using results of Chai--Faltings \cite{MR1083353})
that there exists a number field $E_\varPi$ such that $\phi_{\varPi}(\cH(N,\Z))$ lies in its ring of integers $\cO_{E_\varPi}$. 
For each $\ell\nmid N$, the abstract Hecke polynomial $Q_\ell(X)\in \cH(N,\Z)[X]$ at $\ell$ is defined as follows: 
\begin{align} \label{eq:abstractdegree4poly}
Q_\ell(X):=X^4 - T_\ell X^3  + \ell(R_\ell+(\ell^2+1)S_\ell)X^2 - \ell^3T_\ell S_\ell X + \ell^6S_\ell^2.
\end{align}
Let $Q_{\varPi_\ell}(X):=\phi_{\varPi}(Q_\ell(X))\in \cO_{E_\varPi}[X]$.

\begin{theorem}\label{thm:GSP4weissauertaylorlaumon}[Laumon  \cite{MR2234859}, Taylor \cite{MR1240640},  Weissauer \cite{MR2234860}]\mbox{}
Let $\varPi$ be a cuspidal cohomological automorphic representation of $\GSp_4(\AAA)$ with $\varPi_\infty$ holomorphic discrete series of weight $(k_1,k_2)$ and $k_1>k_2\geq 3$.
\begin{enumerate}
\item\label{thm:galoispoint2} Let $\eta$ be a finite place of $\Q$. There exists a unique continuous semi-simple Galois representation $\varrho_{\varPi,\eta} :G_\Q \rightarrow \Gl_4(E_{\varPi,\eta})$ such that for every finite place $\ell \notin \rm Ram(\varPi) \cup \{\eta\}$, we have the following equality of characteristic polynomials:
\begin{align*}
\det\left(X \cdot \bfone_4-\varrho_{\varPi,\eta}(\Frob_\ell)\right) = Q_{\varPi_\ell}(X).
\end{align*}
Here, $E_{\varPi,\eta}$ denotes the compositum of the fields $E_\varPi$ and $\Q_\eta$ inside $\overline{\Q}_\eta$. In particular, the Galois representation $\varrho_{\varPi,\eta}$ is unramified at all primes $\ell \notin \rm Ram(\varPi) \cup \{\eta\}$. 
\item\label{thm:galoispoint3} The Galois representation $\varrho_{\varPi,\eta}$ is Hodge--Tate  with weights (we normalize the $\eta$-adic cylotomic character to have Hodge--Tate weight $1$)
\begin{align*}
0, \quad k_2 -2, \quad k_1-1, \quad k_1+k_2-3.
\end{align*}
\end{enumerate}
\end{theorem}
\begin{proof} 
See \cite[Theorem 10.1.3]{LSZ22JEMS}.\end{proof}

Following the notation in \cite[Definition 2.5.2]{Gee19}, the Arthur parameter $\psi_\varPi$ can be written as a formal sum $\boxplus_{i=1}^r \pi_i[d_i]$, where $\pi_i$ are $\omega_\varPi$-self dual automorphic representations for $\GL_{N_i}$ and $d_i$ are positive integers such that $\sum_{i =1}^rN_id_i=4$. Based on Arthur's work, there are six types of Arthur parameters indexed by (a)-(f) in \cite[Remark 6.1.8]{Gee19}. We say $\varPi$ is stable if $\psi_\varPi=\pi$, where $\pi$ is a cuspidal automorphic representation of $\GL_4(\AAA)$, and $\varPi$ is of Yoshida type if $\psi_\varPi=\pi_1\boxplus \pi_2$, where $\pi_1$ and $\pi_2$ are cuspidal automorphic representations of $\GL_2(\AAA)$ with the same central character and $\pi_1\not\iso\pi_2$. 

\begin{proposition}\label{P:32}
Let $\varPi$ be a cohomological, cuspidal automorphic representation of weight $(k_1,k_2)$. Then \begin{enumerate}
\item $\varPi$ is either stable or of Yoshida type if $\varPi$ is tempered.
\item  $\varPi$ is \textit{tempered} if $k_1 > k_2 \geq  3$.
\end{enumerate}
\end{proposition}
\begin{proof} Let $\psi_\varPi=\boxplus_i \pi_i[d_i]$ be the Arthur parameter of $\varPi$. The Ramanujan conjecture for Siegel modular forms is known due to work of Weissauer \cite{MR2498783}. As a result, $\varPi$ is tempered if and only if  $d_i=1$ for all $i$. In view of the classification \cite[Remark 6.18]{Gee19}, we find that $\varPi$ is tempered if $\varPi$ is stable or of Yoshida type (types (a) and (b)) and is non-tempered for types (c)-(f). For (2), the inspection on the archimedean $L$-parameter rules out the possibility for types (c)-(f).
\end{proof}

\subsection{Yoshida lifts}\label{sec:yoshidalifts}
Cuspidal automorphic representations of $\GSp_4(\AAA)$ of \textit{Yoshida type} (as in Arthur's classification \cite{MR2058604}) are generated by so-called Yoshida lifts. 
The original constructions of Yoshida lifts by Yoshida and the generalization by  B\"{o}cherer--Schulze-Pillot involve producing explicit automorphic forms utilizing the theory of global theta correspondences. Let $S_\kappa^{\rm new}(N,\psi)$ be the set of cuspidal elliptic newforms of weight $\kappa$, level $\Gamma_1(N)$ and nebentypus $\psi$. 
For a newform $f=\sum_{n>0} a_n(f)q^n\in S_\kappa^{\rm new}(N,\psi)$, we let $\pi_f$ denote the associated cuspidal automorphic representation of $\Gl_2(\Q)$. Then the central character of $\pi_f$ is given by $\psi^{-1}\norm{\cdot}^\frac{1-\kappa}{2}$, whereas on the other hand, $\pi_{f,\infty}\norm{\cdot}^\frac{\kappa-1}{2}$ is a unitary discrete series representation of weight $\kappa$. Let 
\[ \varrho_{f,p}:G_\Q \rightarrow \Gl_2(\overline{\Q}_p)\] be the $p$-adic Galois representation associated to $f$ constructed by Shimura and Deligne. The Galois representation $\rho_{f,p}$ is unramified outside $pN$ and  for any prime $\ell \nmid N_fN_gp$, we have
\begin{align*}
\mathrm{Trace}(\rho_{f,p}(\mathrm{Frob}_\ell)) = a_\ell(f).\end{align*}

Now we let \[(f_1,f_2)\in S_{k_{f_1}}^{\rm new}(N_1,\psi_1)\times S_{k_{f_2}}^{\rm new}(N_2,\psi_2)\] be a pair of elliptic newforms with $k_{f_1}\geq k_{f_2}\geq 2$. Decompose $\pi_{f_1}$ and $\pi_{f_2}$ into the restricted tensor products
\begin{align*}
  \pi_{f_1} \cong   \pi_{f_1,\infty} \bigotimes \pi_{f_1,\ell}, \quad \pi_{f_2} \cong \pi_{f_2,\infty} \bigotimes \pi_{f_2,\ell}.
\end{align*}
Consider the following set:
\[S_{\mathrm{disc},f_1,f_2}=\left\{\text{ finite prime }\ell \mid  \pi_{f_1,\ell}, \ \pi_{f_2,\ell} \text{ are discrete series}\right\}.\]

We consider the following hypotheses on $f_1$, $f_2$:

\begin{hypothesis} \label{hyp:yoshida1} $f_1\neq f_2$. 
\end{hypothesis}

\begin{hypothesis}\label{hyp:yoshida2}    $k_{f_1}\equiv k_{f_2}\pmod{2}$.
\end{hypothesis}

\begin{hypothesis} \label{hyp:yoshida3} The set $S_{\mathrm{disc},f_1,f_2}$ is non-empty.
\end{hypothesis}

\begin{hypothesis} \label{hyp:yoshida4}
$\psi_1\psi_2^{-1}=\chi^2$ for some Dirichlet character $\chi$. 
\end{hypothesis}

By Arthur's multiplicity formula as established in \cite{Gee19} (or utilising the representation theoretic approach of Roberts \cite[Theorem 8.6
(2)]{MR1871665}), there exists a cuspidal automorphic representation $\varPi$ of Yoshida type such that \[\psi_\varPi=\pi_{f_1}\boxplus \pi_{f_2}\otimes\left(\chi\norm{\cdot}^\frac{k_{f_2}-k_{f_1}}{2}\right),\]
and the archimedean component $\varPi_\infty$ is a discrete series of weight $(\frac{k_{f_1}+k_{f_2}}{2},\frac{k_{f_1}-k_{f_2}}{2}+2)$. Moreover, every cuspidal automorphic representation $\varPi$ of Yoshida type with $\varPi_\infty$ a discrete series of weight $(k_1,k_2)$ arises in this way. The $p$-adic Galois representation associated to $\varPi$ in Theorem  \ref{thm:GSP4weissauertaylorlaumon} is given by \[\varrho_{\varPi,p} \cong  \varrho_{f_1,p}\oplus   \varrho_{f_2,p} \left(\chi\varepsilon^{\frac{k_{f_1}-k_{f_2}}{2}} \right),\]
where $\varepsilon$ is the $p$-adic cyclotomic character. \\

For simplicity, we assume that $\chi$ is trivial from now onwards. 

\begin{remark}\label{rem:auxtamelevel}
 Since the principal congruence subgroups form a basis of neighborhood of the identity element in $\mathrm{Sp}_4(\Z_\ell)$, to obtain vectors fixed by the principal congruence subgroup $\Gamma^{(2)}(N)$, one has  to choose $N$ divisible by high enough powers of the primes $\ell$ dividing $N_fN_g$. 
 
 To have some optimal control on the tame level (away from the prime $l=2$), one can use Ganapathy's result \cite[Proposition 9.2]{MR3432266} that the local Langlands correspondence for $\GSp_4$  proved by Gan--Takeda \cite{MR2800725} preserves \textit{depth} for primes $l \geq 3$. The local Langlands correspondence for $\Gl_2$ also preserves depth (\cite{MR3579297}). It will allow one to work with the following tame level (if $2$ is a ramified prime, one needs to choose a large enough $s$): 
\begin{align}\label{levelauxdepth}
    2^{s} \times \prod_{\substack{l \mid N_\hF N_\hG \\ l \neq 2}} l^{\max\{2,\val_p(N_\hF)+1, \val_p(N_\hG)+1\}}, \qquad s \gg 0
\end{align}
 For the relationship between the depth and fixed vectors for the Iwahori level, see \cite[Lemma 7.2]{MR3432266}.  This representation theoretic approach introduces an extra factor of $+1$ in the exponents of primes in the tame level. 
 \end{remark}

\subsection{$p$-ordinary automorphic representations}
Let $\varPi$ be a cuspidal automorphic representation of $\GSp_4(\AAA)$ with $\varPi_\infty$ a discrete series of weight $(k_1,k_2)$ and $k_1\geq k_2\geq 3$. Suppose that $p\nmid N$. Having fixed an embedding $E_\varPi \hookrightarrow \overline{\Q} \hookrightarrow \overline{\Q}_p$, we let $\{\al_0,\al_1,\al_2,\al_3\}$ be the roots of $Q_{\varPi_p}(X)$ such that 
\[\val_p(\al_0)\leq \val_p(\al_1)\leq \val_p(\al_2)\leq \val_p(\al_3).\]
We say $\varPi$ is $p$-ordinary if 
\begin{align} \label{eq:rootsvaluation}
\val_{p}(\alpha_0) = 0, \qquad \val_p(\alpha_1) = k_2-2, \qquad \val_p(\alpha_2) =k_1-1,  \qquad \val_p(\alpha_3) = k_1+k_2-3.
\end{align}

\begin{proposition}\label{P:33}Suppose that $\varPi$ is of Yoshida type with $\psi_\varPi=\pi_1\boxplus\pi_2$ and $\varPi_p$ is spherical. Then $\varPi$ is $p$-ordinary if and only if $\pi_1$ and $\pi_2$ are $p$-ordinary.
\end{proposition}
\begin{proof}Let $c=\frac{k_1+k_2-3}{2}$. Since $\varPi_\infty$ is a discrete series, $\pi_{1,\infty}\norm{\cdot}^c$ and $\pi_{2,\infty}\norm{\cdot}^c$ are unitary discrete series of weights $a_1$ and $a_2$ respectively. From the comparison between the archimedean $L$-parameters of $\varPi$ and $\pi_1\boxplus\pi_2$, we find that
\[\{k_1+k_2-3, k_1-1,k_2-2,0\}=\{a_1+c,a_2+c,-a_2+c,-a_1+c\}.\]
On the other hand, Since $p\nmid N$, $\varPi_p$ is spherical, and hence the representations $\pi_1$, $\pi_2$ are unramified at $p$. Let $(b_1+c,-b_1+c)$ and $(b_2+c, -b_2+c)$ with $0\leq b_i\leq c$ be the $p$-adic valuations of the Satake parameters at $p$ for $\pi_1$ and $\pi_2$ respectively. 
It is clear that if $\pi_1$ and $\pi_2$ are $p$-ordinary, then $a_i=b_i$ and $\varPi$ is $p$-ordinary. 
Conversely, suppose that $\varPi$ is $p$-ordinary. We may assume $a_1\geq a_2> 0$. Since $k_1>k_2>3$, we obtain 
\[a_1=c=\frac{k_1+k_2-3}{2};\quad a_2=\frac{k_1-k_2+1}{2}.\]
The $p$-ordinary assumption tells us 
\[\{k_1+k_2-3, k_1-1,k_2-2,0\}=\{b_1+c,b_2+c,-b_2+c,-b_1+c\}.\]
Since $-b_1+c\leq -a_1+c$ and $-b_2+c\leq -a_2+c$, we conclude that $a_1=b_1$ and $a_2=b_2$. This shows that $\pi_1$ and $\pi_2$ are also $p$-ordinary.  
\end{proof}

\begin{theorem}[Urban] \label{thm:urban1}
Suppose that $\varPi$ is $p$-ordinary. Then there exists a matrix $C$ in $\Gl_4(\overline{\Q}_p)$ and four $\overline{\Q}_p$-valued characters $\psi_0, \ \psi_1, \ \psi_2,\ \psi_3$ unramified at $p$, such that we have the following equality of matrices in $\Gl_4(\overline{\Q}_p)$:
\begin{align} \label{eq:thmgaloisordinary}
 \varrho_{\varPi,p}(\sigma) = C \left[\begin{array}{cccc} \psi_3 \cyc^{k_1+k_2-3}(\sigma) & \star & \star & \star \\ 0 & \psi_2 \cyc^{k_1-1}(\sigma) & \star & \star \\ 0 & 0 & \psi_1 \cyc^{k_2-2}(\sigma) & \star \\ 0 & 0 & 0 & \psi_0(\sigma) \end{array} \right] C^{-1}, \qquad \forall \sigma \in \Gal{\overline{\Q}_p}{\Q_p}.
\end{align}
Following the notations in equation (\ref{eq:rootsvaluation}), the four unramified characters $\psi_0, \ \psi_1, \ \psi_2,\ \psi_3$ are determined by the following property:
\begin{align}\label{eq:eigenvaluation}
\psi_0(\Frob_p) = \alpha_0, \quad \psi_1(\Frob_p) = \dfrac{\alpha_1}{p^{k_2-2}}, \quad \psi_2(\Frob_p) = \dfrac{\alpha_2}{p^{k_1-1}}, \quad \psi_3(\Frob_p) = \dfrac{\alpha_3}{p^{k_1+k_2-3}}.
\end{align}
\end{theorem}
\begin{proof}
By Proposition \ref{P:32}, $\varPi$ is either stable or of Yoshida type. If $\varPi$ is stable, the theorem follows from \cite[Theorem 1, Corollary 1(iii)]{MR2234861}. Suppose that $\varPi$ is of Yoshida type with $\psi_\varPi=\pi_1\boxplus\pi_2$. We have $\rho_{\varPi,p}=\rho_{\pi_1,p}\oplus\rho_{\pi_2,p}$, where $\rho_{\pi_i,p}$ are the $p$-adic Galois representations associated to $\pi_i$ for $i=1,2$. By Proposition \ref{P:33}, $\pi_1$ and $\pi_2$ are $p$-ordinary, so the associated Galois representations $\rho_{\pi,p}$ are also ordinary at $p$ by results of Deligne and Mazur--Wiles (see \cite[Theorem 4.2.7]{MR2894984}). Namely $\rho_{\pi_i,p}$ restricted to $\Gal{\overline{\Q}_p}{\Q_p}$ must be upper-triangular for $i=1,2$. Up to a choice of basis over $\overline{\Q}_p$, the Galois representation associated to $\varPi$ must then have the following shape:
\[\left[\begin{array}{cccc} \psi_3 \cyc^{k_1+k_2-3} & \star & 0 & 0 \\ 0 & \psi_0  & 0 & 0 \\ 0 & 0 & \psi_2 \cyc^{k_1-1} & \star \\ 0 & 0 & 0 & \psi_1 \cyc^{k_2-2} \end{array} \right].\]
The theorem follows immediately.
\end{proof}

\begin{remark}
The contragradient $(\varrho_{\varPi,\eta}^{-1})^{T}$ of the Galois representation $\varrho_{\varPi,\eta}$ is the Galois  representation appearing in Urban's work \cite{MR2234861}, where the characteristic polynomial of $\Frob_\nu^{-1}$ (the geometric Frobenius) is equated to the Hecke polynomial $Q_{\varPi_\nu}(X)$. In our statements, we have used the arithmetic Frobenius. 
\end{remark}

\begin{proposition}\label{P:new}Let $\ulk=(k_1,k_2)$ with $k_1>k_2>3$. Any cuspidal automorphic representation generated by elements in $S_\ulk^\ord(K_0(p),\cO)$ must be spherical at $p$.\end{proposition}
\begin{proof}
We have a decomposition of $\cH(N,\cI_p,\C)$-modules
\beq\label{E:decomposition}S_\ulk(K_0(p),\C)=\bigoplus_{\varPi}m_{\varPi}\varPi_f^{K_0(p)},\eeq
where $\varPi$ runs over cuspidal automorphic representations of $\GSp_4(\AAA)$ with cohomological weight $\ulk$ and $\varPi_f^{K_0(p)}\neq0$. We claim that if $\varPi$ is ordinary, then $\varPi_p$ must be spherical. To see the claim, we note that according to the classification of Iwahori-spherical representations at non-archimedean places in \cite[Table A13, page 293]{MR2344630}, and by the ordinary hypothesis, the $p$-adic valuations of the Satake parameters of $\varPi_p$ are given by 
\[a_1=0\leq a_2=k_2-2 \leq a_3=k_1-1 \leq a_4=k_1+k_2-3.\]
We find that $a_{i+1}-a_i>1$ for all $i=1,2,3$.
It follows from the $L$-parameters of representations of Type II-VI in \cite[page 53-55]{MR2344630} that $\varPi_p$ must be of Type I, and hence $\varPi_p$ is spherical.
\end{proof}
\subsection{Hecke eigensystems associated with $p$-ordinary representations}
We retain the notation in the previous subsection. Let $\varPi$ be a $p$-ordinary automorphic representations with $\varPi^{K}\neq 0$. By \cite[Proposition 7.4 (2)]{MR4771233}, there exists a non-zero eigenvector $v_\ord\in \varPi^{K_0(p)}$ of $U_\cP$ and $U_\cQ$ such that 
\begin{align*}
U_\cP v_\ord=\alpha_0v_\ord; \quad U_\cQ v_\ord=p^{2-k_2}\alpha_1\alpha_0 v_\ord.
\end{align*}
Moreover, by \cite[Proposition 7.4 (3)]{MR4771233}, this ordinary vector $v_\ord$ is obtained by taking a suitable $p$-stabilization of a spherical vector $v^0$, namely $v_\ord= e^0_\ord v^0$ for the explicit ordinary projector $e^0_\ord$ in \cite[Definition 7.2]{MR4771233}. 

This allows us to define the Hecke eigensystem $\phi_\varPi^\ord: \cH(N,\cI_p,\Z) \to \overline{\Z}$ by
\begin{align*}
\notag \phi_{\varPi}^\ord(T) &=\phi_{\varPi}(T), \qquad \forall \ T\in \cH(N,\cI_p,\Z),\\
 \phi_{\varPi}^\ord(U_\cP)&=\alpha_0,\quad \phi_{\varPi}^\ord( U_\cQ)=p^{2-k_2}\alpha_1\alpha_0. 
\end{align*}
Suppose that $\varPi$ is a Yoshida lift with $\psi_\varPi=\pi_{f_1}\boxplus \pi_{f_2}\otimes\norm{\cdot}^\frac{k_{f_2}-k_{f_1}}{2}$ for a pair of elliptic newforms $(f_1,f_2)$. For $i=1,2$, let $a_p(f_i^\dagger)$ be the $p$-adic unit root of the Hecke polynomial $X^2-a_p(f_i)X+\psi_i(p)p^{k_{f_i}-1}$. Then we have
\begin{equation}\label{E:HeckeYoshida}
\begin{aligned}\phi_{\varPi}^\ord(Q_\ell(X))&=(X^2-a_\ell(f_1)X+\psi_1(\ell)\ell^{k_{f_1}-1})(X^2-a_\ell(f_2)\ell^\frac{k_{f_1}-k_{f_2}}{2}\chi(\ell)X+\psi_1(\ell)\ell^{k_{f_1}-1}),\\
\phi_{\varPi}^\ord(U_\cP)&=a_p(f_1^\dagger);\quad \phi^\ord_{\varPi}(U_\cQ)=a_p(f_1^\dagger)a_p(f_2^\dagger).
\end{aligned}
\end{equation}

\section{Pseudocharacters and Generalized Matrix Algebras}\label{sec:pseudocharactersGMA}

Let $S$ denote the set of primes dividing $Np$ and containing the infinite primes. Let $\Q_S$ denote the maximal extension of $\Q$ unramified outside $S$. Let $G_S$ denote $\Gal{\Q_S}{\Q}$. Let $\RRR$ be a finitely generated $\Lambda$-algebra. Let $n$ be a positive integer such that $n!$ is invertible in $\RRR$. We will consider a \textit{pseudocharacter of dimension $n$} (following Taylor \cite{MR1115109} and Bella\"iche--Chenevier \cite{MR2656025}), which is a  function $\mathrm{T}: G_S \rightarrow \RRR$, satisfying certain properties (see \cite{MR2656025}).
 Furthermore, we will say that the pseudocharacter is \textit{continuous} if the map $\pcT:G_S \rightarrow \RRR$ is continuous. The map $\pcT$ extends to an $\RRR$-linear map $\RRR[G_S] \rightarrow \RRR$ which is also a \textit{pseudocharacter} in the language of \cite[Definition 1.2.1]{MR2656025} and which by abuse of notation will also be denoted $\pcT$. We refer the reader to \cite{MR2656025} for the precise properties that $\pcT$ satisfies. Following \cite{MR2656025}, we call the following two-sided ideal of $\RRR[G_S]$ the kernel of $\pcT$. 
\begin{align}\label{def:kerneltrace}
	\ker(\pcT) \coloneqq \left\{x \in \RRR[G_S], \text{ such that } \pcT(xy)=0, \text{ for all $y$ in $\RRR[G_S]$} \right\}.  
\end{align}

Let $\RRR$ be a reduced ring that is  a finite integral extension of $\Lambda$. We set 
\begin{align*}\frakX^{\rm cls}_{\charcomp}(\RRR)&:=\{\p \in \Spec(\RRR), \text{ such that } \p \cap \Lambda = P_{\ulk} \text{ with }\ulk=(k_1,k_2)\in\frakX^{\rm cls}_\charcomp\}, \\ \frakX^{\rm temp}_{\charcomp}(\RRR) &:=\{\p \in \Spec(\RRR), \text{ such that } \p \cap \Lambda = P_{\ulk} \text{ with }\ulk=(k_1,k_2)\in\frakX^{\rm temp}_\charcomp\}.
\end{align*}

Let $\bbT=\hhh(N)_\m=\hhh(N,\charcomp)_\m$ denote a local component of the ordinary Hida Hecke algebra introduced in Section \ref{sec:Hecke}. Let $S$ denote the set of primes in $\Q$ containing the infinite primes, $p$  and the primes dividing the tame level $N$. As an immediate corollary to Proposition \ref{P:reduced} and Lemma \ref{lem:density}, we have 
\begin{corollary}\label{cor:densityhomega} The sets $\frakX^{\rm temp}_{\charcomp}\left(\bbT\right)$ and $\frakX^{\rm cls}_{\charcomp}\left(\bbT\right)$ are dense in  $\Spec(\hhh(N,\charcomp)_{\frak{m}})$.
\end{corollary}

Corollaries \ref{cor:everyspecclassical} and \ref{cor:densityhomega} along with Theorem \ref{thm:GSP4weissauertaylorlaumon} allow us to deduce that there exists a unique continuous pseudocharacter $\pcT: G_S \rightarrow \bbT$ such that for all primes $\ell$ not in $S$, we have $\pcT(\Frob_\ell) = T_\ell$. See  Theorem \ref{thm:pseudochar} for the proof. \\

Let $\gamma_0$ denote the topological generator of $\Gal{\Q_\Cyc}{\Q}$ that maps to $1+p$ under the isomorphism provided by the cylotomic character $\Gal{\Q_\Cyc}{\Q} \cong 1+p\Z_p$. Consider  the tautological characters $\kappa_1$ and $\kappa_2$ of the global Galois group $G_S$ obtained by identifying the topological generator $\gamma_0$ of $\Gal{\Q_\Cyc}{\Q}$ with $1+X_1$ and $1+X_2$ respectively. 
\begin{align*}
	\kappa_1: G_S \twoheadrightarrow \Gal{\Q_\Cyc}{\Q} \hookrightarrow \Z_p[\![\Gal{\Q_\Cyc}{\Q}]\!]^\times &\xrightarrow{\cong} \Z_p[\![X_1]\!]^\times \hookrightarrow \Lambda^\times \hookrightarrow \TTT^\times,  \\	\kappa_2: G_S \twoheadrightarrow \Gal{\Q_\Cyc}{\Q} \hookrightarrow  \Z_p[\![\Gal{\Q_\Cyc}{\Q}]\!]^\times &\xrightarrow{\cong} \Z_p[\![X_2]\!]^\times  \hookrightarrow \Lambda^\times \hookrightarrow \TTT^\times.
\end{align*}

For each classical prime $\p$ in $\frakX^{\rm temp}_{\charcomp}\left(\TTT\right)$,  the similitude character $\lambda_\p:G_S\rightarrow \overline{\Z}_p^\times$ is given by the composition of the similitude character of $\GSp_4$ with the Galois representation (as given in Theorem \ref{thm:GSP4weissauertaylorlaumon}) associated to the Siegel cuspidal automorphic representation for the specialization corresponding to $\p$. 

\begin{lemma}\label{lem:similitudechar}
There exists an unramified-at-$p$ global character $\widetilde{\lambda}_{\mathrm{ur}}: G_S \rightarrow \mathbb{F}_q^\times \hookrightarrow \OO \hookrightarrow \TTT^\times$ such that the specialization  of the  global character $\widetilde{\lambda} \coloneqq \widetilde{\lambda}_{\mathrm{ur}}  \cyc^{-3} \omega^a \omega^b \kappa_1 \kappa_2	:G_S \rightarrow \TTT^\times$ at each classical prime $\p$ in $\frakX^{\rm temp}_{\charcomp}\left(\TTT\right)$ equals $\lambda_\p$.
\end{lemma}
\begin{proof}

The global character  $\lambda_\p$ of $G_S$, corresponds via class field theory, to the central character of the corresponding specialized automorphic representation (say $\varPi_\p$). Furthermore the global character $\lambda_\p$ is determined, via the Chebotarev density theorem, by the following assignment given below:
\begin{align} \label{eq:tl0}
\notag	\lambda_\p: G_S &\rightarrow \overline{\Q}_p^\times, \\ 
	\Frob_\ell &\mapsto \ell^3S_\ell(\varPi_\p), \quad \forall \ \ell  \notin S.
\end{align}

For a Galois deformation theoretic deduction of this assignment, one can refer to \cite[sections 3 and 4]{MR2234862}. For a direct deduction of this assignment from a Hecke theoretic perspective, one can refer to the calculations in \cite[Section 7]{MR4771233}. The lemma now follows by employing a density argument as in Theorem \ref{thm:pseudochar} and by considering the following assignment:
\begin{align*}
	\widetilde{\lambda}: G_S &\rightarrow \TTT^\times, \\
	\Frob_\ell &\mapsto \ell^3S_\ell,  \quad \forall \ \ell  \notin S.
\end{align*}

The precise shape of the similitude character  follows by using Theorem \ref{thm:urban1}. One sees that the global $p$-adic character  $\dfrac{\widetilde{\lambda}}{\cyc^{-3}\omega^a\omega^b\kappa_1 \kappa_2}$ of $G_S$ is valued in $\TTT^\times$ and  is unramified at $p$ (since at a dense set of classical points, its specializations  are unramified at $p$ by Theorem \ref{thm:urban1}); hence it must have finite image.  
\end{proof}

%\begin{remark}
%For each $g\in G_S$, the determinant of the image  of $g$ in the GMA (see Section~\ref{subsec:gmas}) afforded by equation~(\ref{iso:gma}) equals $\widetilde{\lambda}^2(g)$.
%\end{remark}

One may also employ a density argument as in Theorem \ref{thm:pseudochar} for pseudocharacters of $G_{\Q_p}$. Using such a density argument along with Theorem \ref{thm:urban1}, we similarly obtain two unramified characters of $G_{\Q_p}$ as follows: 
\begin{align}\label{eq:psi01def}
\widetilde{\psi}_0: G_{\Q_p} &\rightarrow \TTT^\times, \qquad \qquad & \widetilde{\psi}_1: G_{\Q_p} &\rightarrow \TTT^\times,  \\
\notag \Frob_p &\mapsto U_{\cP}, & \Frob_p &\mapsto \dfrac{U_{\cQ}}{U_{\cP}}.
\end{align}
Furthermore, one can conclude that the restriction of the pseudocharacter $\pcT$ to $G_{\Q_p}$ is a sum of four characters as described below:
\begin{align}\label{eq:bigpseudorepdecompositionlocal}
\pcT \big\vert_{G_{\Q_p}} = 	\dfrac{\widetilde{\lambda}}{\widetilde{\psi}_0} +  \dfrac{\widetilde{\lambda}}{ \widetilde{\psi}_1\cyc^{-2}\omega^b\kappa_2} + \widetilde{\psi}_1\cyc^{-2}\omega^b\kappa_2 +  \widetilde{\psi}_0. 
\end{align}

\subsection{Generalized Matrix Algebras}\label{subsec:gmas}

We let $\TTT=\hhh(N)_\m$ be a local component of the Hecke algebra $\hhh(N)$.  Note that $\TTT$ is also a reduced local ring.  We let $\m$ denote the maximal ideal of $\TTT$ as well. We let $\Frac(\TTT)$ denote $\TTT \otimes_\Lambda \Frac(\Lambda)$.  In this section, we shall recall the structure theory of pseudocharacters using generalized matrix algebras (following Bella\"iche--Chenevier \cite{MR2656025}). The structure theory shall be applied to the  pseudocharacter:
\[\pcT: G_S \rightarrow \bbT. \] %\twoheadrightarrow \TTT. \] 
We will use the same symbol $\pcT$ for its  natural extension to $\TTT[G_S]$.

\begin{hypothesis}\label{item:globalhyp} 	For $i\in \{1,2\}$, suppose there exist two absolutely irreducible \textit{$p$-ordinary} \textit{modular} Galois representations \[\overline{\varrho}_i:G_S \rightarrow \Gl_2(\overline{\mathbb{F}}_p),\]  
such that for each $g$ in $G_S$, we have 
\begin{enumerate}[noitemsep]
\item $\overline{\varrho}_1 \ncong \overline{\varrho}_2$ as representations over $G_S$, and  
\item $\det(\overline{\varrho}_1)= \det(\overline{\varrho}_2)$ as characters of $G_S$, and 
\item $\pcT(g) \equiv \mathrm{Trace}(\overline{\varrho}_1(g)) + \mathrm{Trace}(\overline{\varrho}_2(g))\  \left(\mathrm{mod} \  \mmm \right)$. \\
\end{enumerate}

\end{hypothesis}
\begin{hypothesis}\label{item:localhypord}
$\det(\overline{\varrho}_1)   \equiv \widetilde{\lambda} \ (\mathrm{mod} \ \m)$ and $\det(\overline{\varrho}_2)   \equiv \widetilde{\lambda} \ (\mathrm{mod} \ \m)$.
\end{hypothesis}

\begin{hypothesis}\label{item:localhypdisgsp4}
 The residual characters of $G_{\Q_p}$ associated to the four characters appearing in equation (\ref{eq:bigpseudorepdecompositionlocal}) are pairwise distinct.  
\end{hypothesis}

Without loss of generality, we may thus write 
\begin{align} \label{eq:residuallyordinary}
\overline{\varrho}_2\big\vert_{G_{\Q_p}}  \equiv  \dfrac{\widetilde{\lambda}}{\widetilde{\psi}_0}  + \widetilde{\psi}_0  \ (\mathrm{mod}\ \mmm), \qquad 	\overline{\varrho}_1\big\vert_{G_{\Q_p}}  \equiv \dfrac{\widetilde{\lambda}}{ \widetilde{\psi}_1\cyc^{-2}\omega^b\kappa_2} + \widetilde{\psi}_1\cyc^{-2}\omega^b\kappa_2 \ (\mathrm{mod}\ \mmm).
\end{align}

The reducibility ideal $\JJJ_{\mathrm{red}}$ is the smallest integral ideal inside $\TTT$ such that 
\begin{enumerate}[label=(\roman*)]
\item The pseudocharacter $G_S \xrightarrow{\pcT} \TTT \rightarrow \TTT/\JJJ_{\mathrm{red}}$ is a sum of two $2$-dimensional pseudocharacters. 
\item If $I$ is an integral ideal in $\TTT$ such that the pseudocharacter $G_S \xrightarrow{\pcT} \TTT \rightarrow \TTT/I$ is a sum of two $2$-dimensional pseudocharacters,then $I \supset \JJJ_{\mathrm{red}}$. 
\end{enumerate}

The existence of the reducibility ideal is established in \cite[Proposition 1.5.1]{MR2656025}. Let $\BBB$ and $\CCC$ be two finitely generated fractional ideals inside $\Frac(\TTT)$ satisfying the following conditions:
\begin{enumerate}
\item (reducibility loci) $\BBB \CCC = \JJJ_{\mathrm{red}}$. 
\item (semi-simplicity) Let $\eta$ be a minimal prime ideal of $\TTT$.	If under the natural embedding \[i_\eta: \TTT \twoheadrightarrow \TTT/\eta \hookrightarrow \Frac(\TTT_\eta),\] the fractional ideal $i_\eta(\BBB\CCC)$ equals zero, then the  fractional ideals $i_\eta(\BBB)$ and $i_\eta(\CCC)$ are both equal to zero.
\end{enumerate}

 The algebra
\begin{align}
	\left(\begin{array}{cc} M_2(\TTT) & M_2(\mathcal{B}) \\ M_2(\mathcal{C}) & M_2(\TTT) \end{array}\right),
\end{align}
naturally viewed as a subalgebra of $M_4\left(\Frac(\TTT)\right)$, is said to be a generalized matrix algebra (or a \textit{GMA} for short).  One can restrict the trace function on $M_4\left(\Frac(\TTT)\right)$ to obtain a trace function on the GMA. The image of the trace function on the GMA is valued in $\TTT$. The GMA naturally acts on the following torsion-free $\TTT$-module of rank four:
\begin{align} \label{eq:defininglattice}
\LLLL \coloneqq \BBB' \oplus \BBB' \oplus \TTT \oplus \TTT.
\end{align}
where $\BBB'$ equals the $\TTT$-module $\BBB + \TTT$. Here, the $\TTT$-module sum is considered inside $\Frac(\TTT)$. Since $\LLLL$ is a torsion-free module of rank four inside $\Frac(\TTT)$, we can find a non-zero element $\theta$ such that 
\begin{align}\label{eq:opencondnfreeness}
\text{$\LLLL_\p$ is a free $\TTT_\p$-module, for all primes $\p$ in $\TTT$ not containing $\theta$.} 	
\end{align}
 Note that since $\LLLL$ is a torsion-free $\TTT$-module of rank four, one has a natural inclusion \[\LLLL \hookrightarrow \LLLL \otimes_\TTT \Frac(\TTT) \cong \Frac(\TTT)^4.\] By considering the standard basis over $\Frac(\TTT)$, we may equip $\LLLL$ with a natural action of the GMA induced from the natural action of $M_4\left(\Frac(\TTT)\right)$ on $\Frac(\TTT)^4$. Since $\BBB'$ is finitely generated over $\TTT$, one can naturally equip $\LLLL$ with the $\Lambda$-adic topology.

\begin{theorem}[Bella\"iche--Chenevier, \cite{MR2656025}]\label{thm:pseudorepBC} Suppose that Hypotheses \ref{item:globalhyp}, \ref{item:localhypord} and \ref{item:localhypdisgsp4} hold. There exist two finitely generated $\TTT$-fractional ideals $\BBB$ and $\CCC$  that satisfy the reducibility loci condition (see condition (\ref{item:pseudo2}) below) and the semi-simplicity condition 
(see condition (\ref{item:pseudo3}) below), and also satisfy the following conditions:
\begin{enumerate}
\item\label{item:pseudo1} 	There exists an isomorphism of $\TTT$-algebras 
{\def\arraystretch{2}
	\begin{align}\label{iso:gma}
	\dfrac{\TTT[G_S]}{\ker(\pcT)} \cong \left(\begin{array}{cc} M_2(\TTT) & M_2(\mathcal{B}) \\ M_2(\mathcal{C}) & M_2(\TTT) \end{array}\right).	
	\end{align}}
	such that under the isomorphism given in equation (\ref{iso:gma}), the pseudocharacter $\pcT$ on $G_S$ coincides with the trace function on the GMA on the RHS of equation (\ref{iso:gma}). 	
\item\label{item:pseudo2}   $\BBB \CCC = \JJJ_{\mathrm{red}}$.
	\item\label{item:pseudo3}  Let $\eta$ be a minimal prime in $\TTT$. Suppose that under the natural embedding \[i_\eta: \TTT \twoheadrightarrow \TTT/\eta \hookrightarrow \Frac(\TTT_\eta),\] the ideal $i_\eta(\BBB\CCC)$ equals $(0)$. Then, the  fractional ideals $i_\eta(\BBB)$ and $i_\eta(\CCC)$ are both equal to $(0)$.

	\item\label{item:pseudo4}  The natural induced map $G_S \rightarrow \left(\begin{array}{cc} M_2(\TTT) & M_2(\mathcal{B}) \\ M_2(\mathcal{C}) & M_2(\TTT) \end{array}\right)$ is continuous.
\item\label{item:pseudo5}  There exists a non-zero element $\Theta$ in $\Lambda$ such that for all primes $\p$ in $\frakX^{\rm temp}_{\charcomp}\left(\TTT\right)$ not containing $\Theta$, the $\TTT_\p$-module $\LLLL_\p$ is free of rank four and the induced Galois representation \[G_S \rightarrow \mathrm{Aut}_{\TTT/\p}\left(\LLLL_\p/\p\LLLL_\p\right) \cong \mathrm{GL}_4\left(\Frac(\TTT/\p)\right)\] is semi-simple and is isomorphic to the semi-simple Galois representation given in Theorem \ref{thm:GSP4weissauertaylorlaumon}.

\item\label{item:pseudo6}  Let $I$ be an ideal in $\TTT$ containing $\JJJ_{\mathrm{red}}$.  There exist two pseudocharacters $\pcT_{1,I}:G_S \rightarrow \TTT/I$ and $\pcT_{2,I}:G_S\rightarrow \TTT/I$ such that 
\begin{enumerate}
\item $\pcT (g) = \pcT_{1,I}(g) + \pcT_{2,I}(g) \ (\mathrm{mod} \ I)$, for all $g$ in $G_S$, 
\item for each $i$ in $\{1,2\}$, we have $\pcT_{i,I} \equiv \mathrm{Trace}(\overline{\varrho}_i) \ (\mathrm{mod} \ \mmm)$,
\item For each $i$ in $\{1,2\}$, the pseudocharacter $\pcT_{i,I}$ is given by the following composition of maps:
\begin{align*}
G_S  \rightarrow \dfrac{\TTT[G_S]}{\ker(\pcT)} \xrightarrow{\cong} \left(\begin{array}{cc} M_2(\TTT) & M_2(\mathcal{B}) \\ M_2(\mathcal{C}) & M_2(\TTT) \end{array}\right) \rightarrow e_i\left(\begin{array}{cc} M_2(\TTT) & M_2(\mathcal{B}) \\ M_2(\mathcal{C}) & M_2(\TTT) \end{array}\right)e_i &\xrightarrow{\cong}  M_2(\TTT) \rightarrow  \\ & \rightarrow M_2(\TTT/I) \xrightarrow {\mathrm{Trace}} \TTT/I,
\end{align*}

\end{enumerate}

Here, $e_1$ and $e_2$ are idempotents in the GMA given below: 
\begin{align}
e_1 =  \left(\begin{array}{cc} I_2 & 0_2 \\ 0_2  & 0_2 \end{array}\right), \qquad e_2 =  \left(\begin{array}{cc} 0_2 & 0_2 \\ 0_2  & I_2 \end{array}\right),
\end{align}
where, $I_2$ and $0_2$ are the $2\times 2$ identity matrix and zero matrix respectively. 

\item\label{item:pseudo7} For every element $\sigma$ in the decomposition group $\Gal{\overline{\Q}_p}{\Q_p}$, its image in the GMA $\left(\begin{array}{cc} M_2(\TTT) & M_2(\mathcal{B}) \\ M_2(\mathcal{C}) & M_2(\TTT) \end{array}\right)$ has the following shape: 	
\begin{align} \label{eq:galoisordinary}
  \left(\begin{array}{cccc}  \dfrac{\widetilde{\lambda}}{ \widetilde{\psi}_1\cyc^{-2}\omega^b\kappa_2}(\sigma) & \star & 0 & \star \\ 0 &\widetilde{\psi}_1\cyc^{-2}\omega^b\kappa_2(\sigma)   & 0 & \star \\ \star & \star &  \dfrac{\widetilde{\lambda}}{\widetilde{\psi}_0}(\sigma) & \star \\ 0 & 0 & 0 & \widetilde{\psi}_0(\sigma) \end{array} \right).
\end{align}
\end{enumerate}

\end{theorem}

\begin{proof}
	One can use \cite[Theorem 1.4.4]{MR2656025} to conclude that there exist two $\TTT$-fractional ideals $\BBB$ and $\CCC$ such that there is an isomorphism of $\TTT$-algebras, as given in equation (\ref{iso:gma}), such that the pseudocharacter on $G_S$ coincides with the trace function on the GMA on the RHS of equation (\ref{iso:gma}). The trace function and the pseudocharacter $\pcT$ are faithful (in the sense of \cite[Section 1.2.4]{MR2656025}) as we have taken the quotient by the two-sided ideal $\ker(\pcT)$. Note that (\ref{item:pseudo2}) follows from \cite[Proposition 1.5.1]{MR2656025}. To complete the proof of (\ref{item:pseudo1}), we will have to show that $\BBB$ and $\CCC$ are finitely generated $\TTT$-modules.

We will first be required to prove (\ref{item:pseudo3}). Let $\{\eta_1,\cdots\eta_t\}$ denote the minimal primes of $\TTT$. Suppose without loss of generality that $\eta=\eta_1$. For the sake of contradiction, assume that $i_\eta(\CCC)$ equals $(0)$, but $i_\eta(\BBB)$ does not equal $(0)$. Choose an element $b \in \TTT$ such that $i_\eta(b) \neq 0$. If $t > 1$, choose an element $x$ that belongs to the $\displaystyle \bigcap \limits_{i=2}^t \eta_i \setminus \eta$. Otherwise if $t=1$, note that $\eta$ must equal $(0)$ since $\TTT$ is reduced; choose $x$ to equal $1$. We claim that the non-zero element $\left(\begin{array}{cc}0_{2\times2} & x b\cdot I_{2\times2} \\ 0_{2\times2} & 0_{2\times2} \end{array}\right)$ of the GMA belongs to kernel of the trace map as defined in equation (\ref{def:kerneltrace}). To see this, observe that this element of the GMA reduces to the zero matrix under $i_{\eta_j}$ for every minimal prime $\eta_j \neq \eta$, since $\displaystyle x \in \bigcap_{i=2}^t \eta_i \setminus \eta$. Since $i_\eta(\CCC)=(0)$, the following computation then shows that the non-zero element $\left(\begin{array}{cc}0_{2\times2} & x b\cdot I_{2\times2} \\ 0_{2\times2} & 0_{2\times2} \end{array}\right)$ of the GMA belongs to kernel of the trace map:

\[\left(\begin{array}{cc} 0_{2\times2} & i_\eta(x b)I_{2\times2}  \\  0_{2\times2} &  0_{2\times2} \end{array}\right) \cdot \left(\begin{array}{cc}U & V \\ 0_{2\times2}  & W \end{array}\right) = \left(\begin{array}{cc}0_{2\times2} & i_\eta(x b) W\\ 0_{2\times2}  & 0_{2\times2}  \end{array}\right), \qquad \forall \  U,W \in M_2(\TTT/\eta), \ V \in M_2(i_\eta(\BBB)). \]

This contradicts the fact that the trace map on the GMA is faithful. \\ 

To show that $\BBB$ and $\CCC$ are finitely generated $\TTT$-modules, it is enough to show that  for each minimal prime $\eta$, the $\TTT/\eta$-modules $i_\eta(\BBB)$ and $i_\eta(\CCC)$ are finitely generated since we have natural injections $\BBB \hookrightarrow \prod i_\eta(\BBB)$ and $\CCC \hookrightarrow \prod i_\eta(\CCC)$ over the Noetherian ring $\TTT$. Fix a minimal prime $\eta$. Without loss of generality, we may thus assume that both $i_\eta(\BBB)$ and $i_\eta(\CCC)$ are non-zero fractional ideals since we have just shown that if one of the fractional ideals equals $(0)$, then the other fractional ideal is automatically $(0)$. Let $b$ and $c$ be non-zero elements of $i_\eta(\BBB)$ and $i_\eta(\CCC)$ respectively. Since the product $i_\eta(\BBB\CCC)$ equals the integral $\TTT/\eta$-ideal $i_\eta(\JJJ_{\mathrm{red}})$, the product ideals $c \cdot i_\eta(\BBB)$ and $b \cdot i_\eta(\CCC)$ must also be integral $\TTT/\eta$-ideals and hence finitely generated. Since $\TTT/\eta$ is an integral domain, the natural isomorphisms of $\TTT/\eta$-modules $i_\eta(\BBB) \cong c \cdot i_\eta(\BBB)$ and $i_\eta(\CCC) \cong b \cdot i_\eta(\CCC)$ now tell us that $i_\eta(\BBB)$ and $i_\eta(\CCC)$ are finitely generated $\TTT/\eta$-modules. \\ 

 To prove continuity in (\ref{item:pseudo4}), note that $\TTT$ is a profinite topological ring and hence compact. Note also that the GMA $\left(\begin{array}{cc} M_2(\TTT) & M_2(\mathcal{B}) \\ M_2(\mathcal{C}) & M_2(\TTT) \end{array}\right)$ is finitely generated over $\TTT$. Let $g_1,\cdots,g_s$ denote elements of $G_S$ that generate the LHS of equation (\ref{iso:gma}) as a $\TTT$-module. Since the pseudocharacter is continuous, we can define a continuous map $\varsigma:G_S \rightarrow \TTT^s$, as given in equation (\ref{eq:isocontinuity}), from the LHS of  equation (\ref{iso:gma}) to $\TTT^s$. The map $\varsigma$ must be continuous since the pseudocharacter is continuous.  For each element $g$ in $G_S$, we let $A_g$ denote the corresponding element of the GMA under (\ref{iso:gma}). We may similarly define a continuous map $\varsigma_0$, as given in equation (\ref{eq:isocontinuity}), from the RHS of equation  (\ref{iso:gma}) to $\TTT^s$. Since $\pcT$ is continuous, faithful and since its domain is compact, the map $\varsigma_0$ defines a homeomorphism onto $\mathrm{Image}(\varsigma_0)$. We obtain the desired continuity result:
\begin{align}\label{eq:isocontinuity}
\xymatrix@R=0.5mm{
G_S \ar[r]^{\varsigma} & \mathrm{Image}(\varsigma_0)  & \ar[l]^{\varsigma_0 \qquad}_{\qquad \cong \qquad \qquad} {\left(\begin{array}{cc} M_2(\TTT) & M_2(\mathcal{B}) \\ M_2(\mathcal{C}) & M_2(\TTT) \end{array}\right) }  \\ 
 g  \ar@{|->}[r]& (\pcT(gg_1),\cdots 	\pcT(gg_s)) &  A_g  \ar@{|->}[l]
}
\end{align}

To prove (\ref{item:pseudo5}), we first choose $\Theta_1 \in \Lambda$ that annihilates the cokernel of the following natural injections of $\Lambda$-modules with ranks equal to  $4 \times \mathrm{rank}_{\Lambda}(\TTT)$  and  $\mathrm{rank}_{\Lambda}(\TTT)$  respectively:
\[\TTT \oplus \TTT  \oplus \TTT  \oplus \TTT \hookrightarrow \LLLL, \qquad \TTT \hookrightarrow \BBB + \TTT \]
The sum $\BBB + \TTT$ is considered inside $\Frac(\TTT)$. Let $\p$ be a prime in $\frakX^{\rm temp}_{\charcomp}\left(\TTT\right)$ not containing $\Theta_1$. Note that the action of $G_S$ on the localization $\LLLL_\p$ may not be continuous. The action of $G_S$ on the quotient $\LLLL_\p/\p\LLLL_\p$ (which is a free $\TTT_\p/\p\TTT_\p$-module of rank four) is given by the image of the GMA in equation (\ref{iso:gma}) under the natural reduction $\TTT_\p \rightarrow \TTT_\p/\p\TTT_\p$. To ensure semi-simplicity, for each minimal prime $\eta$ such that $i_\eta(\BBB\CCC) \neq (0)$, choose an element $x_\eta$ in the integral ideal $i_\eta(\BBB\CCC)$; let $\Theta_\eta$ equal $\mathrm{Norm}_{\Frac(\Lambda)}^{\TTT_\eta}(x_\eta)$. If, for a minimal prime $\eta$, the ideal $i_\eta(\BBB\CCC)$ equals  $(0)$, then let $\Theta_\eta$ equal $1$. Observe that the following choice of $\Theta$ works to prove (\ref{item:pseudo5}). 
 \[\Theta = \Theta_1 \times \prod_{i=1}^t \Theta_{\eta_i}.\]

(\ref{item:pseudo6}) follows from \cite[Lemma 1.5.4]{MR2656025}.

To prove (\ref{item:pseudo7}), 
it suffices to produce a \textit{change of basis} matrix $M$ in  $\left(\begin{array}{cc} \Gl_2(\TTT) & M_2(\mathcal{B}) \\ M_2(\mathcal{C}) & \Gl_2(\TTT) \end{array}\right)$ and then considering the conjugate $M^{-1}\left(\begin{array}{cc} M_2(\TTT) & M_2(\mathcal{B}) \\ M_2(\mathcal{C}) & M_2(\TTT) \end{array}\right)M$  of the GMA $\left(\begin{array}{cc} M_2(\TTT) & M_2(\mathcal{B}) \\ M_2(\mathcal{C}) & M_2(\TTT) \end{array}\right)$ satisfying properties (\ref{item:pseudo1})-(\ref{item:pseudo6}) of Theorem \ref{thm:pseudorepBC} by the change of basis matrix $M$. \\

	  Since $\overline{\varrho}_1$ and $\overline{\varrho}_2$ are absolutely irreducible, they are valued in $\Gl_2(\overline{\mathbb{F}}_p)$. Since $p \geq 3$, we can apply the Brauer--Nesbitt theorem. The Brauer--Nesbitt theorem and equation (\ref{eq:residuallyordinary}) then tell us that there exist constant matrices $\overline{S_1}$ and $\overline{S}_2$ in $\Gl_2(\overline{\mathbb{F}}_p)$ such that the conjugates $\overline{S}_i^{-1}\overline{\varrho}_i\vert_{G_{\Q_p}} \overline{S}_i$  are either upper-triangular or lower-triangular. Choose (any) lifts $S_1$ and $S_2$ in $\Gl_2(\TTT)$. Let $\mathcal{A}_0$ denote the GMA $\left(\begin{array}{cc} M_2(\TTT) & M_2(\mathcal{B}) \\ M_2(\mathcal{C}) & M_2(\TTT) \end{array}\right)$  afforded by Theorem \ref{thm:pseudorepBC}. Consider its conjugate \[\mathcal{A}_1 \coloneqq  \left(\begin{array}{cc} S_1 & 0_{2 \times 2} \\ 0_{2 \times 2} & S_2\end{array}\right)^{-1}\mathcal{A}_0 \left(\begin{array}{cc} S_1 & 0_{2 \times 2} \\ 0_{2 \times 2} & S_2\end{array}\right),\] which, as a GMA, is also equal to  $\left(\begin{array}{cc} M_2(\TTT) & M_2(\mathcal{B}) \\ M_2(\mathcal{C}) & M_2(\TTT) \end{array}\right)$.  Consequently, for each $\sigma \in G_{\Q_p}$, its image in $\mathcal{A}_1$ is equal to 	
{\setstretch{2.5}	\begin{align}\label{eq:localmatrixshape} \left(\begin{array}{cc:cc}  \dfrac{\widetilde{\lambda}}{ \widetilde{\psi}_1\cyc^{-2}\omega^b\kappa_2}(\sigma) + m_{\sigma,(1,1)} & x_{\sigma,(1,2)} & b_{\sigma,(1,3)} & b_{\sigma,(1,4)}\\ x_{\sigma,(2,1)} &\widetilde{\psi}_1\cyc^{-2}\omega^b\kappa_2(\sigma) + m_{\sigma,(2,2)}   & b_{\sigma,(2,3)} & b_{\sigma,(2,4)} \\ \hdashline[4pt/4pt]    c_{\sigma,(3,1)} & c_{\sigma,(3,2)} &  \dfrac{\widetilde{\lambda}}{\widetilde{\psi}_0}(\sigma) + m_{\sigma,(3,3)} &  x_{\sigma,(3,4)}  \\ c_{\sigma,(4,1)} & c_{\sigma,(4,2)} & x_{\sigma,(4,3)}  & \widetilde{\psi}_0(\sigma)+ m_{\sigma,(4,4)}  \end{array} \right),
		\end{align}}		
	where, the elements $m_{\sigma,(1,2)}$, $m_{\sigma,(2,2)}$, $m_{\sigma,(3,3)}$, $m_{\sigma,(4,4)}$ belong to $\mmm$ while the elements $x_{\sigma,(1,2)}$, $x_{\sigma,(2,1)}$,$x_{\sigma,(3,4)}$ and $x_{\sigma,(4,3)}$ belong to $\TTT$. Let $B^{\mathrm{loc}}_{(1,2)}$ and $B^{\mathrm{loc}}_{(3,4)}$ denote the integral $\TTT$-ideal generated by $x_{\sigma,(1,2)}$'s and $x_{\sigma,(3,4)}$'s respectively, as $\sigma$ varies over all elements of $G_{\Q_p}$.  Similarly, let $C^{\mathrm{loc}}_{(2,1)}$ and $C^{\mathrm{loc}}_{(4,3)}$ denote the integral $\TTT$-ideal generated by $x_{\sigma,(2,1)}$'s and $x_{\sigma,(4,3)}$'s respectively, as $\sigma$ varies over all elements of $G_{\Q_p}$. Since $\overline{\varrho}_1$ and $\overline{\varrho}_2$ are $p$-ordinary, furthermore we have 
\begin{align}\label{eq:pordinclusion}
	B^{\mathrm{loc}}_{(1,2)} C^{\mathrm{loc}}_{(2,1)} \subset \mmm, \qquad 	B^{\mathrm{loc}}_{(3,4)} C^{\mathrm{loc}}_{(4,3)} \subset \mmm.
\end{align}

We will first show that there exists a non-zero vector $\vec{v}$ in $\Frac(\TTT)^4$ (unique upto a scalar) satisfying equation (\ref{eq:propertyvector}). We follow \cite[Proof of Lemma 7.2]{MR1693583}. Choose an element $\sigma_0$ in $G_{\Q_p}$ mapping to a topological generator of $\Gal{\Q_{p,\mathrm{cyc}}}{\Q_p}$. One can use equation (\ref{eq:bigpseudorepdecompositionlocal}) and the Brauer--Nesbitt theorem over each of the fields appearing in the field decomposition of $\Frac(\TTT)$ to conclude that we can construct a vector $\vec{v}$ (unique upto a scalar) in $\Frac(\TTT)$ such that $\sigma_0(\vec{v})=\dfrac{\widetilde{\lambda}}{\widetilde{\psi}_0}(\sigma_0)\vec{v}$. By considering the numerators and denominators of the coordinates of $\vec{v}$, we can find a non-zero element $\Theta_2$ in $\Lambda$ such that for all prime ideals $\p$ in $\TTT$ such that $\Theta_2 \notin \p$, the vector $\vec{v}$ is an element in $\TTT_\p^4-\p\TTT_\p^4$. One can now use Theorem \ref{thm:pseudorepBC}(\ref{item:pseudo5}) to consider that the natural image $i_\p(\vec{v})$ (which is a non-zero vector) of $\vec{v}$ inside the semi-simple Galois module $\LLLL_\p/\p\LLLL_\p \cong \Frac(\TTT/\p)^4$, for all primes $\p \in \frakX^{\rm temp}_{\charcomp}\left(\TTT\right)$ such that $\Theta \Theta_2 \notin \p$. The action of the Galois group is provided by the natural reduction modulo $\p$ of the GMA $\mathcal{A}_1$. For all such classical primes $\p$, Theorem \ref{thm:urban1} tells us that inside $\LLLL_\p/\p\LLLL_\p$, there is a unique one-dimensional $\Frac(\TTT/\p)$-subspace, generated by $i_\p(\vec{v})$, on which $G_{\Q_p}$ acts via the character $\dfrac{\widetilde{\lambda}}{\widetilde{\psi}_0}$ modulo $\p$. As a result, we have 
\begin{align}\label{eq:propertyvector}\sigma \cdot i_\p(\vec{v}) - \dfrac{\widetilde{\lambda}}{\widetilde{\psi}_0}(\sigma)\vec{v}   \equiv 0 \ \mathrm{mod} \  \p, \qquad \text{ for all} \ \sigma \in G_{\Q_p} \text{ and for all } \p \in \frakX^{\rm temp}_{\charcomp}\left(\TTT\right) \text{ such that } \Theta\Theta_2 \notin \p. \end{align}
By using the density result provided in Lemma \ref{lem:density} and the fact that the local ring $\TTT$ is reduced, one can conclude that $\sigma \vec{v} =  \dfrac{\widetilde{\lambda}}{\widetilde{\psi}_0}(\sigma)\vec{v} $, for all $\sigma$ in $G_{\Q_p}$.
 Observe that the $p$-distinguished hypothesis \ref{item:localhypdisgsp4} tells us that there exist elements $\sigma_1$, $\sigma_2$ and $\sigma_4$ in $G_{\Q_p}$ such that modulo $\mmm$, we have
  \[\dfrac{\widetilde{\lambda}}{ \widetilde{\psi}_1\cyc^{-2}\omega^b\kappa_2}(\sigma_1) \not\equiv \dfrac{\widetilde{\lambda}}{\widetilde{\psi}_0}(\sigma_1), \quad \widetilde{\psi}_1\cyc^{-2}\omega^b\kappa_2(\sigma_2) \not\equiv \dfrac{\widetilde{\lambda}}{\widetilde{\psi}_0}(\sigma_2), \qquad \widetilde{\psi}_0(\sigma_4) \not\equiv \dfrac{\widetilde{\lambda}}{\widetilde{\psi}_0}(\sigma_4). \]
 As a result, the $(1,1)$, $(2,2)$ and $(4,4)$ entries, of the  matrix in equation (\ref{eq:matricesgaussian}), are units in $\TTT$.  This matrix is obtained by choosing the first, second, third and fourth row respectively from the images of $\sigma_1$, $\sigma_2$, $\sigma_1$ and $\sigma_4$ respectively in $\mathcal{A}_1$.   One observes that the nullspace of this matrix is one-dimensional, and that the vector $\vec{v}$, being a non-zero vector in its nullspace, must generate it. Let $\vec{v}=\left[ \begin{array}{c}b_{1,3} \\ b_{2,3} \\ r \\ r_{4,3} \end{array}\right]$.  Applying Gaussian elimination tells us that 
 
 \begin{align} \label{eq:columnvectorchoice}
 		\dfrac{b_{1,3}}{r} \in \BBB, \qquad \dfrac{b_{2,3}}{r} \in \BBB, \qquad \dfrac{r_{4,3}}{r} \in \TTT.
 	\end{align}
 One can then normalize the vector $\vec{v}$ by choosing $r$ to equal $1$.  This normalized choice of $\vec{v}$ becomes the third column of the change of basis matrix $M$.

{\setstretch{2}\begin{align}\label{eq:matricesgaussian} \left(\begin{array}{c:c:c:c}  \dfrac{\widetilde{\lambda}}{ \widetilde{\psi}_1\omega^b\cyc^{-2}\kappa_2}(\sigma_1) - \dfrac{\widetilde{\lambda}}{\widetilde{\psi}_0}(\sigma_1)+ m_{\sigma_j,(1,1)} & x_{\sigma_1,(1,2)} & b_{\sigma_1,(1,3)} & b_{\sigma_1,(1,4)}\\ \hdashline[4pt/4pt]  x_{\sigma_2,(2,1)} &\widetilde{\psi}_1\omega^b\cyc^{-2}\kappa_2(\sigma_2) - \dfrac{\widetilde{\lambda}}{\widetilde{\psi}_0}(\sigma_2) + m_{\sigma_2,(2,2)}   & b_{\sigma_2,(2,3)} & b_{\sigma_2,(2,4)} \\ \hdashline[4pt/4pt]    c_{\sigma_1,(3,1)} & c_{\sigma_1,(3,2)} &    m_{\sigma_1,(3,3)} &  x_{\sigma_1,(3,4)}  \\ \hdashline[4pt/4pt]  c_{\sigma_4,(4,1)} & c_{\sigma_4,(4,2)} & x_{\sigma_4,(4,3)}  & \widetilde{\psi}_0(\sigma_4) - \dfrac{\widetilde{\lambda}}{\widetilde{\psi}_0}(\sigma_4) + m_{\sigma_4,(4,4)}  \end{array} \right).
		\end{align}}

 One next conjugates the GMA $\mathcal{A}_1$ to obtain a new GMA $\mathcal{A}_2$ by updating the basis from the standard basis for $\mathcal{A}_1$ to $\left\{\left[\begin{array}{c}1 \\0 \\ 0 \\0 \end{array}\right], \left[\begin{array}{c}0 \\1 \\ 0 \\0 \end{array}\right],\vec{v},\left[\begin{array}{c}0 \\0 \\ 0 \\1 \end{array}\right]\right\}$. Note that this change of basis matrix  belongs to $\left(\begin{array}{cc} \Gl_2(\TTT) & M_2(\mathcal{B}) \\ M_2(\mathcal{C}) & \Gl_2(\TTT) \end{array}\right)$ and hence induces an automorphism of $\LLLL$. The updated basis produces a surjection of $\TTT[G_{\Q_p}]$-modules: \[\LLLL \twoheadrightarrow \BBB' \oplus \BBB' \oplus \{0\} \oplus \TTT.\] Combining our observations we note that for each $\sigma$ in $G_{\Q_p}$ its image in $\mathcal{A}_2$ will be of the following form, similar to the matrix in equation (\ref{eq:localmatrixshape}):
 
 {\setstretch{2.5}	\begin{align}\label{eq:localmatrixshape2} \left(\begin{array}{cc:c:c}  \dfrac{\widetilde{\lambda}}{ \widetilde{\psi}_1\cyc^{-2}\kappa_2}(\sigma) + m'_{\sigma,(1,1)} & x'_{\sigma,(1,2)} & 0 & b'_{\sigma,(1,4)}\\ x'_{\sigma,(2,1)} &\widetilde{\psi}_1\cyc^{-2}\kappa_2(\sigma) + m'_{\sigma,(2,2)}   & 0 & b'_{\sigma,(2,4)} \\ \hdashline[4pt/4pt]    c'_{\sigma,(3,1)} & c'_{\sigma,(3,2)} &  \dfrac{\widetilde{\lambda}}{\widetilde{\psi}_0}(\sigma)  &  x'_{\sigma,(3,4)}  \\ \hdashline[4pt/4pt] c'_{\sigma,(4,1)} & c'_{\sigma,(4,2)} & 0  & \widetilde{\psi}_0(\sigma)+ m'_{\sigma,(4,4)}  \end{array} \right).
		\end{align}}

 One now repeats the above procedure by deleting the third column and row from the matrices in equation (\ref{eq:localmatrixshape2}) and applying Gaussian elimination again, thereby iteratively deriving the shape of the $1$\textsuperscript{st} column and similarly then the $2$\textsuperscript{nd} column of $M$. These iterations would yield a constant change of basis matrix $M$ (with respect to $\mathcal{A}_1$) of the form 
 $\left(\begin{array}{cccc} 1 & 0 & b_{1,3} & 0 \\ r_{2,1} & 1 & b_{2,3} & 0 \\ 0 & 0 & 1 & 0 \\ c_{4,1} & c_{4,2} & r_{4,3} & 1  \end{array}\right)$. 
 where the elements $r_{2,1}, r_{4,3}$ belong to $\TTT$, the elements  $c_{4,1}, c_{4,2}$ belong to  $\CCC$ and the elements $b_{1,3}, b_{2,3}$ belong to $\BBB$. One sees that this constant matrix $M$ belongs to $\left(\begin{array}{cc} \Gl_2(\TTT) & M_2(\mathcal{B}) \\ M_2(\mathcal{C}) & \Gl_2(\TTT) \end{array}\right)$. This completes the proof.

\end{proof}

\begin{remark}
For our application to the Iwasawa main conjecture using Proposition \ref{prop:inclusion}, we need the change of basis matrix $M$ to have an integral structure so as to induce an automorphism of the $\TTT$-module $\LLLL$. The $p$-distinguished condition for $\mathrm{GSp}_4$ plays an important role. Keeping this in mind, we prefer to avoid directly applying Sen's theory (e.g.  used in \cite{MR1693583,MR1813234}) since we can then only deduce that the change of basis matrix has coefficients in $\mathrm{Frac}(\TTT)$.
\end{remark}

\subsection{Producing Selmer classes in the $\GSp_4$ setting}\label{sec:selmerclassprod}

Throughout this subsection, we maintain all the hypotheses in Theorem \ref{thm:pseudorepBC}. Let $\JJJ_0$ be a closed ideal in $\TTT$  containing the reducibility ideal $\JJJ_{\mathrm{red}}$. 
Consider a continuous map of local rings
\[\sigma:\TTT \twoheadrightarrow \dfrac{\TTT}{\JJJ_0} \hookrightarrow \III_0,\]
such that 
 $\III_0$ is finitely generated as a module over $\TTT/\JJJ_0$. Let $\widehat{\III_0} \coloneqq \Hom_{\mathrm{cont}}(\III_0, \Q_p/\Z_p)$ denote the Pontryagin dual of $\III_0$. Consider the GMA $\left(\begin{array}{cc} M_2(\TTT) & M_2(\mathcal{B}) \\ M_2(\mathcal{C}) & M_2(\TTT) \end{array}\right)$, call it $\mathfrak{A}$, afforded by Theorem \ref{thm:pseudorepBC}.  Using Theorem \ref{thm:pseudorepBC}(\ref{item:pseudo6}), we have two free $\III_0$-modules of rank $2$, call them $\SLLL_{1,\III_0}$ and $\SLLL_{2,\III_0}$ respectively, on which $G_S$ acts via the $2 \times 2$ GMAs $\sigma(e_1\mathfrak{A} e_1)$ and $\sigma(e_2\mathfrak{A} e_2)$ respectively. Furthermore, let $\Fil^+\SLLL_{2,\III_0} $ denote the free $\III_0$-module of rank $1$ with an action of $G_{\Q_p}$ via the character $\sigma \circ \dfrac{\widetilde{\lambda}}{\widetilde{\psi}_0}$. Theorem \ref{thm:pseudorepBC}(\ref{item:pseudo7}) provides us a natural inclusion $\Fil^+\SLLL_{2,\III_0} \hookrightarrow \SLLL_{2,\III_0}$ of free $\III_0$-modules that is $G_{\Q_p}$-equivariant. We consider the related discrete modules:
\begin{align*}
\DDD_{\III_0}\coloneqq \Hom_{\III_0}\left(\SLLL_{2,\III_0},\ \SLLL_{1,\III_0}\right)\otimes_{\III_0} \widehat{\III_0}, \qquad \DDD_{\III_0,\Fil^+} \coloneqq \Hom_{\III_0}\left(\dfrac{\SLLL_{2,\III_0}}{\Fil^+\SLLL_{2,\III_0}},\ \SLLL_{1,\III_0}\right)\otimes_{\III_0} \widehat{\III_0}.
\end{align*}
We remind the reader that $\Sigma_0 = S \setminus \{p\}$. We consider the discrete non-primitive Selmer group and the discrete non-primitive \textit{strict} Selmer group:
\begin{align*}
\Sel^{\Sigma_0}( \DDD_{\III_0}) &\coloneqq \ker\bigg(H^1_{\mathrm{cont}}\left(G_S, \DDD_{\III_0} \right) \xrightarrow{\mathrm{Res}} H^1_{\mathrm{cont}}\left(I_p, \ \dfrac{\DDD_{\III_0}}{\DDD_{\III_0,\Fil^+}} \right)  \bigg),\\
\Sel^{\Sigma_0}_{\mathrm{str}}( \DDD_{\III_0}) &\coloneqq \ker\bigg(H^1_{\mathrm{cont}}\left(G_S, \DDD_{\III_0} \right) \xrightarrow{\mathrm{Res}} H^1_{\mathrm{cont}}\left(\Gal{\overline{\Q}_p}{\Q_p}, \ \dfrac{ \DDD_{\III_0}}{\DDD_{\III_0,\Fil^+}} \right)  \bigg).
\end{align*}

Here, the quotient $\III_0$-module $\dfrac{ \DDD_{\III_0}}{\DDD_{\III_0,\Fil^+}}$ is isomorphic to $\Hom\left(\Fil^+\SLLL_{2,\III_0},\SLLL_{1,\III_0}\right)\otimes_{\III_0} \widehat{\III_0}$. Since $\SLLL_{1,\III_0}$ is a free $\III_0$-module, we have the natural $\III_0$-module isomorphism that is $G_S$-equivariant:
\begin{align}\label{iso:freehomtensor}
\Hom\left(\SLLL_{2,\III_0},\SLLL_{1,\III_0}\right)\otimes_{\III_0} \widehat{\III_0} \cong \Hom\left(\SLLL_{2,\III_0},\SLLL_{1,\III_0}\otimes_{\III_0} \widehat{\III_0}\right).
\end{align} 
 
To describe the action of $G_S$ on the module on the RHS of the equation above, note that if $g\in G_S$ and $\varphi \in \Hom\left(\SLLL_{2,\III_0},\SLLL_{1,\III_0}\otimes_{\III_0} \widehat{\III_0}\right)$, then $(g \cdot \varphi)(x)$ equals $\sigma_1(g) \cdot \varphi\left(\sigma_2^{-1}(g)\cdot x\right)$. Here, $\sigma_{1}(g)$ and $\sigma_{2}(g)$ are the natural images of $g$ inside the $2 \times 2$ GMAs $\sigma(e_1\mathfrak{A} e_1)$ and $\sigma(e_2\mathfrak{A} e_2)$ respectively. We let $\left[\begin{array}{cc}b_{g,1} & b_{g,2} \\ b_{g,3} & b_{g,4}\end{array}\right]$ denote the matrix element in $M_2(\BBB)$ corresponding to the natural image of $g$ in the GMA $\left(\begin{array}{cc} M_2(\TTT) & M_2(\mathcal{B}) \\ M_2(\mathcal{C}) & M_2(\TTT) \end{array}\right)$. We denote elements inside $\underbrace{\SLLL_{2,\III_0}}_{\cong \III_0^2}$ and $\underbrace{\SLLL_{1,\III_0} \otimes_{\III_0} \widehat{\III_0}}_{\cong \widehat{\III_0}^2}$ by column vectors. \\

We now construct an $\III_0$-module homomorphism of discrete $\III_0$-modules:

\begin{align}i_{\III_0}: \Hom_{\TTT}\left(\BBB,\widehat{\III_0}\right) \rightarrow H^1\left(G_S,\Hom\left(\SLLL_{2,\III_0},\SLLL_{1,\III_0}\right)\otimes_{\III_0} \widehat{\III_0}\right).
\end{align}
 Let $\phi \in \Hom_{\TTT}\left(\BBB,\widehat{\III_0}\right)$. The $1$-cocycle $i_{\III_0}(\phi)(g)$ is defined as follows:

\begin{align}\label{def:1cocycle}
i_{\III_0}(\phi): G_S &\rightarrow \Hom\left(\SLLL_{2,\III_0},\SLLL_{1,\III_0}\otimes_{\III_0} \widehat{\III_0}\right), \\[2mm] 
i_{\III_0}(\phi)(g) : \underbrace{\left[\begin{array}{c}x_2 \\ y_2\end{array}\right]}_{\in \SLLL_{2,\III_0}} &\mapsto \phi\left(\underbrace{\left[\begin{array}{cc}b_{g,1} & b_{g,2} \\ b_{g,3} & b_{g,4} \end{array}\right]\sigma_2(g)^{-1}\left[\begin{array}{c}x_2 \\ y_2\end{array}\right]}_{\in \BBB^2}\right) \in \SLLL_{1,\III_0}\otimes_{\III_0} \widehat{\III_0}. \notag
\end{align}

Theorem \ref{thm:pseudorepBC}(\ref{item:pseudo6}) allows one to check that $i_{\III_0}(\phi)(g)$ defines a $1$-cocycle. One can also verify that $i_{\III_0}$ is an $\III_0$-module homomorphism.  One can use the cocycle $i_{\III_0}(\phi)$ to build a $G_S$-extension as below: 

\begin{align}\label{extfrom1cocycle}
0 \rightarrow \underbrace{\SLLL_{1,\III_0}\otimes_{\III_0} \widehat{\III_0}}_{\cong \widehat{\III_0}^2}  \rightarrow E_{i_{\III_0}(\phi)} \rightarrow \underbrace{\SLLL_{2,\III_0}}_{\cong \III_0^2} \rightarrow 0.
\end{align}
As an $\III_0$-module, $E_{i_{\III_0}(\phi)}$ is isomorphic to $\widehat{\III_0}^2 \oplus \III_0^2$. The action of $G_S$ on $E_{i_{\III_0}(\phi)}$ is given below:

\begin{align}\label{formula:GSext}
g \cdot \left(\underbrace{\left[\begin{array}{c}x_1 \\ y_1\end{array}\right]}_{\in \widehat{\III_0}^2}, \underbrace{\left[\begin{array}{c}x_2 \\ y_2\end{array}\right]}_{\in \III_0^2}\right) = 
\left( \underbrace{\sigma_1(g)\left[\begin{array}{c}x_1 \\ y_1\end{array}\right] +i_{\III_0}(\phi)(g)\left(\sigma_2(g)\left[\begin{array}{c}x_2 \\ y_2\end{array}\right] \right)}_{\in \widehat{\III_0}^2}, \ \underbrace{\sigma_2(g)\left[\begin{array}{c}x_2 \\ y_2\end{array}\right]}_{\in \III_0^2} \right)
\end{align}
Conversely, any such $G_S$-extension as given in equation (\ref{extfrom1cocycle}) defines a $1$-cocycle in $H^1\left(G_S,\Hom\left(\SLLL_{2,\III_0},\SLLL_{1,\III_0}\right)\otimes_{\III_0} \widehat{\III_0}\right)$. The $1$-cocycle $i_{\III_0}(\phi)$ is trivial if and only if the $G_S$-extension in equation (\ref{extfrom1cocycle}) splits.

\begin{proposition}\label{prop:inclusionfirst}
We have a natural injection of discrete $\III_0$-modules:
\begin{align}i_{\III_0}: \Hom_{\TTT}\left(\BBB,\widehat{\III_0}\right) \hookrightarrow H^1_\cont\left(G_S,\Hom\left(\SLLL_{2,\III_0},\SLLL_{1,\III_0}\right)\otimes_{\III_0} \widehat{\III_0}\right)
\end{align} 
\end{proposition}

\begin{proof}
The proof of the injectivity follows the proof of \cite[Theorem 1.5.5]{MR2656025}. Suppose for the sake of contradiction that we have a non-trivial element $\phi$  in $\Hom_{\TTT}\left(\BBB,\widehat{\III_0}\right)$ and a splitting of a $G_S$-extension in equation (\ref{extfrom1cocycle}) via a $\TTT[G_S]$-module homomorphism:
\begin{align}
s: \SLLL_{2,\III_0} \rightarrow E_{i_{\III_0}(\phi)}.
\end{align}

Let $b_0$ be an element of $\BBB$ such that $\phi(b_0)$ is a non-zero element in $\widehat{\III_0}$. Choose an element $\alpha$ in the group ring $\TTT[G_S]$ such that its natural image in the GMA $\left(\begin{array}{cc} M_2(\TTT) & M_2(\mathcal{B}) \\ M_2(\mathcal{C}) & M_2(\TTT) \end{array}\right)$ equals $\left[\begin{array}{cccc}0 & 0 & b_0 & 0 \\ 0 & 0 & 0 & 0 \\ 0 & 0 & 0 & 0 \\ 0 & 0 & 0 & 0 \end{array}\right]$. \\

Note that \[s\left(\left[\begin{array}{c}1 \\ 0 \end{array}\right]\right)= \left(\left[\begin{array}{c}x_0 \\ y_0\end{array}\right], \left[\begin{array}{c}1 \\ 0\end{array}\right]\right), \quad \text{for some } x_0,y_0 \in \widehat{\III_0}.\]

On the one hand, $\alpha\cdot\left[\begin{array}{c} 1 \\ 0 \end{array}\right]$ equals $\left[\begin{array}{c} 0 \\ 0 \end{array}\right]$. Thus, $s\left(\alpha\cdot\left[\begin{array}{c} 1 \\ 0 \end{array}\right]\right)=\left(\left[\begin{array}{c} 0 \\ 0 \end{array}\right],\left[\begin{array}{c} 0 \\ 0 \end{array}\right]\right)$. \\

On the other hand, the formula in equation (\ref{formula:GSext}) tells us that $\alpha\cdot s \left(\left[\begin{array}{c} 1 \\ 0 \end{array}\right]\right)=  \left(\left[\begin{array}{c} \phi(b_0)  \\ 0 \end{array}\right], \left[\begin{array}{c} 0\\ 0 \end{array}\right]\right)$.  

We thus arrive at a contradiction that the splitting map is $G_S$-equivariant. Thus, the map $i_{\III_0}$ must be an injection.  \\

The proof of the continuity of the cocycle $i_{\III_0}(\phi)$ follows the proof of \cite[Proposition 1.5.10]{MR2656025}. Firstly, observe that $\phi$ is continuous. To see this, observe that since $\BBB$ is finitely generated, the image of $\phi$ lands inside $\widehat{\III_0}[\mmm^n]$, for some power $\mmm^n$ of the maximal ideal $\mmm$ of $\TTT$. As a result, $\ker(\phi)$ contains $\mmm^n\BBB$. This allows us to deduce the continuity of $\phi$ since $\BBB$ is endowed the $\mmm$-adic topology. The continuity of the cocycle $i_{\III_0}(\phi)$, as given in equation (\ref{def:1cocycle}), 
now follows from the continuity of the $G_S$-action from Theorem \ref{thm:pseudorepBC} and the continuity of $\phi$.
\end{proof}

\begin{proposition}\label{prop:inclusion}
We have a natural injection of discrete $\III_0$-modules:
	\begin{align}\label{eq:propdiscreteinj}
	i_{\III_0}:  \Hom_\TTT\left(\mathcal{B}, \ \widehat{\III_0}\right)	 \ \lhook\joinrel\longrightarrow  \Sel^{\Sigma_0}_{\mathrm{str}}( \DDD_{\III_0}) \ \lhook\joinrel\longrightarrow  \Sel^{\Sigma_0}( \DDD_{\III_0}).
	\end{align}
	Consequently, we have a surjection of compact $\III_0$-modules:
	\begin{align}\label{eq:compactsurjection}
		\Sel^{\Sigma_0}( \DDD_{\III_0})^\vee \twoheadrightarrow 		\Sel^{\Sigma_0}_{\mathrm{str}}( \DDD_{\III_0})^\vee \twoheadrightarrow \BBB \otimes_\TTT \III_0.
	\end{align}
\end{proposition}

\begin{proof}
It will be sufficient to prove the proposition for the strict Selmer group.  \\

The restriction of the cocycle $i_{\III_0}$ to $G_{\Q_p}$ for the $\III_0$-module $\Hom\left(\Fil^+\SLLL_{2,\III_0},\SLLL_{1,\III_0}\right)\otimes_{\III_0} \widehat{\III_0}$ is given below
\begin{align}\label{def:1cocycleres}
i_{\III_0}(\phi): G_{\Q_p} &\rightarrow \Hom\left(\Fil^+\SLLL_{2,\III_0},\SLLL_{1,\III_0}\otimes_{\III_0} \widehat{\III_0}\right), \\[2mm] 
i_{\III_0}(\phi)(g) : x &\mapsto \phi\left(x\left[\begin{array}{c}b_{g,1}  \\ b_{g,3}\\  \end{array}\right]\sigma_2(g)^{-1}\right) \in \SLLL_{1,\III_0}\otimes_{\III_0} \widehat{\III_0}. \notag
\end{align}

By Theorem \ref{thm:pseudorepBC}(\ref{item:pseudo7}), both $b_{g,1}$ and $b_{g,3}$ equal zero for all elements $g$ in $G_{\Q_p}$. Consequently, the image of $i_{\III_0}$ lies in $\Sel^{\Sigma_0}_{\mathrm{str}}( \DDD_{\III_0})$. \\ 

Equation (\ref{eq:compactsurjection}) follows from equation (\ref{eq:propdiscreteinj}) by considering Pontryagin duals and using the fact that one can identify the $\III_0$-module $\Hom_{\TTT}\left(\BBB, \ \widehat{\III_0}\right)$, by Hom-Tensor adjunction, with $\Hom_{\III_0}\left(\BBB \otimes_\TTT \III_0, \ \widehat{\III_0}\right)$; this latter module can now be identified with the Pontryagin dual of the $\III$-module $\BBB \otimes_\TTT \III_0$. See \cite[Section 2.9.1]{MR2333680} and \cite[Lemma 4.7]{MR4385094}. The proposition follows. 
\end{proof}

\section{Proofs of main results} \label{sec:proofs}

Let $\hF$ and $\hG$ be two primitive Hida families of tame levels $N_{\hF}$ and $N_{\hG}$. Let $f_0$ and $g_0$ be the specializations 
of the Hida families $\hF$ and $\hG$ at weights $k_{f_0}$ and $k_{g_0}$. Let 
$S_{\hF,\hG,\mathrm{disc}}$ be the set of finite places 
$q\mid \gcd(N_\hF,N_\hG)$ such that $\pi_{f_0,q}$ and $\pi_{g_0,q}$ 
are discrete series representations. We assume 
$S_{\hF,\hG,\mathrm{disc}}$ is non-empty.  In fact, we may simply choose the set $S_{\hF,\hG,\mathrm{disc}}$ to equal $\{l_0\}$. Throughout this section, we assume that the hypotheses from the  introduction are in force. In particular, they imply the hypotheses of Section~\ref{sec:yoshidalifts} needed to construct Yoshida lifts:  Hypothesis \ref{lab:uni} implies Hypothesis \ref{hyp:yoshida1}. Hypothesis \ref{lab:yosdet} implies Hypothesis \ref{hyp:yoshida2}. Hypothesis~\ref{lab:yos} implies Hypothesis \ref{hyp:yoshida3}. Hypothesis \ref{lab:yos-neb} implies Hypothesis \ref{hyp:yoshida4}.  One can consider the automorphic representation $\varPi_0$ corresponding to the Yoshida lift of $f_0$ and $g_0$, denoted $\varPi(f_0,g_0)$ as in the introduction. The corresponding weights $(k_1,k_2)$ of  $\varPi_0$ --- as given in  equation \eqref{eq:siegelweight} ---  satisfy $(k_1,k_2) \equiv (a,b)\pmod{p-1}\Z^2$. In fact, due to the congruence condition in \ref{lab:yosdet}, we have $b=2 \pmod{p-1}$.

Proposition \ref{P:33} lets us deduce that the automorphic representation $\varPi_0$ is $p$-ordinary. The similitude character of the automorphic representation $\varPi_0$, which is given by $\det(\rho_{f_0})$, equals the $p$-adic character:
\begin{align}
\label{eq:similitude}\chi_\hF \cyc^{k_{f_0}-1}.	
\end{align}
Observe that the  hypothesis \ref{lab:yosdet} which tells us that $k_{f_0} \equiv k_{g_0} \ (\mathrm{mod} \ 2(p-1))$ is stronger than hypothesis \ref{hyp:yoshida2} which only requires $k_{f_0} \equiv k_{g_0} \ (\mathrm{mod} \ 2)$. We need the stronger hypotheses to ensure that the residual representation associated to the Yoshida lift $\varPi_0$ equals
\[\overline{\rho}_{f_0} \oplus \overline{\rho}_{g_0}. \]
These observations along with the hypotheses \ref{lab:irr}, \ref{lab:multfree} and \ref{lab:pdist-gsp4} tell us that Hypotheses \ref{item:globalhyp}, \ref{item:localhypord} and \ref{item:localhypdisgsp4} from Section \ref{sec:pseudocharactersGMA} hold. 

One can consider the local component $\TTT$ determined by (the residual representation associated to) $\varPi_0$, as in Section \ref{subsec:gmas}. The tame level for $\TTT$ corresponds to the tame level of the Yoshida lift of $f_0$ and $g_0$ as in Remark \ref{rem:auxtamelevel}. The congruence class $(a,b)$ modulo $(p-1)\Z^2$ associated to $\TTT$ is determined by the weight \[\left((k_{f_0}+k_{g_0})/2, (k_{f_0}-k_{g_0})/2+2\right)\] of $\Pi_0$. Since we have $k_{f_0} \equiv k_{g_0} \ (\mathrm{mod} \ 2(p-1))$, we have the following equality of congruence classes: \[(a,b) \equiv (k_{f_0},2) \ \mathrm{mod} (p-1)\Z^2.\] We can consider the GMA afforded by Theorem \ref{thm:pseudorepBC}:
\[\left(\begin{array}{cc} M_2(\TTT) & M_2(\mathcal{B}) \\ M_2(\mathcal{C}) & M_2(\TTT) \end{array}\right).\]

Following the notations in Section \ref{subsec:gmas}, we have
\[\overline{\varrho}_2 \cong \overline{\rho}_{f_0}, \qquad \overline{\varrho}_1 \cong \overline{\rho}_{g_0}\]

One can determine the similitude character $\widetilde{\lambda}$, afforded by Lemma \ref{lem:similitudechar}, associated to the local component $\TTT$ by studying the similitude character of the specialization $\varPi_0$, since we only need to identify the finite global unramified-at-$p$ character $\widetilde{\lambda}_{\mathrm{ur}}$.  We have
 \[ \widetilde{\lambda} = \chi_\hF \cyc^{-3} \omega^a \omega^b \kappa_1 \kappa_2.\]

Let $\alpha_p(f_0)$ and $\alpha_p(g_0)$ denote the $p$-adic unit roots of the Hecke polynomial at $p$ for the eigenforms $f_0$ and $g_0$ respectively. Let $\p_{\varPi_0}$ denote the classical specialization of $\TTT$ corresponding to the Yoshida lift $\varPi_0$. Proposition \ref{P:33} lets us deduce that we have an inclusion $\p_{\varPi_0} \supset \JJJ_{\mathrm{red}}$. Let $\eta$ denote the unique minimal prime contained in the classical specialization $\p_{\varPi_0}$. Consider the following natural ring homomorphism:
\begin{align}\label{eq:specializationf0g0}
	\varpi_\OO : \TTT \twoheadrightarrow \TTT/\eta \twoheadrightarrow \TTT/\p_{\varPi_0} \hookrightarrow \OO.
\end{align}

The data afforded by the map $\varPi_0$  can be placed in the setting of Section \ref{sec:selmerclassprod}. Comparing $p$-adic valuations tells us that 
\begin{align}
\varpi_\OO(U_{\cP}) = \alpha_p(f_0), \qquad \varpi_\OO\left(\dfrac{U_{\cQ}}{U_{\cP}}\right) = \alpha_p(g_0).
\end{align}

Consider the two dimensional Galois representations $\rho_{f_0}:G_\Q \rightarrow \Gl_2(\OO)$, attached to the $p$-ordinary forms $f_0$ and $g_0$. Let $L_{f_0}$ and $L_{g_0}$ denote the free $\OO$-modules of rank two . Since $f_0$ is $p$-ordinary and because $\overline{\rho}_{f_0}$ is $p$-distinguished, we have a short exact sequence of free $\OO$-modules that is $G_{\Q_p}$-equivariant:

\begin{align}
0 \rightarrow \Fil^+L_{f_0} \rightarrow L_{f_0} \rightarrow \frac{ L_{f_0}}{\Fil^+ L_{f_0}} \rightarrow 0.
\end{align}

The action of $G_{\Q_p}$ on  the rank one $\OO$-module $\frac{ L_{f_0}}{\Fil^+ L_{f_0}}$ is given by the unramified character $\varpi_\OO \circ \widetilde{\psi}_0:G_{\Q_p} \rightarrow \OO^\times$ (as in equation~\eqref{eq:psi01def}), which sends $\Frob_p$ to $\alpha_p(f_0)$. The action of $G_{\Q_p}$ on  the rank one $\OO$-module $\Fil^+ L_{f_0}$ is given by the ramified character $\varpi_\OO \circ \left(\frac{\widetilde{\lambda}}{\widetilde{\psi}_0}\right):G_{
Q_p} \rightarrow \OO^\times$, that equals the character $\frac{\det(\rho_{f_0})}{\varpi_\OO \circ \widetilde{\psi}_0}$. We consider the following discrete modules over $\OO$:
\begin{align}
	\mathrm{D}_{\OO} &\coloneqq \Hom_{\OO}\left(L_{f_0},L_{g_0}\right)\left(\cyc^{k_2-2}\right)  \otimes_\OO \widehat{\OO}, \qquad \mathrm{D}_{\OO,\Fil^+} \coloneqq \Hom_{\OO}\left(\dfrac{L_{f_0}}{\Fil^+L_{f_0}},L_{g_0}\right)\left(\cyc^{k_2-2}\right)  \otimes_\OO \widehat{\OO}. 
\end{align}

Here, $\widehat\OO \coloneqq \Hom_{\mathrm{cont}}(\OO,\Q_p/\Z_p)$ denotes the Pontryagin dual of $\OO$. As earlier, $\Sigma_0 = S\setminus\{p\}$. The non-primitive Selmer group associated to the Rankin--Selberg product of $f_0$ and $g_0$, alluded in hypothesis \ref{lab:tor} is given below: 
\begin{align}\label{eq:defRS}
\Sel^{\Sigma_0}(\mathrm{D}_{\OO}) \coloneqq \ker \bigg( H^1_{\mathrm{cont}}\left(G_S, \ \mathrm{D}_{\OO}\right) \rightarrow H^1_{\mathrm{cont}}\left(I_p, \ \dfrac{\mathrm{D}_{\OO}}{\mathrm{D}_{\OO,\Fil^+}} \right) \bigg).
\end{align}

\subsection{Proof of Theorem \ref{thm:yoshidafamily}: $p$-adic families of Hecke eigensystems associated to Yoshida lifts}

For our proofs in this section, we will  need to choose an auxiliary $\Gl_2$-tame level $N_2$ divisible by both $N_{\hF}$ and $N_{\hG}$, as given below:
\begin{align}\label{eq:GL2tamelevel}
	N_2 \coloneqq \prod_{\ell \mid N_{\hF}N_{\hG}} \ell^{\max\left\{2,\val_\ell(N_{\hF}) +  2\val_p(\ell-1),\val_\ell(N_{\hG}) +  2\val_p(\ell-1)\right\}}.  
\end{align}
 We will also consider Hida's Hecke algebras $\bbT_{\hF,N_2}$ and $\bbT_{\hG,N_2}$, defined analogous to $\mathbb{T}_{\hF,N_{\hF}}$ and $\mathbb{T}_{\hG,N_{\hG}}$, but with tame level $N_2$ instead. One could have primes $\ell \neq p$ dividing $N_{\hG}$ but not $N_{\hF}$.	 Such primes could contribute to introducing new irreducible components for the ring $\bbT_{\hF,N_2}$ which were not present in the ring $\mathbb{T}_{\hF,N_{\hF}}$. For example, one could have primes $\ell \neq p$ participating in Ribet's level raising congruences (even if we work with squarefree levels). For the tame level chosen in equation (\ref{eq:GL2tamelevel}), one could also twist the Hida family $\hF$ by characters with conductor $\ell$ and order $p$, if a prime $\ell$ dividing $N_{\hG}$ is also congruent to $1$ modulo $p$. For these reasons, one could truly have new Yoshida lifts of Hida families appearing in the decomposition given in equation (\ref{eq:GSp4decomphecke}).

{\theoremone*}

We will first show that the minimal prime $\eta$ contains the reducibility ideal $\JJJ_{\mathrm{red}}$. That is, we will show $i_\eta(\BBB)=(0)$ and $i_\eta(\CCC)=(0)$. By Theorem \ref{thm:pseudorepBC}(\ref{item:pseudo3}), it is enough to show $i_\eta(\BBB)=(0)$.  Proposition \ref{prop:inclusion} lets us deduce the following surjection of $\OO$-modules:
\begin{align} \label{eq:surjselmer}
\Sel^{\Sigma_0}(\mathrm{D}_{\OO})^\vee \twoheadrightarrow 	\BBB \otimes_{\varpi_\OO} \OO.  
\end{align}

Note that equation (\ref{eq:specializationf0g0}) tells us that $\BBB \otimes_{\varpi_\OO} \OO$ is isomorphic to $ (\BBB \otimes_\TTT \TTT/\eta) \ \otimes_{\varpi_\OO} \OO$.\\
 
We will use properties of (zeroth) Fitting ideals.  Suppose, for the sake of contradiction, that the fractional ideal $i_\eta(\BBB)$ over the domain $\TTT/\eta$ is non-zero. Then, the $\TTT/\eta$-module $i_\eta(\BBB)$ must be faithful. By \cite[Proposition 20.7]{MR1322960}, we have $\Fitt_{\TTT/\eta}(i_\eta(\BBB)) = (0)$. The $\TTT/\eta$-module surjection $\BBB \otimes_\TTT \TTT/\eta \twoheadrightarrow  i_\eta(\BBB)$ then lets us conclude that $\Fitt_{\TTT/\eta}(\BBB \otimes_\TTT \TTT/\eta) =(0)$. Using properties of Fitting ideals  under base change \cite[Corollary 20.5]{MR1322960}, we can then conclude that $\Fitt_{\OO}\left((\BBB \otimes_\TTT \TTT/\eta) \ \otimes_{\varpi_\OO} \OO\right) = (0)$. The surjection in equation (\ref{eq:surjselmer}) then lets us conclude that $\Fitt_{\OO}(\Sel^{\Sigma_0}(\mathrm{D}_{\OO})^\vee)=(0)$. By hypothesis \ref{lab:tor}, the $\OO$-module $\Sel^{\Sigma_0}(\mathrm{D}_{\OO})^\vee$ must be finitely generated and torsion. Since $\OO$ is a principal ideal domain, this is impossible, unless $i_\eta(\BBB)$ equals $(0)$.  What we have thus shown is that the minimal prime ideal $\eta$ contains the reducibility ideal. \\

  Applying Theorem \ref{thm:pseudorepBC}(\ref{item:pseudo6}), we can conclude that the induced  pseudocharacter $G_S \rightarrow \TTT/\eta$ attached to the branch $\TTT/\eta$ is a sum of two $2$-dimensional pseudocharacters, say $\mathrm{Trace}(\varrho_1) \oplus \mathrm{Trace}(\varrho_2)$.  What we have implicitly used here is that the residual $G_S$-representations $\overline{\rho}_\hF$ and $\overline{\rho}_\hG$ are both irreducible. As a result, theorems of Nyssen \cite{MR1411348}  and Rouquier \cite{MR1378546} tell us that any pseudocharacter valued in $\TTT/\eta$, lifting $\overline{\rho}_\hF$ or $\overline{\rho}_\hG$, comes from the trace of a $2$-dimensional $G_S$ representation $\varrho_i: G_\Q \rightarrow \Gl_2(\TTT/\eta)$, for $i \in \{1,2\}$. Note that this can also be directly observed from the structure of the GMA as afforded by Theorem \ref{thm:pseudorepBC} since $i_\eta(\BBB) = i_\eta(\CCC)=0$. \\ 
  
We will now prove that the branch $\eta$ corresponds to the $p$-adic family of Hecke eigensystems associated to the space of Yoshida lifts of $\hF$, $\hG$, thus proving both Theorem \ref{thm:yoshidafamily}\ref{thmcond:yoshidafamily} and \ref{thmcond:unique}. For every classical prime $\p_{k_1,k_2}$ in $\frakX^{\rm temp}_{\charcomp}\left(\TTT/\eta\right)$, we fix a ring homomorphism into $\overline{\Q}_p$: \[\varphi_{\p_{k_1,k_2}}: \TTT/\eta \rightarrow \dfrac{\TTT/\eta }{\p_{k_1,k_2}} \hookrightarrow \overline{\Q}_p. \]

This allows us to consider the Galois representations: \[\varphi_{\p_{k_1,k_2}} \circ\varrho_1 : G_\Q \rightarrow \Gl_2(\overline{\Q}_p), \quad \varphi_{\p_{k_1,k_2}} \circ\varrho_2 : G_\Q \rightarrow \Gl_2(\overline{\Q}_p)\]

\begin{enumerate}
\item The hypothesis \ref{lab:irr} lets us conclude that the residual representations $\overline{\varphi_{\p_{k_1,k_2}} \circ\varrho_1(\cyc^{2}\omega^{-b}\kappa_2^{-1})}$ and $\overline{\varphi_{\p_{k_1,k_2}} \circ\varrho_2}$ are absolutely irreducible. 

\item Theorem \ref{thm:pseudorepBC}(\ref{item:pseudo7}) and hypothesis \ref{lab:pdist-gsp4} let us conclude that the semi-simplifications of the local $p$-ordinary residual representations $\overline{\varphi_{\p_{k_1,k_2}} \circ\varrho_1(\cyc^{2}\omega^{-b}\kappa_2^{-1})} \vert_{G_{\Q_p}}$ and $\overline{\varphi_{\p_{k_1,k_2}} \circ\varrho_2}\vert_{G_{\Q_p}}$, are both sums of two distinct characters. 

\item Since our local component $\TTT$ passes through the Yoshida lift of $f_0$ and $g_0$, the residual global representations $\overline{\varphi_{\p_{k_1,k_2}} \circ\varrho_1(\cyc^{2}\omega^{-b}\kappa_2^{-1})}$ and $\overline{\varphi_{\p_{k_1,k_2}} \circ\varrho_2}$ \textit{come from} modular forms $g_0$ and $f_0$ respectively. 

\item Theorem \ref{thm:pseudorepBC}(\ref{item:pseudo7}) lets us conclude that the restrictions $\varphi_{\p_{k_1,k_2}} \circ \varrho_1(\cyc^{2}\omega^{-b}\kappa_2^{-1})\vert_{I_p}$ and $\varphi_{\p_{k_1,k_2}} \circ \varrho_2\vert_{I_p}$ of the characteristic zero representations to the local inertia group $I_p$ at $p$ have a unique $1$-dimensional quotient on which $I_p$ acts trivially.  The image of $\Frob_p$ on these unique unramified quotients is given by $\varphi_{\p_{k_1,k_2}} \circ \widetilde{\psi}_1(\Frob_p)$  and $\varphi_{\p_{k_1,k_2}} \circ \widetilde{\psi}_0(\Frob_p)$ respectively. 

\item\label{pt:five} The following $G_\Q$-characters $\dfrac{\det(\varphi_{\p_{k_1,k_2}} \circ \varrho_1(\cyc^{2}\omega^{-b}\kappa_2^{-1}))}{\cyc^{k_1-k_2+1}}$ and $\dfrac{\det(\varphi_{\p_{k_1,k_2}} \circ \varrho_2)}{\cyc^{k_1+k_2-3}}$ are  unramified outside primes dividing $N$, and are valued in the units of the ring $\TTT/\eta$. Consequently, both of them  are of finite order. 
\end{enumerate}

All the hypotheses of \cite[Theorem]{MR1928993} are satisfied.  Consequently, the $G_S$-representations \[\varphi_{\p_{k_1,k_2}} \circ\varrho_1(\cyc^{2}\omega^{-b}\kappa_2^{-1}), \quad \varphi_{\p_{k_1,k_2}} \circ\varrho_2\] are modular.  One can deduce that the forms are $p$-ordinary using results of Mazur--Wiles since their associated Galois representation at $p$ is $p$-ordinary (see \cite{MR1658011}). One can use results of Carayol \cite[Proposition 2]{MR1046750} and Livn\'{e} \cite[Theorem 2]{MR0987567} to further conclude that the tame level of these modular representations is divisible by $N_2$. Therefore, the $\Gl_2$-Hecke eigensystems 
{\small \[\left\{\varphi_{\p_{k_1,k_2}} \circ \widetilde{\psi}_1(\Frob_p), \ \left\{\mathrm{Trace}\left( \varphi_{\p_{k_1,k_2}} \circ\varrho_1(\cyc^{2}\omega^{-b}\kappa_2^{-1})(\Frob_\ell)\right)\right\}_{l \nmid Np} \right\}, \left\{\varphi_{\p_{k_1,k_2}} \circ \widetilde{\psi}_0(\Frob_p), \ \left\{\mathrm{Trace}\left(  \varphi_{\p_{k_1,k_2}} \circ\varrho_2(\Frob_\ell)\right)\right\}_{l \nmid Np} \right\}\] }

determine ring homomorphisms $ \bbT_{\hG,N_2} \rightarrow \dfrac{\TTT/\eta	}{\p_{k_1,k_2}}$ and $ \bbT_{\hF,N_2} \rightarrow \dfrac{\TTT/\eta	}{\p_{k_1,k_2}}$. See equation \eqref{eq:commdiaghidafamilies} and the explanations provided below it. The density of classical primes $\p_{k_1,k_2}$ of $\frakX^{\rm temp}_{\charcomp}\left(\TTT/\eta\right)$ due to  Lemma \ref{lem:density} tells us that the natural diagonal map $\TTT/\eta \hookrightarrow \prod  \dfrac{\TTT/\eta	}{\p_{k_1,k_2}}$ is an inclusion.  Collecting all of these ring homomorphisms afforded by the Hecke eigensystems
\[\left\{\widetilde{\psi}_1(\Frob_p), \ \left\{\mathrm{Trace}\left( \varrho_1(\cyc^{2}\omega^{-b}\kappa_2^{-1})(\Frob_\ell)\right)\right\}_{l \nmid Np} \right\}, \left\{\widetilde{\psi}_0(\Frob_p), \ \left\{\mathrm{Trace}\left(  \varrho_2(\Frob_\ell)\right)\right\}_{l \nmid Np} \right\},\] we get
 natural ring homomorphisms $i_{\hG,\eta}:\bbT_{\hG,N_2} \rightarrow  \TTT/\eta$ and $i_{\hF,\eta}:\bbT_{\hF,N_2} \rightarrow \TTT/\eta$ given below:

\begin{align}\label{eq:commdiaghidafamilies}
\xymatrix @R=.3pc{
\bbT_{\hG,N_2} \ar[r] \ar@{-->}[ddd]  \ar[rddd]&  \prod  \dfrac{\TTT/\eta	}{\p_{k_1,k_2}}  \\ \\ \\
 \mathbb{T}_{\hG,N_{\hG}}  \ar@{-->}[r] & \displaystyle   \TTT/\eta  \ar@{^{(}->}[uuu]  \\ 
 U_p  \ar@{|->}[r] & \widetilde{\psi}_1(\Frob_p), \\ 
 T_\ell \ar@{|->}[r]& \mathrm{Trace}\left( \varrho_1(\cyc^{2}\omega^{-b}\kappa_2^{-1})(\Frob_\ell)\right). 
}
 \qquad \qquad \qquad 
\xymatrix@R=.3pc{
\bbT_{\hF,N_2} \ar[r] \ar@{-->}[ddd] \ar[rddd] &  \prod  \dfrac{\TTT/\eta	}{\p_{k_1,k_2}}  \\ \\ \\ 
 \mathbb{T}_{\hF,N_{\hF}} \ar@{-->}[r] & \displaystyle \TTT/\eta  \ar@{^{(}->}[uuu] \\
  U_p  \ar@{|->}[r] & \widetilde{\psi}_0(\Frob_p), \\ 
 T_\ell \ar@{|->}[r]& \mathrm{Trace}\left( \varrho_2(\Frob_\ell)\right). 
}
\end{align}
 
Since $\TTT/\eta$ is a domain, the morphisms  $i_{\hG,\eta}$ and  $i_{\hF,\eta}$ must factor through the unique minimal prime ideals $\eta_\hG$ and $\eta_\hF$, in $\bbT_{\hG,N_2}$ and $\bbT_{\hF,N_2}$ respectively, passing through $g_0$ and $f_0$ respectively. Uniqueness of Hida families (passing through $g_0$ and $f_0$ respectively) then lets us conclude that these ring maps must naturally factor through $\mathbb{T}_{\hG,N_2}$ and $\mathbb{T}_{\hF,N_2}$ respectively. Let $\mathbb{T}_{\hF,\eta_\hF}$ and $\mathbb{T}_{\hG,\eta_\hG}$ denote  the quotient rings $\mathbb{T}_{\hF,N_{\hF}}/\eta_\hF$ and $\mathbb{T}_{\hG,N_{\hG}}/\eta_\hG$ respectively. Since the integral closure of $\Z_p$ inside $\mathbb{T}_{\hF,\eta_\hF}$ and $\mathbb{T}_{\hG,\eta_\hG}$  is $\OO$, the completed  tensor product $\mathbb{T}_{\hF,\eta_\hF} \hotimes \mathbb{T}_{\hG,\eta_\hG}$ over $\OO$ is a domain. To see this, one may identify the completed tensor product $\mathbb{T}_{\hF,\eta_\hF} \hotimes \mathbb{T}_{\hG,\eta_\hG}$ with the tensor product $\TTT_{\hF,\eta_\hF}[\![x_\hG]\!] \otimes_{\OO[\![x_\hG]\!]} \mathbb{T}_{\hG,\eta_\hG}$ and embed this tensor product into a tensor product of two linearly disjoint fields $\Frac(\TTT_{\hF,\eta_\hF}[\![x_\hG]\!]) \otimes_{\Frac(\OO[\![x_\hG]\!])} \Frac(\mathbb{T}_{\hG,\eta_\hG})$.  See \cite[Corollary to Proposition 1, Chapter V, \S 17]{MR1994218}. We get an induced ring map of integral domains: 

\[\iota:\mathbb{T}_{\hF,\eta_\hF} \hotimes \mathbb{T}_{\hG,\eta_\hG} \hookrightarrow \TTT/\eta.\]

We need to now show that the ring homomorphism $\iota$ given above is a surjection. One can directly check the assignments of 
Hecke operators as given in equation (\ref{eq:Upassignments}) of Proposition \ref{prop:LRYoshidafam}. Observe that the commutative diagrams in equation (\ref{eq:commdiaghidafamilies}) let us conclude the following isomorphism of Galois representations:
\[\iota \circ \rho_\hG \cong \varrho_1 (\cyc^{2}\omega^{-b}\kappa_2^{-1}), \qquad \iota\circ \rho_\hF \cong \varrho_2.\]
Moreover, observation \ref{pt:five} given above lets us obtain the following equality of characters whose image is pro-$p$:
\begin{align*}
 \iota \circ \kappa_\hG = \kappa_1 \kappa_2^{-1}\cyc^{2} \omega^{-b}, \qquad \iota \circ \kappa_\hF = \kappa_1 \kappa_2 \cyc^{-2} \omega^{b}.
\end{align*}
We can now conclude that the minimal prime ideal $\eta$ corresponds to the $p$-adic family of Hecke eigensystems associated to the space of Yoshida lifts of the Hida families $\hF$ and $\hG$. This is because the Galois representation $\varrho_\eta$ is isomorphic to  $\varrho_1 \oplus \varrho_2$, and by combining  our observations above, we see that
\begin{align}
\varrho_1 \cong  \iota \circ \rho_\hG (\sqrt{\kappa_\hF/\kappa_\hG}), \qquad \varrho_2 \cong \iota \circ \rho_\hF.	
\end{align}
Since $p$ is odd, one can take square-roots of elements in $1+\mmm_{\Lambda_\OO}$. These observations let us conclude that we have the following isomorphism of $\Lambda$-algebras under $\iota$:
\begin{align}\label{eq:assignmentslambda}
\notag \OO[\![x_\hF,x_\hG]\!]	&\cong \OO[\![x_1,x_2]\!], \\
\notag x_\hF +1 &\mapsto (x_1 + 1)(x_2+1)(\cyc^{-2}\omega^{b}), \\
\notag x_\hG + 1 &\mapsto (x_1 + 1)(x_2+1)^{-1}(\cyc^{2}\omega^{-b}), \\
\notag \sqrt{(x_\hF+1)(x_\hG+1)} & \mapsfrom x_1 +1, \\
\sqrt{(x_\hF+1)(x_\hG+1)^{-1}} \cyc^{-2} \omega^{b} & \mapsfrom x_2 +1.
\end{align}

Let $\iota_\eta:\TTT \rightarrow \TTT/\eta$ denote the natural reduction map. To prove that the map $\iota$ is surjective, we need to show that for all primes $\ell \nmid Np$, the image of the Hecke operators $i_\eta(T_\ell), i_\eta(R_\ell)$ and $i_\eta(S_\ell)$, along with $\iota_\eta(U_{\cP}), \iota_\eta(U_{\cQ})$ lie inside $\mathrm{Image}(\iota)$. First note that the characteristic polynomials for $\Frob_\ell$ for $\rho_\hF$ and $\rho_\hG$ are respectively equal to 
\begin{align} \label{eq:charpolygl2}
x^2 - T_\ell(\hF)x + \psi_\hF\omega^{k_{f_0}-1}(\Frob_\ell) \langle \ell \rangle_\hF, \qquad 	x^2 - T_\ell(\hG)x + \psi_\hG\omega^{k_{g_0}-1}(\Frob_\ell) \langle \ell \rangle_\hG.
\end{align}

Here, $T_\ell(\hF),  \langle  \ell \rangle_\hF$ and $T_\ell(\hG), \langle  \ell \rangle_\hG$ are the image of the Hecke operator and (pro $p$-part of the) diamond operator at $\ell$, associated to $\hF$ and $\hG$ respectively. By using the construction of Yoshida lifts as in Section \ref{sec:yoshidalifts} along with a density argument, one can deduce that the global similitude character, as given in Lemma \ref{lem:similitudechar}, equals $\det(\rho_\hF)$. Using equation (\ref{eq:tl0}), one  can conclude that $i_\eta(S_\ell)$ lies in the image of $\iota$. Characteristic polynomials of matrices are invariant under conjugation. Therefore, for all primes $\ell \nmid Np$, a comparison of the $\GSp_4$-Hecke polynomial at $l$ given in equation (\ref{eq:abstractdegree4poly}) and the product of the characteristic polynomials of $\rho_\hG (\sqrt{\kappa_\hF/\kappa_\hG})(\Frob_\ell)$ and $\rho_\hF(\Frob_\ell)$ will let us conclude that the Hecke operators $i_\eta(T_\ell)$ and $i_\eta(R_\ell)$ lie in the image of $\iota$. 

One can use Theorem \ref{thm:pseudorepBC}(\ref{item:pseudo7}) to compare the image of $\Frob_p$ under the unique $1$-dimensional unramified quotients of the Galois representations $\rho_\eta \mid_{G_{\Q_p}}$ and $\rho_\eta(\kappa_2^{-1}\cyc^{2}\omega^{-b}) \mid_{G_{\Q_p}}$ respectively. This allows us to deduce the following equality, which in turn lets us conclude that $\iota_\eta(U_{\cP}), \iota_\eta(U_{\cQ})$ also lie in the image of $\iota$.  
\begin{align}
	\iota(U_p(\hF)) = \iota_\eta(U_{\cP}), \qquad \iota(U_p(\hG)) = \iota_\eta\left(\dfrac{U_{\cQ}}{U_{\cP}}\right).
\end{align}

To see that the ring homomorphism $\iota$ given above is an injection, notice that the kernel of the ring homomorphism must be a minimal prime ideal (and hence equal to $(0)$), since $\iota$ is a homomorphism of rings with the same Krull dimension. This completes the proof.

\begin{proposition} \label{prop:LRYoshidafam}
Suppose that one of the following hypotheses holds:
\begin{enumerate}[style=sameline,  align=left,label=(\scshape{ODDL}), ref=\scshape{ODDL},partopsep=0pt,parsep=0pt]
  \item\label{lab:oddl} The levels $N_f$ and $N_g$ are odd. 
\end{enumerate} 
\begin{enumerate}[style=sameline, align=left,label=(\scshape{LR}), ref=\scshape{LR},partopsep=0pt,parsep=0pt]
  \item\label{lab:hypLR} The tame levels $N_{\hF}$ and $N_{\hG}$ are squarefree. The Nebentypus of the eigenforms $f_0$ and $g_0$ are trivial. Furthermore, for all primes $\ell$ dividing $\mathrm{gcd}(N_{\hF},N_{\hG})$, the Atkin--Lehner eigenvalues at $\ell$ for both $f_0$ and $g_0$ coincide. 
\end{enumerate}

Then, the hypothesis \ref{lab:exist} holds. That is, there exists a $p$-adic family of Hecke eigensystems associated to the space of Yoshida lifts of $\hF$ and $\hG$.
\end{proposition}

\begin{proof}

We illustrate the proof under the hypothesis \ref{lab:hypLR}. The proof works similarly under hypothesis \ref{lab:oddl}. See also Remark \ref{rem:auxtamelevel}. As indicated in the introduction, the key point is to show that the tame level is bounded as one varies over the classical specializations of $\hF$ and $\hG$. \\

Consider the following subsets of classical specializations of $\II_{\hF}$ and $\II_{\hG}$ respectively: 

\begin{multline*}
 S_{f_0}  \coloneqq \{\mathfrak{P} \in \mathrm{Spec}_{\mathrm{ht}=1}\II_{\hF},  \text{ where $\mathfrak{P}$ is a classical height one prime in $\II_{\hF}$ with weight } \\   k(\mathfrak{P}) \equiv \mathrm{wt}(f_0) \ \mathrm{mod} \ 2( p-1), \ k(\mathfrak{P}) \geq 2 \}.
 \end{multline*}
 \begin{multline*}
S_{g_0} \coloneqq \{\mathfrak{P}  \in \mathrm{Spec}_{\mathrm{ht}=1}\II_{\hG},  \text{ where $\mathfrak{P}$ is a classical height one prime in $\II_{\hG}$ with weight }  \\ k(\mathfrak{P}) \equiv \mathrm{wt}(g_0) \ \mathrm{mod} \ 2( p-1), \ k(\mathfrak{P}) \geq 2\}.
\end{multline*}

Since the $\Gl_2$ forms $f_0$ and $g_0$ have levels $\Gamma_0(N_{\hF})$ and $\Gamma_0(N_{\hG})$ respectively, the congruence condition modulo $p-1$ ensures that all the classical specializations of $\II_{\hF}$ and $\II_{\hG}$  belonging respectively to $S_{f_0}$ and $S_{g_0}$ also have levels $\Gamma_0(N_{\hF})$ and $\Gamma_0(N_{\hG})$ respectively. The Atkin--Lehner involution preserves the $p$-integrality of the Fourier coefficients and hence also preserves congruences modulo $p$. See, for example, \cite[Theorem A1]{MR2501296}. As a result, the eigenvalues of the Atkin--Lehner involution (which are either $+1$ or $-1$) remain constant in a $\Gl_2$ Hida family, as the Fourier coefficients are congruent modulo $p$. Consequently, all the classical specializations $f_\mathfrak{P}$ and $g_{\mathfrak{P}'}$ of $\II_{\hF}$ and $\II_{\hG}$ also satisfy hypothesis \ref{lab:hypLR} of the Proposition. For the $\Gamma_0(N_\hF)$ and $\Gamma_0(N_{\hG})$ forms $f_\mathfrak{P}$ and $g_{\mathfrak{P}'}$, note that hypotheses \ref{hyp:yoshida2} and \ref{hyp:yoshida3} of section \ref{sec:yoshidalifts} are guaranteed by hypothesis \ref{lab:hypLR} of the proposition. Similarly, hypothesis \ref{hyp:yoshida4} of section \ref{sec:yoshidalifts} is guaranteed by the congruence conditions modulo $p-1$ of the sets $S_{f_0}$ and $S_{g_0}$. Therefore, under the hypothesis \ref{lab:hypLR} and subject to Hypotheses \ref{hyp:yoshida1} of section \ref{sec:yoshidalifts}, one has a classical Yoshida  lift of $f_\mathfrak{P}$ and $g_{\mathfrak{P}'}$ with Siegel congruence level $\Gamma^{(2)}_0\left(\mathrm{lcm}(N_{\hF},N_{\hG})\right)$. See works of the first author with Namikawa \cite{MR3623733} and Saha--Schmidt \cite[Proposition 3.1]{MR3092267}. Let $\TTT$ denote Hida's  $p$-ordinary Hecke algebra with tame level $\mathrm{lcm}(N_{\hF},N_{\hG})$ for $\GSp_4$ as in Section \ref{subsec:gmas}. The Hecke eigensystem afforded by the Yoshida lift of $f_{\mathfrak{P}}$ and $g_{\mathfrak{P}'}$ determines a ring homomorphism $\phi_{k(\mathfrak{P}),k(\mathfrak{P'})}:\TTT \rightarrow \overline{\Q}_p$ for the tempered weight $(w_1,w_2)=\left(\dfrac{k(\mathfrak{P})+k(\mathfrak{P}')}{2}, \dfrac{k(\mathfrak{P})-k(\mathfrak{P}')}{2}+2\right)$.  

Let $\mathfrak{P}$ and $\mathfrak{P}'$ belong to $S_{f_0}$ and $S_{g_0}$ respectively. Let \[\phi_{\mathfrak{P}}: \II_{\hF} \hookrightarrow (\II_{\hF})_{\mathfrak{P}} \rightarrow \overline{\Q}_p, \qquad \phi_{\mathfrak{P}'}: \II_{\hG} \hookrightarrow (\II_{\hG})_{\mathfrak{P}'} \rightarrow \overline{\Q}_p\] 
denote the corresponding ring homomorphisms determined by the corresponding $\Gl_2$ Hecke eigensystems. Let $\p$ and $\p'$ denote the prime ideals of $\OO[\![x_\hF]\!]$ and $\OO[\![x_\hG]\!]$ lying below $\mathfrak{P}$ and $\mathfrak{P}'$ respectively. Let $\mathfrak{Q}$ denote the prime ideal corresponding to the kernel of the induced ring homomorphism $\II_{\hF}[\![x_\hG]\!]  \otimes_{\OO[\![x_\hG]\!]} \II_{\hG} \rightarrow \overline{\Q}_p$. Let $\mathfrak{q}$ denote the prime ideal of $\II_{\hF}[\![x_\hG]\!]$ lying below $\mathfrak{Q}$. We have the following commutative digram of ring homomorphisms:
 
\begin{align*}
\xymatrix{
\left(\II_{\hF}[\![x_\hG]\!]\right)_{\mathfrak{q}}\ar@{^{(}->}[r]& \left(\II_{\hF}[\![x_\hG]\!] \right)_{\mathfrak{q}} \otimes_{(\OO[\![x_\hG]\!])_{\mathfrak{p}'}} \left(\II_{\hG}\right)_{\mathfrak{P}'} \\\ 
\left(\OO[\![x_\hG]\!]\right)_{\mathfrak{p}'} \ar@{^{(}->}[u] \ar@{^{(}->}[r]& \left(\II_{\hG}\right)_{\mathfrak{P}'} \ar@{^{(}->}[u]
}	
\end{align*}

Linear disjointness lets us conclude that  $\left(\II_{\hF}[\![x_\hG]\!] \right)_{\mathfrak{q}} \otimes_{(\OO[\![x_\hG]\!])_{\mathfrak{p}'}} \left(\II_{\hG}\right)_{\mathfrak{P}'}$ is a domain. Since $\mathfrak{P}'$ corresponds to a classical prime of $\II_{\hG}$ with weight $\geq 2$, the extension $(\OO[\![x_\hG]\!])_{\mathfrak{p}'} \hookrightarrow \left(\II_{\hG}\right)_{\mathfrak{P}'}$ is \'etale. See \cite[Corollary 1.4]{MR848685}.  \'Etale morphisms are preserved under base change (\cite[Chapter 4, Proposition 3.22]{MR1917232}). Since the height two prime ideal $\mathfrak{q}$ contains a classical height one prime of $\OO[\![x_\hG]\!]$, note that \cite[Lemma 4.16]{MR3919711} now forces $(\II_{\hF}[\![x_\hG]\!])_{\mathfrak{q}}$ to be a regular local ring with Krull dimension $2$. Since the morphism $\left(\II_{\hF}[\![x_\hG]\!]\right)_{\mathfrak{q}} \rightarrow  \left(\II_{\hF}[\![x_\hG]\!] \right)_{\mathfrak{q}} \otimes_{(\OO[\![x_\hG]\!])_{\mathfrak{p}'}} \left(\II_{\hG}\right)_{\mathfrak{P}'}$ is \'etale, the ring $\left(\II_{\hF}[\![x_\hG]\!] \right)_{\mathfrak{q}} \otimes_{(\OO[\![x_\hG]\!])_{\mathfrak{p}'}} \left(\II_{\hG}\right)_{\mathfrak{P}'}$ is regular too (Corollary 3.24 in Chapter 4 of \cite{MR1917232}), and therefore integrally closed (\cite[\href{https://stacks.math.columbia.edu/tag/0567}{Tag 0567}]{stacks-project}). Thus, we get a ring homomorphism:
\[\II_{\hF,\hG} \hookrightarrow \left(\II_{\hF}[\![x_\hG]\!] \right)_{\mathfrak{q}} \otimes_{(\OO[\![x_\hG]\!])_{\mathfrak{p}'}} \left(\II_{\hG}\right)_{\mathfrak{P}'} \rightarrow \overline{\Q}_p.\] 

Let $\mathfrak{Q}_{k(\mathfrak{P}),k(\mathfrak{P}')}$ denote the kernel of the ring homomorphism $\II_{\hF,\hG} \rightarrow \overline{\Q}_p$. Lemma \ref{lem:density} (rather its proof) tells us that the set of such $\mathfrak{Q}_{k(\mathfrak{P}),k(\mathfrak{P}')}$'s satisfying (i) $\mathfrak{P} \in S_{f_0}$, (ii) $\mathfrak{P}' \in S_{g_0}$ and (iii) $k(\mathfrak{P}) > k(\mathfrak{P}')$, is a dense subset of $\Spec(\II_{\hF,\hG})$. Combining all our observations provides us the following commutative diagram:

	\begin{align*}
\xymatrix @R=.3pc{
\TTT \ar@{-->}[r]   \ar[rddd]&  \II_{\hF,\hG} \ar@{^{(}->}[ddd] \\ \\ \\
  & \displaystyle   \prod_{\substack{\mathfrak{P} \in S_{f_0} \\ \mathfrak{P}' \in S_{g_0} \\ k(\mathfrak{P})> k(\mathfrak{P}')}} \dfrac{\II_{\hF,\hG}}{\mathfrak{Q}_{k(\mathfrak{P}),k(\mathfrak{P}')}}
  }
  \end{align*}     
 determined by the following consistent assignments:
\begin{align}\label{eq:Upassignments}
\notag U_{\cP}  &\mapsto U_{p}(\hF), \qquad 
& U_{\cQ} &\mapsto U_p(\hF)U_p(\hG), \\
\notag T_\ell  &\mapsto   T_{\ell}(\hF)+T_\ell(\hG) \sqrt{\dfrac{\langle \ell \rangle_\hF}{\langle \ell \rangle_\hG}}, \qquad & \ell R_\ell &\mapsto \dfrac{1}{\ell^3}(\ell^3-\ell) \chi_\hF \omega^{k_{f_0}-1}(\Frob_\ell)\langle \ell \rangle_\hF + T_\ell(\hF)T_\ell(\hG) \sqrt{\dfrac{\langle \ell \rangle_\hF}{\langle \ell \rangle_\hG}},\\
\ell^3S_\ell  &\mapsto  \chi_\hF \omega^{k_{f_0}-1}(\Frob_\ell)\langle \ell \rangle_\hF,
\end{align}
 along with the mapping  $\OO[\![x_1,x_2]\!] \xrightarrow{\cong} \OO[\![x_\hF,x_\hG]\!]$ of Iwasawa algebras as given in equation (\ref{eq:assignmentslambda}).  We thus obtain the desired ring homomorphism $\TTT \rightarrow \II_{\hF,\hG}$ of the proposition interpolating the Hecke eigensystems associated to a $p$-adic family of Yoshida lifts of two Hida families $\hF$ and $\hG$.  
\end{proof}

\subsection{Proof of Theorem \ref{thm:mainconj2}: the congruence ideal and the Selmer group} \label{sec:proofthm2}

{\theoremtwo*}
\begin{proof}
 % \ref{thm:yoshidafamily}\ref{thmcond:otheryoshida}. Therefore, to avoid repetition and to keep our discussions shorter, we simply focus on the proof of Theorem \ref{thm:yoshidafamily}\ref{thmcond:yoshidafamily} and leave the details of the proof of Theorem \ref{thm:yoshidafamily}\ref{thmcond:otheryoshida} to the interested reader to complete. The interested reader can prove Theorem \ref{thm:yoshidafamily}\ref{thmcond:otheryoshida} using an argument by contradiction: if the Galois representation associated to an irreducible component $\eta'$ in $\TTT^{\rm st}$ turns out to be reducible, one obtains a semi-simple Galois representation for all tempered classical specializations  of $\TTT^{\rm st}$ as in Theorem \ref{thm:pseudorepBC}; one arrives at a contradiction by repeating the arguments involving modularity results of Skinner--Wiles deducing that the irreducible component $\eta'$ (belonging to $\TTT^{\rm st}$) must also correspond to Yoshida lifts of Hida families. \\ 

The arguments now required to prove  Theorem \ref{thm:yoshidafamily}\ref{thmcond:yoshidafamily} and \ref{thmcond:unique} are nearly identical to the proof of part (\ref{thmcond:otheryoshida}) of this theorem, so we leave the details for the reader to fill in.  Let $\eta'$ be a minimal prime of $\bbT^{\rm st}$. If the Galois representation associated to an irreducible component $\eta'$ in $\TTT^{\rm st}$ turns out to be reducible, one obtains a semi-simple Galois representation for all tempered classical specializations  of $\TTT^{\rm st}$ as in Theorem \ref{thm:pseudorepBC}; one arrives at a contradiction by repeating the arguments involving modularity results of Skinner--Wiles (and Diamond) deducing that the irreducible component $\eta'$ (belonging to $\TTT^{\rm st}$) must also correspond to Yoshida lifts of Hida families.

%Each tempered classical specialization of the Galois representation of $\rho_{\eta'}:G_S\to \GSp_4(\Frac(\bbT/{\eta'}))$ as in Theorem \ref{thm:pseudorepBC} is the Galois representation associated to a cuspidal tempered automorphic representations on $\GSp(4)$. If the Galois representation $\rho_{\eta'}$ associated to an irreducible component $\eta'$ in $\TTT^{\rm st}$ is reducible, then we can deduce that the semi-simplification $\rho_{\eta'}^{ss}
%=\rho_{1}\oplus \rho_2$ is a direct sum of two Galois representations from Arthur's classification. By modularity results of Taylor, Skinner--Wiles, Diamond etc, we see that $\rho_1$ and $\rho_2$ are Galois representations associated with Hida families of elliptic modular forms, which shows that $\rho_{\eta'}$ comes from Yoshida lifts of Hida families.

To prove part (\ref{part:divinequality}) of this theorem, we reintroduce all the objects involved in our setup. We have $\II_{\hF,\hG}[G_S]$-modules: 
\begin{align*}
\SLLL_{\II_{\hF,\hG}}  \coloneqq \Hom_{\II_{\hF,\hG}}\left(L_\hF \otimes_{\II_{\hF}} \II_{\hF,\hG}, L_\hG \otimes_{\II_{\hG}} \II_{\hF,\hG}\right) \otimes_{\II_{\hF,\hG}} \II_{\hF,\hG}(\sqrt{\kappa_\hF \kappa_\hG^{-1}}), \qquad \DDD_{\II_{\hF,\hG}} \coloneqq \SLLL_{\II_{\hF,\hG}}  \otimes_{\II_{\hF,\hG}} \widehat{\II_{\hF,\hG}}.
\end{align*}
We also have $\II_{\hF,\hG}[\Gal{\overline{\Q}_p}{\Q_p}]$-modules:
\begin{align*}
\SLLL_{\II_{\hF,\hG},\Fil^+}  \coloneqq \Hom_{\II_{\hF,\hG}}\left(\dfrac{L_\hF}{\Fil^+L_\hF} \otimes_{\II_{\hF}} \II_{\hF,\hG}, L_\hG \otimes_{\II_{\hG}} \II_{\hF,\hG}\right) \otimes_{\II_{\hF,\hG}} \II_{\hF,\hG}(\sqrt{\kappa_\hF \kappa_\hG^{-1}}), \qquad \DDD_{\II_{\hF,\hG},\Fil^+} \coloneqq \SLLL_{\II_{\hF,\hG},\Fil^+}  \otimes_{\II_{\hF,\hG}} \widehat{\II_{\hF,\hG}}.
\end{align*}
Here, we have used the fact that there is a natural short exact sequence of free $\II_{\hF}$-modules that is $\Gal{\overline{\Q}_p}{\Q_p}$-equivariant, attached to the Hida family $\hF$ since it satisfies the hypothesis \ref{lab:pdist}.
\begin{align}\label{eq:filF}
	0 \rightarrow \Fil^+L_\hF \rightarrow L_\hF \rightarrow \dfrac{L_\hF}{\Fil^+L_\hF} \rightarrow 0.
\end{align}

We can define the usual non-primitive Selmer group (which is denoted $\Sel^{\Sigma_0}_{\rhopmb{4}{2}}(\Q)$ in the introduction):
\begin{align*}
\Sel^{\Sigma_0}( \DDD_{\II_{\hF,\hG}}) \coloneqq \ker\left(H^1_{\mathrm{cont}}\left(G_S, \DDD_{\II_{\hF,\hG}} \right) \xrightarrow{\mathrm{Res}} H^1_{\mathrm{cont}}\left(I_p, \ \dfrac{\DDD_{\II_{\hF,\hG}}}{\DDD_{\II_{\hF,\hG},\Fil^+}} \right)  \right).
\end{align*}
We will also consider the non-primitive strict Selmer group:
\begin{align*}
\Sel^{\Sigma_0}_{\mathrm{str}}( \DDD_{\II_{\hF,\hG}}) \coloneqq \ker\left(H^1_{\mathrm{cont}}\left(G_S, \DDD_{\II_{\hF,\hG}} \right) \xrightarrow{\mathrm{Res}} H^1_{\mathrm{cont}}\left(\Gal{\overline{\Q}_p}{\Q_p}, \ \dfrac{\DDD_{\II_{\hF,\hG}}}{\DDD_{\II_{\hF,\hG},\Fil^+}} \right)  \right).
\end{align*}
Observe that, we have the following natural inclusion of discrete $\II_{\hF,\hG}$-modules:
\begin{align}\label{eq:strictincl}
	\Sel^{\Sigma_0}_{\mathrm{str}}( \DDD_{\II_{\hF,\hG}}) \subseteq \Sel^{\Sigma_0}( \DDD_{\II_{\hF,\hG}}).
\end{align}

We have the following commutative diagram:

\begin{align} \label{eq:controlselmer1}
\xymatrix{
H^1_{\mathrm{cont}}(G_S,\DDD_{\II_{\hF,\hG}}[\CCC]) \ar[d]_{\mathrm{Res}} \ar[r]^{\cong}& H^1_{\mathrm{cont}}(G_S,\DDD_{\II_{\hF,\hG}})[\CCC] \ar[d]^{\mathrm{Res} \quad }\\ 
H^1_{\mathrm{cont}}\left(\Gal{\overline{\Q}_p}{\Q_p}, \dfrac{\DDD_{\II_{\hF,\hG}}[\CCC]}{\DDD_{\II_{\hF,\hG},\Fil^+}[\CCC]}\right) \ar[r]^{\cong}& H^1_{\mathrm{cont}}\left(\Gal{\overline{\Q}_p}{\Q_p}, \dfrac{\DDD_{\II_{\hF,\hG}}}{\DDD_{\II_{\hF,\hG},\Fil^+}}\right) [\CCC]
}	
\end{align}

The isomorphisms in the top and bottom row both follow from an application of \cite[Proposition 3.4]{MR2290593}. While the isomorphism in the top row relies on the hypothesis \ref{lab:multfree}, the isomorphism on the bottom row relies on the hypothesis \ref{lab:pdist-gsp4}. This allows us to deduce an exact control theorem for the non-primitive strict Selmer group. We have an isomorphism of $\dfrac{\II_{\hF,\hG}}{\CCC}$-modules:  
\begin{align}\label{eq:controlexact}
	\Sel^{\Sigma_0}_{\mathrm{str}}( \DDD_{\II_{\hF,\hG}}[\CCC] )\cong \Sel^{\Sigma_0}_{\mathrm{str}}( \DDD_{\II_{\hF,\hG}})[\CCC].
\end{align}

Theorem \ref{thm:pseudochar} in the appendix will allow us to consider pseudocharacters $\pcT:G_S \rightarrow \TTT$, along with $\pcT_\YYY:G_S \rightarrow \TTT_\YYY$ and  $\pcT_\perp:G_S \rightarrow \TTT^{\rm st}$ respectively. Recall that $\JJJ$ denotes the Yoshida ideal and $\CCC$ denotes the congruence ideal. See \cref{item:yoshidaideal} and \cref{def:congruenceideal}.
 The associated commutative diagram below
\begin{align*}
\xymatrix{
G_S \ar[r]^{\pcT}& \TTT\ar[d] \ar[r]& \TTT^{\rm st} \ar[r]& \TTT^{\rm st}/\JJJ\ar[d]^{\cong}\ar[r]& \dfrac{\II_{\hF,\hG}}{\CCC}  \\
& \TTT_\YYY \ar[rr]&& \dfrac{\TTT_\YYY}{\CCC} \ar@{^{(}->}[ru]
}	
\end{align*}
lets us conclude that  $\JJJ$  is contained in the reducibility ideal for the pseudocharacter $\pcT_{\perp}:G_S \rightarrow \TTT^{\rm st}$. We have the following equality of pseudocharacters valued in $\II_{\hF,\hG}/\CCC$:
\begin{align*}\pcT_{\perp} \ \mathrm{mod}\ \JJJ & \equiv \pcT_\YYY \  \mathrm{mod} \  \CCC, \\ &\equiv \mathrm{Trace}(\rho_\hF) + \mathrm{Trace}\left(\rho_\hG(\kappa_\hF\kappa_\hG^{-1})\right) \ \mathrm{mod} \ \CCC. \end{align*}
Note that by Theorem \ref{thm:yoshidafamily}, the integral closure $\widetilde{\TTT}_\YYY$ of $\TTT_\YYY$ coincides with $\II_{\hF,\hG}$. By results of Nagata \cite[Theorem 7]{MR0063354}, also observe that $\II_{\hF,\hG}$ is finitely generated over $\TTT_\YYY$. Let $\BBB_{\perp,\stab}$ denote $\BBB \otimes_\TTT \TTT^{\rm st}$. Let $i_{\perp,\stab}(\BBB)$ denote the natural image of the fractional ideal $\BBB$ inside the fraction field of $\TTT^{\rm st}$. Using Proposition  \ref{prop:inclusion} along with equations (\ref{eq:strictincl}) and (\ref{eq:controlexact}), we obtain the natural surjection of $\II_{\hF,\hG}$-modules in the top row of the commutative diagram below:

\begin{align}\label{eq:selcongruenceidealsurj}
\xymatrix@C-=50cm{
\Sel^{\Sigma_0}( \DDD_{\II_{\hF,\hG}} )^\vee \ar@{->>}[d]_{\text{eq.~(\ref{eq:strictincl})} } \ar@{->>}[rr]&& \BBB \otimes_\TTT \dfrac{\II_{\hF,\hG}}{\CCC} \cong \BBB_\perp \otimes_{\TTT^{\rm st}}\dfrac{ \II_{\hF,\hG}}{\CCC}, \\
\Sel^{\Sigma_0}_{\mathrm{str}}( \DDD_{\II_{\hF,\hG}} )^\vee \ar@{->>}[r]& \left(\Sel^{\Sigma_0}_{\mathrm{str}}( \DDD_{\II_{\hF,\hG}} )[\CCC]\right)^\vee \ar[r]^{\cong}_{\text{eq.~(\ref{eq:controlexact})}} & 
 \Sel^{\Sigma_0}_{\mathrm{str}}( \DDD_{\II_{\hF,\hG}}[\CCC] )^\vee \ar@{->>}[u]_{\text{\qquad \ \ \ \ \ \ Prop. \ref{prop:inclusion}}} }
\end{align}

Theorem \ref{thm:yoshidafamily}\ref{thmcond:otheryoshida} lets us conclude that for each minimal prime ideal $\eta$ of $\TTT^{\rm st}$, the image $i_\eta(\BBB)$ of the fractional ideal $\BBB_\perp$ inside the field $\TTT_\eta$ is non-zero. Since $\TTT^{\rm st}$ is a reduced ring, this observation in turn lets us conclude that $i_{\perp,\stab}(\BBB)$ is a faithful $\TTT^{\rm st}$-module. As a result, we have $\mathrm{Fitt}_{\TTT^{\rm st}} (i_{\perp,\stab}(\BBB)) = (0)$. See \cite[Proposition 20.7]{MR1322960}. Note that we have a natural surjective map $\BBB_\perp \twoheadrightarrow i_{\perp,\stab}(\BBB)$. By using properties of Fitting ideals under surjection, we have $\mathrm{Fitt}_{\TTT^{\rm st}} (\BBB_\perp) = (0)$. One can now use, following \cite[Corollary 20.5]{MR1322960}, properties of Fitting ideals under the base change  $\TTT \rightarrow \TTT^{\rm st} \rightarrow \dfrac{\II_{\hF,\hG}}{\CCC}$  and the surjections given in equation (\ref{eq:selcongruenceidealsurj}).  
\begin{align}\label{eq:fittingidealinclusion}
\notag \mathrm{Fitt}_{\TTT^{\rm st}} (\BBB_\perp) = (0) &\implies 	\mathrm{Fitt}_{\dfrac{ \II_{\hF,\hG}}{\CCC}} \left(\BBB_\perp \otimes_{\TTT^{\rm st}}\dfrac{ \II_{\hF,\hG}}{\CCC}\right) = (0), \\ 
\notag& \implies \mathrm{Fitt}_{\dfrac{ \II_{\hF,\hG}}{\CCC}}\left(\left(\Sel^{\Sigma_0}( \DDD_{\II_{\hF,\hG}})[\CCC]\right)^\vee\right) =(0), \\ 
 \notag& \implies \mathrm{Fitt}_{\II_{\hF,\hG}}\left(\left(\Sel^{\Sigma_0}( \DDD_{\II_{\hF,\hG}})[\CCC] \right)^\vee\right) \subseteq \CCC, \\ 
 & \implies \mathrm{Fitt}_{\II_{\hF,\hG}}\left(\Sel^{\Sigma_0}( \DDD_{\II_{\hF,\hG}})^\vee\right) \subseteq \CCC.
\end{align}

For every height one prime ideal $\p$ in $\II_{\hF,\hG}$, the localization $(\II_{\hF,\hG})_\p$ is a DVR. Using properties of Fitting ideals under the base change $\II_{\hF,\hG} \rightarrow (\II_{\hF,\hG})_\p$, we get the following inequality of lengths using equation (\ref{eq:fittingidealinclusion}):

\begin{align}\label{eq:lenfit}
\mathrm{len}_{(\II_{\hF,\hG})_\p}	\left(\Sel^{\Sigma_0}( \DDD_{\II_{\hF,\hG}})^\vee\right)_\p \geq \mathrm{len}_{(\II_{\hF,\hG})_\p} \left(\dfrac{(\II_{\hF,\hG})_\p}{\CCC_\p}\right), \quad \text{for every height one prime ideal $\p$ in $\II_{\hF,\hG}$}.
\end{align}

Equation (\ref{eq:main1}) now follows from the inequality (\ref{eq:lenfit}). The theorem follows. 
\end{proof}

\subsection{Proof of Corollary \ref{cor:mainconj3}: specialization of Selmer groups}

{\corollarythirteen*}

In the setting of Corollary \ref{cor:mainconj3}, we have the following $\II_{\hF,\hG}[\![\Gamma_\Cyc]\!][G_S]$-modules: 
\begin{align*}
\SLLL_{\II_{\hF,\hG}[\![\Gamma_\Cyc]\!]}  &\coloneqq \Hom_{\II_{\hF,\hG}}\left(L_\hF \otimes_{\II_{\hF}} \II_{\hF,\hG}, L_\hG \otimes_{\II_{\hG}} \II_{\hF,\hG}\right) \otimes_{\II_{\hF,\hG}} \II_{\hF,\hG}[\![\Gamma_\Cyc]\!](\widetilde{\kappa}^{-1}), \\ \qquad \DDD_{\II_{\hF,\hG}[\![\Gamma_\Cyc]\!]} &\coloneqq \SLLL_{\II_{\hF,\hG}[\![\Gamma_\Cyc]\!]}  \otimes_{\II_{\hF,\hG}[\![\Gamma_\Cyc	]\!]} \widehat{\II_{\hF,\hG}[\![\Gamma_\Cyc]\!]}.
\end{align*}
Just as in Section \ref{sec:proofthm2}, we also have the following discrete non-primitive Selmer group (denoted $\Sel^{\Sigma_0}_{\rhopmb{4}{3}}(\Q)$ in the introduction) and the non-primitive strict Selmer group:
\begin{align}
\Sel^{\Sigma_0}( \DDD_{\II_{\hF,\hG}[\![\Gamma_\Cyc]\!]}) &\coloneqq \ker\bigg(H^1_{\mathrm{cont}}\left(G_S, \DDD_{\II_{\hF,\hG}[\![\Gamma_\Cyc]\!]} \right) \xrightarrow{\mathrm{Res}} H^1_{\mathrm{cont}}\left(I_p, \ \dfrac{\DDD_{\II_{\hF,\hG}[\![\Gamma_\Cyc]\!]}}{\DDD_{\II_{\hF,\hG}[\![\Gamma_\Cyc]\!],\Fil^+}} \right)  \bigg), 
\end{align}
\begin{align} \label{eq:strictselmer43}
 \Sel^{\Sigma_0}_{\mathrm{str}}( \DDD_{\II_{\hF,\hG}[\![\Gamma_\Cyc]\!]}) &\coloneqq \ker\bigg(H^1_{\mathrm{cont}}\left(G_S, \DDD_{\II_{\hF,\hG}[\![\Gamma_\Cyc]\!]} \right) \xrightarrow{\mathrm{Res}} H^1_{\mathrm{cont}}\left(\Gal{\overline{\Q}_p}{\Q_p}, \ \dfrac{\DDD_{\II_{\hF,\hG}[\![\Gamma_\Cyc]\!]}}{\DDD_{\II_{\hF,\hG}[\![\Gamma_\Cyc]\!],\Fil^+}} \right)  \bigg).
\end{align}

We consider the following claim: 
\begin{claim}\label{claim:isoselm}
We have the following isomorphisms of Selmer groups over $\II_{\hF,\hG}[\![\Gamma_\Cyc]\!]$ and $\II_{\hF,\hG}$ respectively: 
\begin{align}
	\Sel^{\Sigma_0}_{\mathrm{str}}( \DDD_{\II_{\hF,\hG}[\![\Gamma_\Cyc]\!]}) \stackrel{\cong}{\hookrightarrow}   \Sel^{\Sigma_0}( \DDD_{\II_{\hF,\hG}[\![\Gamma_\Cyc]\!]}), \qquad \Sel^{\Sigma_0}_{\mathrm{str}}( \DDD_{\II_{\hF,\hG}})  \stackrel{\cong}{\hookrightarrow} \Sel^{\Sigma_0}( \DDD_{\II_{\hF,\hG}}).
\end{align}
\end{claim}

\begin{proof}[Proof of Claim  \ref{claim:isoselm}]
	We provide the proof of the isomorphism of $\II_{\hF,\hG}$-modules. The proof for the isomorphism of $\II_{\hF,\hG}[\![\Gamma_\Cyc]\!]$-modules can be similarly deduced. The hypothesis \ref{lab:pdist-gsp4} crucially comes into play. Without this assumption (e.g. in the case of trivial zeros), such an equality of divisors need not always hold.  One can use the inflation-restriction map and the snake lemma to show that the cokernel of the natural inclusion $\Sel^{\Sigma_0}_{\mathrm{str}}( \DDD_{\II_{\hF,\hG}}) \hookrightarrow \Sel^{\Sigma_0}( \DDD_{\II_{\hF,\hG}})$ is an $\II_{\hF,\hG}$-submodule of 
	\begin{align}\label{eq:claimpn} 
	H^1_{\mathrm{cont}}\left(\Gamma_p\,H^0_{\mathrm{cont}}\left(I_p, \dfrac{\DDD_{\II_{\hF,\hG}}}{\DDD_{\II_{\hF,\hG},\Fil^+}}\right)\right).
	\end{align}
	Here $\Gamma_p$, which is isomorphic to $\hat{\Z}$, denotes the Galois group over $\Q_p$ of the maximal unramified extension of $\Q_p$. We let the Frobenius $\Frob_p$ denote a topological generator of $\Gamma_p$. To prove the claim, it suffices to show that the Pontryagin dual of the $\II_{\hF,\hG}$-module appearing in equation (\ref{eq:claimpn}) is zero. Consider the following short exact sequence of $\II_{\hF,\hG}$-modules.
\begin{align}\label{ses:frobp1}
\notag	0 \rightarrow H^1_{\mathrm{cont}}\left(\Gamma_p\,H^0_{\mathrm{cont}}\left(I_p, \dfrac{\DDD_{\II_{\hF,\hG}}}{\DDD_{\II_{\hF,\hG},\Fil^+}}\right)\right)^\vee   \rightarrow H^0_{\mathrm{cont}}\left(I_p, \dfrac{\DDD_{\II_{\hF,\hG}}}{\DDD_{\II_{\hF,\hG},\Fil^+}}\right)^\vee \xrightarrow {\Frob_p-1} & H^0_{\mathrm{cont}}\left(I_p, \dfrac{\DDD_{\II_{\hF,\hG}}}{\DDD_{\II_{\hF,\hG},\Fil^+}}\right)^\vee \\ &  \underbrace{\rightarrow H^0_{\mathrm{cont}}\left(\Gal{\overline{\Q}_p}{\Q_p}, \dfrac{\DDD_{\II_{\hF,\hG}}}{\DDD_{\II_{\hF,\hG},\Fil^+}}\right)^\vee}_{0} \rightarrow 0. 
\end{align}

Hypothesis \ref{lab:pdist-gsp4} and \cite[Corollary 3.1.1]{MR2290593} let us conclude that $H^0\left(\Gal{\overline{\Q}_p}{\Q_p}, \dfrac{\DDD_{\II_{\hF,\hG}}}{\DDD_{\II_{\hF,\hG},\Fil^+}}\right)~=~0$. An application of Nakayama's lemma would tell us that a surjective endomorphism of a finitely generated $\II_{\hF,\hG}$-module must be an isomorphism. Consequently, the kernel of $\Frob_p-1$ in equation (\ref{ses:frobp1}), and hence the $\II_{\hF,\hG}$-module appearing in equation (\ref{eq:claimpn}), equals zero. The claim follows. 
\end{proof}

We consider the continuous $\II_{\hF,\hG}$-algebra homomorphism:
\begin{align*}
	\varpi_{3,2}:\II_{\hF,\hG}[\![\Gamma_\Cyc]\!] &\twoheadrightarrow \II_{\hF,\hG},\\
	\gamma &\mapsto \left(\sqrt{\kappa_\hF(\gamma)\kappa_\hG(\gamma)^{-1}}\right)^{-1}, \qquad \forall \ \gamma \in \Gamma_\Cyc.
\end{align*}

By the construction of the ring homomorphism, we have a natural isomorphism $ \DDD_{\II_{\hF,\hG}} \cong \DDD_{\II_{\hF,\hG}[\![\Gamma_\Cyc]\!]}[\ker{\varpi_{3,2}}]$   of discrete $\II_{\hF,\hG}[G_S]$-modules. Note that $\ker(\varpi_{3,2})$ is a principal ideal in $\II_{\hF,\hG}[\![\Gamma_\Cyc]\!]$; a generator can be given by \[\gamma_0 - \left(\sqrt{\kappa_\hF(\gamma_0)\kappa_\hG(\gamma_0)^{-1}}\right)^{-1},\] where $\gamma_0$ is a topological generator of $\Gamma_\Cyc$. Just as in equation (\ref{eq:controlselmer1}), an application of \cite[Proposition 3.4]{MR2290593} along with the hypotheses \ref{lab:multfree} and \ref{lab:pdist-gsp4} will provide us the following commutative diagram:

\begin{align} 
\xymatrix{
H^1_{\mathrm{cont}}(G_S,\DDD_{\II_{\hF,\hG}}) \ar[d]_{\mathrm{Res}} \ar[r]^{\cong}& H^1_{\mathrm{cont}}(G_S,\DDD_{\II_{\hF,\hG}[\![\Gamma_\Cyc]\!]})[\ker(\varpi_{3,2})] \ar[d]^{\mathrm{Res} \quad }\\ 
H^1_{\mathrm{cont}}\left(\Gal{\overline{\Q}_p}{\Q_p}, \dfrac{\DDD_{\II_{\hF,\hG}}}{\DDD_{\II_{\hF,\hG},\Fil^+}}\right) \ar[r]^{\cong \qquad }& H^1_{\mathrm{cont}}\left(\Gal{\overline{\Q}_p}{\Q_p}, \dfrac{\DDD_{\II_{\hF,\hG}[\![\Gamma_\Cyc]\!]}}{\DDD_{\II_{\hF,\hG}[\![\Gamma_\Cyc]\!],\Fil^+}}\right) [\ker(\varpi_{3,2})]
}	
\end{align}

Just as in equation (\ref{eq:controlexact}), we have an exact control theorem for the ring homomorphism $\varpi_{3,2}$, providing the following isomorphism of discrete $\II_{\hF,\hG}$-modules:
\begin{align}\label{eq:controlselmer32}
\Sel^{\Sigma_0}_{\mathrm{str}}( \DDD_{\II_{\hF,\hG}}) \cong \Sel^{\Sigma_0}_{\mathrm{str}}( \DDD_{\II_{\hF,\hG}[\![\Gamma_\Cyc]\!]})[\ker{\varpi_{3,2}}].
\end{align}

\begin{claim}\label{claim:nopn}
The $\II_{\hF,\hG}[\![\Gamma_\Cyc]\!]$-module $\Sel^{\Sigma_0}_{\mathrm{str}}( \DDD_{\II_{\hF,\hG}[\![\Gamma_\Cyc]\!]})^\vee$ has no non-trivial pseudonull submodules.	 
\end{claim} 

\begin{proof}[Proof of Claim \ref{claim:nopn}]

We now consider a few observations to show that all the hypotheses of \cite[Proposition 4.1.1]{MR3629652} are satisfied. 
\begin{itemize}
\item Since $\DDD_{\II_{\hF,\hG}[\![\Gamma_\Cyc]\!]}$ is co-free as a $\II_{\hF,\hG}[\![\Gamma_\Cyc]\!]$-module, the hypothesis labelled $\mathrm{RFX}(D)$ in \cite{MR3629652} is automatically satisfied. Similarly, since $\rhopmb{4}{3}$ is a cyclotomic twist deformation, the hypotheses labelled $\mathrm{LOC}^{(1)}_\nu(\DDD_{\II_{\hF,\hG}[\![\Gamma_\Cyc]\!]})$ and $\mathrm{LOC}^{(2)}_\nu(\DDD_{\II_{\hF,\hG}[\![\Gamma_\Cyc]\!]})$ in \cite{MR3629652} are satisfied for all non-archimedean primes $\nu$. 
\item The local condition at $p$ defining the Selmer group $\Sel^{\Sigma_0}_{\mathrm{str}}( \DDD_{\II_{\hF,\hG}[\![\Gamma_\Cyc]\!]})$ is given by the $\II_{\hF,\hG}[\![\Gamma_\Cyc]\!]$-module:
\begin{align}\label{eq:localcondatpstr} \ker\left(H^1_{\mathrm{cont}}\left(\Gal{\overline{\Q}_p}{\Q_p}, \DDD_{\II_{\hF,\hG}[\![\Gamma_\Cyc]\!]} \right)  \rightarrow H^1_{\mathrm{cont}}\left(\Gal{\overline{\Q}_p}{\Q_p}, \dfrac{\DDD_{\II_{\hF,\hG}[\![\Gamma_\Cyc]\!]}}{\DDD_{\II_{\hF,\hG}[\![\Gamma_\Cyc]\!],\Fil^+}} \right) \right)
\end{align}

This $\II_{\hF,\hG}[\![\Gamma_\Cyc]\!]$-module in turn is a quotient of the $\II_{\hF,\hG}[\![\Gamma_\Cyc]\!]$-module $H^1\left(\Gal{\overline{\Q}_p}{\Q_p}, \DDD_{\II_{\hF,\hG}[\![\Gamma_\Cyc]\!],\Fil^+} \right)$. Note that the second local Galois cohomology group $H^2\left(\Gal{\overline{\Q}_p}{\Q_p}, \DDD_{\II_{\hF,\hG}[\![\Gamma_\Cyc]\!],\Fil^+} \right)$ equals zero. This is because on the one hand, by \cite[Proposition 5.1]{MR2290593}, 
it is a co-reflexive $\II_{\hF,\hG}[\![\Gamma_\Cyc]\!]$-module; while on the other hand, since the Galois representation $\rhopmb{4}{3}$ is a cyclotomic twist deformation, using local duality lets us conclude that it is a co-torsion $\II_{\hF,\hG}[\![\Gamma_\Cyc]\!]$-module. Applying \cite[Proposition 5.3]{MR2290593} lets us conclude that $H^1\left(\Gal{\overline{\Q}_p}{\Q_p}, \DDD_{\II_{\hF,\hG}[\![\Gamma_\Cyc]\!],\Fil^+} \right)$ is an \textit{almost divisible} $\II_{\hF,\hG}[\![\Gamma_\Cyc]\!]$-module (following the terminology of \cite{MR3629652}). 

Similarly, for all non-archimedean primes $\nu \neq p$, the local condition at a prime $\nu \neq p$ defining the Selmer group $\Sel^{\Sigma_0}_{\mathrm{str}}( \DDD_{\II_{\hF,\hG}[\![\Gamma_\Cyc]\!]})$ is equal to $H^1\left(\Gal{\overline{\Q}_\nu}{\Q_\nu}, \DDD_{\II_{\hF,\hG}[\![\Gamma_\Cyc]\!]} \right)$. As earlier, using the fact that the Galois representation $\rhopmb{4}{3}$ is a cyclotomic twist deformation, along with local Tate duality, lets us conclude that $H^2\left(\Gal{\overline{\Q}_\nu}{\Q_\nu}, \DDD_{\II_{\hF,\hG}[\![\Gamma_\Cyc]\!]} \right)$ equals zero. Applying \cite[Proposition 5.3]{MR2290593} once again lets us conclude that $H^1\left(\Gal{\overline{\Q}_\nu}{\Q_\nu}, \DDD_{\II_{\hF,\hG}[\![\Gamma_\Cyc]\!]} \right)$ is an \textit{almost divisible} $\II_{\hF,\hG}[\![\Gamma_\Cyc]\!]$-module (following the terminology of \cite{MR3629652}).

Since the local condition at $p$, given in equation (\ref{eq:localcondatpstr}), is an $\II_{\hF,\hG}[\![\Gamma_\Cyc]\!]$-module quotient of the Galois cohomology group $H^1\left(\Gal{\overline{\Q}_p}{\Q_p}, \DDD_{\II_{\hF,\hG}[\![\Gamma_\Cyc]\!],\Fil^+}\right)$, the local condition at $p$ is also an \textit{almost divisible} $\II_{\hF,\hG}[\![\Gamma_\Cyc]\!]$-module.  One can thus conclude that the entire local condition defining the non-primitive strict Selmer group $\Sel^{\Sigma_0}_{\mathrm{str}}( \DDD_{\II_{\hF,\hG}[\![\Gamma_\Cyc]\!]})$ is \textit{almost divisible}.  

\item As indicated in the proof of Proposition \ref{prop:LRYoshidafam}, the ring homomorphisms $\varpi_f: \II_{\hF} \hookrightarrow (\II_{\hF})_{\p_{f_0}} \rightarrow \overline{\Q}_p$  and $\varpi_g: \II_{\hG} \hookrightarrow (\II_{\hG})_{\p_{g_0}} \rightarrow \overline{\Q}_p$ corresponding to the classical specializations $f_0$, $g_0$ of $\hF$, $\hG$ respectively induce a ring homomorphism $\varpi_{f_0,g_0}: \II_{\hF,\hG} \rightarrow \overline{\Q}_p$. Following the same reasoning that we had used earlier, to obtain the control theorems in equations (\ref{eq:controlexact}) and (\ref{eq:controlselmer32}), one can similarly deduce a control theorem using \cite[Proposition 3.4]{MR2290593}  for the non-primitive strict Selmer group under the ring homomorphism $\varpi_{f_0,g_0}$:
\begin{align}\label{eq:controlselmerfinite}
\Sel^{\Sigma_0}_{\mathrm{str}}(\mathrm{D}_{\OO}) \cong \Sel^{\Sigma_0}_{\mathrm{str}}( \DDD_{\II_{\hF,\hG}})[\ker{\varpi_{f_0,g_0}}].
\end{align}
Equation (\ref{eq:controlselmerfinite}) and the hypothesis \ref{lab:tor} now lets us conclude that $\II_{\hF,\hG}$-module $\Sel^{\Sigma_0}_{\mathrm{str}}( \DDD_{\II_{\hF,\hG}})^\vee$ is torsion. Moreover, the control theorem in equation (\ref{eq:controlselmer32}) then lets us conclude that the $\II_{\hF,\hG}[\![\Gamma_\Cyc]\!]$-module $\Sel^{\Sigma_0}_{\mathrm{str}}( \DDD_{\II_{\hF,\hG}[\![\Gamma_\Cyc]\!]})^\vee$ is torsion, that is, we have the following equality of ranks: 
\begin{align}\label{eq:rankselmer}
		\mathrm{Rank}_{\II_{\hF,\hG}[\![\Gamma_\Cyc]\!]}\Sel^{\Sigma_0}_{\mathrm{str}}( \DDD_{\II_{\hF,\hG}[\![\Gamma_\Cyc]\!]})^\vee=0.
\end{align}

Since $\rhopmb{4}{3}$ is a cyclotomic twist deformation, the local Euler--Poincar\'e formula \cite[Proposition 4.2]{MR2290593} would give us the following equality of ranks for the first local Galois cohomology group 
\begin{align}\label{eq:firstgaloislocal}
	\mathrm{Rank}_{\II_{\hF,\hG}[\![\Gamma_\Cyc]\!]} H^1_{\mathrm{cont}}\left(\Gal{\overline{\Q}_p}{\Q_p}, \dfrac{\DDD_{\II_{\hF,\hG}[\![\Gamma_\Cyc]\!]}}{\DDD_{\II_{\hF,\hG}[\![\Gamma_\Cyc]\!],\Fil^+}} \right)^\vee =2.
	\end{align}
Combining the torsion-ness of the $\II_{\hF,\hG}[\![\Gamma_\Cyc]\!]$-module $\Sel^{\Sigma_0}_{\mathrm{str}}( \DDD_{\II_{\hF,\hG}[\![\Gamma_\Cyc]\!]})^\vee$ with the global Euler--Poincar\'e formula \cite[Proposition 4.1]{MR2290593} would provide us the following equality of ranks:
\begin{align}\label{eq:rankglobalgalois}
\mathrm{Rank}_{\II_{\hF,\hG}[\![\Gamma_\Cyc]\!]} H^1_{\mathrm{cont}}\left(G_S, \DDD_{\II_{\hF,\hG}[\![\Gamma_\Cyc]\!]}\right)^\vee=2, \qquad 	\mathrm{Rank}_{\II_{\hF,\hG}[\![\Gamma_\Cyc]\!]} H^2_{\mathrm{cont}}\left(G_S, \DDD_{\II_{\hF,\hG}[\![\Gamma_\Cyc]\!]}\right)^\vee=0.
\end{align}

Combining the equality of ranks in equations (\ref{eq:rankselmer}), (\ref{eq:firstgaloislocal}) and (\ref{eq:rankglobalgalois}) would allow us to deduce the hypothesis labeled $\mathrm{CRK}(\DDD_{\II_{\hF,\hG}[\![\Gamma_\Cyc]\!]},\mathcal{L})$ in \cite{MR3629652} holds. Combining the second equality in equation (\ref{eq:rankglobalgalois}) along with \cite[Proposition 6.1]{MR2290593} would let us deduce that $H^2_{\mathrm{cont}}\left(G_S, \DDD_{\II_{\hF,\hG}[\![\Gamma_\Cyc]\!]}\right)=0$. The weak Leopoldt conjecture, denoted $\mathrm{LEO}(\DDD_{\II_{\hF,\hG}[\![\Gamma_\Cyc]\!]})$ in \cite{MR3629652}, must also therefore hold. 
\item Since we are looking at the non-primitive Selmer group, the local condition at the prime $\ell_0$ is zero. Note also that $\rhopmb{4}{3}$ is a cyclotomic twist deformation. Consequently, $\mathrm{LOC}^{(1)}_{\ell_0}(\DDD_{\II_{\hF,\hG}[\![\Gamma_\Cyc]\!]})$ and  hypothesis (c) of \cite[Proposition 4.1.1]{MR3629652} are also satisfied. 
\end{itemize}

One can now use \cite[Proposition 4.1.1]{MR3629652} to conclude that $\Sel^{\Sigma_0}_{\mathrm{str}}( \DDD_{\II_{\hF,\hG}[\![\Gamma_\Cyc]\!]})^\vee$ has no non-trivial $\II_{\hF,\hG}[\![\Gamma_\Cyc]\!]$-pseudonull submodules.
\end{proof}

\begin{claim}\label{claim:finproj}
For every height two prime ideal $\QQ$ in the ring $\II_{\hF,\hG}[\![\Gamma_\Cyc]\!]$, the $\left(\II_{\hF,\hG}[\![\Gamma_\Cyc]\!]\right)_\QQ$-module $\left(\Sel^{\Sigma_0}_{\mathrm{str}}( \DDD_{\II_{\hF,\hG}[\![\Gamma_\Cyc]\!]})^\vee\right)_\QQ$ has finite projective dimension. 
\end{claim}

\begin{proof}[Proof of Claim \ref{claim:finproj}]
Observe that using the ingredients in the proof of Claim \ref{claim:nopn} and \cite[Corollary 3.2.3]{MR2740696}, one deduces that the global-to-local map defining the strict Selmer group is surjective. It thus suffices to show that for every height two prime ideal $\QQ$ in the ring $\II_{\hF,\hG}[\![\Gamma_\Cyc]\!]$, the following $\left(\II_{\hF,\hG}[\![\Gamma_\Cyc]\!]\right)_\QQ$-modules have finite projective dimension:
\begin{align}
	 \left(H^1_{\mathrm{cont}}\left(G_S, \DDD_{\II_{\hF,\hG}[\![\Gamma_\Cyc]\!]}\right)^\vee\right)_\QQ, \qquad \left(H^1_{\mathrm{cont}}\left(\Gal{\overline{\Q}_p}{\Q_p}, \dfrac{\DDD_{\II_{\hF,\hG}[\![\Gamma_\Cyc]\!]}}{\DDD_{\II_{\hF,\hG}[\![\Gamma_\Cyc]\!],\Fil^+}} \right)^\vee\right)_\QQ
\end{align}
Hypothesis  \ref{lab:irr} lets us conclude that $H^0_{\mathrm{cont}}\left(G_S, \DDD_{\II_{\hF,\hG}[\![\Gamma_\Cyc]\!]}\right)=0$. In the proof of Claim \ref{claim:nopn}, to establish the weak Leopoldt conjecture, we showed that $H^2_{\mathrm{cont}}\left(G_S, \DDD_{\II_{\hF,\hG}[\![\Gamma_\Cyc]\!]}\right)=0$. Applying \cite[Proposition 5.5]{MR3919711} lets us conclude that  for every height two prime ideal $\QQ$ in the ring $\II_{\hF,\hG}[\![\Gamma_\Cyc]\!]$, the $\left(\II_{\hF,\hG}[\![\Gamma_\Cyc]\!]\right)_\QQ$-module $\left(H^1_{\mathrm{cont}}\left(G_S, \DDD_{\II_{\hF,\hG}[\![\Gamma_\Cyc]\!]}\right)^\vee\right)_\QQ$ has finite projective dimension.

Using local Tate duality and the fact that the Galois representation $\rhopmb{4}{3}$ is a cyclotomic twist deformation, one can conclude that $H^2_{\mathrm{cont}}\left(\Gal{\overline{\Q}_p}{\Q_p}, \dfrac{\DDD_{\II_{\hF,\hG}[\![\Gamma_\Cyc]\!]}}{\DDD_{\II_{\hF,\hG}[\![\Gamma_\Cyc]\!],\Fil^+}} \right)=0$. Moreover, hypothesis \ref{lab:pdist-gsp4} lets us conclude that $H^0_{\mathrm{cont}}\left(\Gal{\overline{\Q}_p}{\Q_p}, \dfrac{\DDD_{\II_{\hF,\hG}[\![\Gamma_\Cyc]\!]}}{\DDD_{\II_{\hF,\hG}[\![\Gamma_\Cyc]\!],\Fil^+}} \right)=0$. 
Applying \cite[Proposition 5.5]{MR3919711} again lets us conclude that  for every height two prime ideal $\QQ$ in the ring $\II_{\hF,\hG}[\![\Gamma_\Cyc]\!]$, the $\left(\II_{\hF,\hG}[\![\Gamma_\Cyc]\!]\right)_\QQ$-module $\left(H^1_{\mathrm{cont}}\left(\Gal{\overline{\Q}_p}{\Q_p}, \dfrac{\DDD_{\II_{\hF,\hG}[\![\Gamma_\Cyc]\!]}}{\DDD_{\II_{\hF,\hG}[\![\Gamma_\Cyc]\!],\Fil^+}} \right)^\vee\right)_\QQ$ has finite projective dimension. 
The proof of the claim follows. 
\end{proof}

Claim \ref{claim:isoselm} and the hypothesis of Corollary \ref{cor:mainconj3} will provide us the following inequality of divisors in $\II_{\hF,\hG}[\![\Gamma_\Cyc]\!]$:
 
  \begin{align} \label{eq:eulersysteminequality}
   & \Div\left(\Sel^{\Sigma_0}_{\mathrm{str}}( \DDD_{\II_{\hF,\hG}[\![\Gamma_\Cyc]\!]})^\vee\right) =   \Div\left(\Sel^{\Sigma_0}( \DDD_{\II_{\hF,\hG}[\![\Gamma_\Cyc]\!]})^\vee\right) \leq  \Div(\theta^{\Sigma_0}_{\pmbtwo{4}{3}}). 
\end{align}	

\cite[Proposition 2.2]{Lei_2020} and hypothesis \ref{lab:irr} tell us that the divisor $\Div\left(\Sel^{\Sigma_0}( \DDD_{\II_{\hF,\hG}[\![\Gamma_\Cyc]\!]})^\vee\right)$ associated to the non-primitive Selmer group of a cyclotomic twist deformation is principal.  Claims \ref{claim:nopn}, \ref{claim:finproj} and this observation allow us to apply \cite[Proposition 5.4]{MR3919711}; we have the following inequality of divisors in $\II_{\hF,\hG}$, with equality in equation (\ref{eq:eulersysteminequality}) if and only if there is equality in equation (\ref{eq:specialization}):
 
\begin{align} \label{eq:specialization}
 \Div\left(\Sel^{\Sigma_0}_{\mathrm{str}}( \DDD_{\II_{\hF,\hG}})^\vee \right) \leq \Div\left(\theta^{\Sigma_0}_{\pmbtwo{4}{2}}\right). 
\end{align}
  
 Here, we have used the fact that the $2$-variable $p$-adic $L$-function  $\theta^{\Sigma_0}_{\pmbtwo{4}{2}}$ is defined as the specialization of the $3$-variable $p$-adic $L$-function $\theta^{\Sigma_0}_{\pmbtwo{4}{3}}$ under the specialization map $\varpi_{3,2}$. 
  
Theorem \ref{thm:mainconj2}, Claim \ref{claim:isoselm} and the hypothesis of Corollary \ref{cor:mainconj3} will provide the following inequality of divisors in $\II_{\hF,\hG}$:

\begin{align} \label{eq:conginequality}
    \Div\left(\Sel^{\Sigma_0}( \DDD_{\II_{\hF,\hG}})^\vee \right) =   \Div\left(\Sel^{\Sigma_0}_{\mathrm{str}}( \DDD_{\II_{\hF,\hG}})^\vee \right) \geq \Div\left({\CCC}\right) \geq \Div(\theta^{\Sigma_0}_{\pmbtwo{4}{2}}).
  \end{align}

Combining the inequalities in equations (\ref{eq:eulersysteminequality}), (\ref{eq:specialization}) and (\ref{eq:conginequality}) completes the proof of Corollary \ref{cor:mainconj3}. \qed \newline

As an immediate application of Proposition \ref{prop:LRYoshidafam}, we have the following corollary under hypothesis \ref{lab:hypLR} (which replaces the hypotheses \ref{lab:tor} required to establish the existence of $p$-adic family of Yoshida lifts):

\begin{corollary} 
 Suppose that the hypotheses \ref{lab:irr},   \ref{lab:pdist-gsp4}, \ref{lab:multfree},  \ref{lab:yosdet}, \ref{lab:yos}, \ref{lab:yos-neb}, \ref{lab:hypLR} and \ref{lab:stab} hold.  Then, we have the following inequality of divisors:
  \begin{align} 
    \Div\left(\Sel^{\Sigma_0}_{\rhopmb{4}{2}}(\Q)^\vee\right) \geq \Div\left({\CCC}\right).
  \end{align}

  Furthermore, if in addition, one also supposes that the following hypotheses hold:
  \begin{enumerate}
  \item We have the following inequality of divisors in $\II_{\hF,\hG}[\![\Gamma_{\Cyc}]\!]$:
  \begin{align}
    \Div\left(\Sel^{\Sigma_0}_{\rhopmb{4}{3}}(\Q)^\vee\right) \leq  \Div(\theta^{\Sigma_0}_{\pmbtwo{4}{3}}),
  \end{align}

 \item We have the following inequality of divisors in $\II_{\hF,\hG}$:
  \begin{align}
    \Div({\CCC}) {\geq} \Div(\theta^{\Sigma_0}_{\pmbtwo{4}{2}}).
  \end{align}
    \end{enumerate}

  Then, Conjectures \ref{conj:mainconj2} and \ref{conj:mainconj3} hold. That is, we obtain the following equality of divisors in  $\II_{\hF,\hG}[\![\Gamma_{\Cyc}]\!]$:
  \begin{align}
    \Div\left(\Sel^{\Sigma_0}_{\rhopmb{4}{3}}(\Q)^\vee\right) =  \Div\left(\theta^{\Sigma_0}_{\pmbtwo{4}{3}}\right),
  \end{align}
  along with the following equality of divisors in  $\II_{\hF,\hG}$:
  \begin{align}
    \Div\left(\Sel^{\Sigma_0}_{\rhopmb{4}{2}}(\Q)^\vee\right) = \Div\left({\CCC}\right) = \Div\left(\theta^{\Sigma_0}_{\pmbtwo{4}{2}}\right).
  \end{align}
\end{corollary}

\subsection{Proof of Theorem \ref{thm:selmstructure}: pseudo-cyclicity of the primitive Selmer group}

{\theoremthree* }

In this section, we consider two \textit{primitive Selmer groups}, each corresponding to a choice of dominant Hida family ($\hF$ or $\hG$). Just as in equation (\ref{eq:filF}), one has a $\Gal{\overline{\Q}_p}{\Q_p}$-equivariant filtration of free $\II_{\hF,\hG}[\![\Gamma_\Cyc]\!]$-modules associated to the Hida family $\hG$. This in turn allows us to define two free $\II_{\hF,\hG}[\![\Gamma_\Cyc]\!][\Gal{\overline{\Q}_p}{\Q_p}]$-modules of rank $2$ as follows: 
\begin{align}\label{eq:filchoicedom}
\SLLL_{\II_{\hF,\hG}[\![\Gamma_\Cyc]\!],\Fil_{(1)}^+}  &\coloneqq \Hom_{\II_{\hF,\hG}}\left(\dfrac{L_\hF}{\Fil^+L_\hF} \otimes_{\II_{\hF}} \II_{\hF,\hG}, L_\hG \otimes_{\II_{\hG}} \II_{\hF,\hG}\right)  \otimes_{\II_{\hF,\hG}} \II_{\hF,\hG}[\![\Gamma_\Cyc]\!](\widetilde{\kappa}^{-1}), \notag \\
\SLLL_{\II_{\hF,\hG}[\![\Gamma_\Cyc]\!],\Fil_{(2)}^+}  &\coloneqq \Hom_{\II_{\hF,\hG}}\left(L_\hF \otimes_{\II_{\hF}} \II_{\hF,\hG}, \Fil^+L_\hG \otimes_{\II_{\hG}} \II_{\hF,\hG}\right)  \otimes_{\II_{\hF,\hG}} \II_{\hF,\hG}[\![\Gamma_\Cyc]\!](\widetilde{\kappa}^{-1}).
\end{align}
We have two discrete $\II_{\hF,\hG}[\![\Gamma_\Cyc]\!][\Gal{\overline{\Q}_p}{\Q_p}]$-modules associated to these compact Galois lattices: 
\begin{align*}
\DDD_{\II_{\hF,\hG}[\![\Gamma_\Cyc]\!],\Fil_{(1)}^+} \coloneqq \SLLL_{\II_{\hF,\hG}[\![\Gamma_\Cyc]\!],\Fil_{(1)}^+}  \otimes_{\II_{\hF,\hG}[\![\Gamma_\Cyc]\!]} \widehat{\II_{\hF,\hG}[\![\Gamma_\Cyc]\!]}, \quad \DDD_{\II_{\hF,\hG}[\![\Gamma_\Cyc]\!],\Fil_{(2)}^+} \coloneqq \SLLL_{\II_{\hF,\hG}[\![\Gamma_\Cyc]\!],\Fil_{(2)}^+}  \otimes_{\II_{\hF,\hG}[\![\Gamma_\Cyc]\!]} \widehat{\II_{\hF,\hG}[\![\Gamma_\Cyc]\!]}.
\end{align*}

Let $\mathrm{Loc}(\Q_\nu, \DDD_{\II_{\hF,\hG}[\![\Gamma_\Cyc]\!]})$ denote the $\II_{\hF,\hG}[\![\Gamma_\Cyc]\!]$-module $H^1_{\mathrm{cont}}\left(\Gal{\overline{\Q}_\nu}{\Q_\nu}, \ \DDD_{\II_{\hF,\hG}[\![\Gamma_\Cyc]\!]} \right)$. We can define two \textit{primitive (cyclotomic) strict} Selmer groups as given below. Note that these primitive (cyclotomic) strict Selmer groups are isomorphic respectively, as shown in Claim \ref{claim:isoselm}, to the corresponding usual primitive (cyclotomic) Selmer groups, which were denoted $\Sel_{\rhopmb{4}{3},\dominant{\hF}}(\Q)^\vee$  and  $\Sel_{\rhopmb{4}{3},\dominant{\hG}}(\Q)^\vee$ in the statement of Theorem \ref{thm:selmstructure}. 
\begin{align*}
\Sel_{\mathrm{str},(1)}( \DDD_{\II_{\hF,\hG}[\![\Gamma_\Cyc]\!]}) \coloneqq \ker\bigg(& H^1_{\mathrm{cont}}\left(G_S, \DDD_{\II_{\hF,\hG}[\![\Gamma_\Cyc]\!]} \right) \xlongrightarrow{\mathrm{Res}} H^1_{\mathrm{cont}}\left(\Q_p, \ \dfrac{\DDD_{\II_{\hF,\hG}[\![\Gamma_\Cyc]\!]}}{\DDD_{\II_{\hF,\hG}[\![\Gamma_\Cyc]\!],\Fil_{(1)}^+}} \right) \times \prod \limits_{\nu \mid N} \mathrm{Loc}(\Q_\nu, \DDD_{\II_{\hF,\hG}[\![\Gamma_\Cyc]\!]}) \bigg).
\end{align*}
\begin{align*}
\Sel_{\mathrm{str},(2)}( \DDD_{\II_{\hF,\hG}[\![\Gamma_\Cyc]\!]}) \coloneqq \ker\bigg(& H^1_{\mathrm{cont}}\left(G_S, \DDD_{\II_{\hF,\hG}[\![\Gamma_\Cyc]\!]} \right) \xlongrightarrow{\mathrm{Res}}  H^1_{\mathrm{cont}}\left(\Q_p, \ \dfrac{\DDD_{\II_{\hF,\hG}[\![\Gamma_\Cyc]\!]}}{\DDD_{\II_{\hF,\hG}[\![\Gamma_\Cyc]\!],\Fil_{(2)}^+}} \right) \times \prod_{\nu \mid N} \mathrm{Loc}(\Q_\nu, \DDD_{\II_{\hF,\hG}[\![\Gamma_\Cyc]\!]}) \bigg).	
\end{align*}
Inside $\SLLL_{\II_{\hF,\hG}[\![\Gamma_\Cyc]\!]}$, which is a free $\II_{\hF,\hG}[\![\Gamma_\Cyc]\!]$-module of rank four, we consider the following free $\II_{\hF,\hG}[\![\Gamma_\Cyc]\!]$-module of rank three (which is $\Gal{\overline{\Q}_p}{\Q_p}$-stable): 
\begin{align}\label{eq:sumlattice}
	\SLLL_{\II_{\hF,\hG}[\![\Gamma_\Cyc]\!],\Fil_{\mathrm{sum}}^+} \coloneqq \SLLL_{\II_{\hF,\hG}[\![\Gamma_\Cyc]\!],\Fil_{(1)}^+} + \SLLL_{\II_{\hF,\hG}[\![\Gamma_\Cyc]\!],\Fil_{(2)}^+}.
\end{align}
The cokernel of the inclusion of the $\II_{\hF,\hG}[\![\Gamma_\Cyc]\!]$-module of equation (\ref{eq:sumlattice}), inside $\SLLL_{\II_{\hF,\hG}[\![\Gamma_\Cyc]\!]} $, is a free $\II_{\hF,\hG}[\![\Gamma_\Cyc]\!]$-module of rank one; it is isomorphic to $\Hom_{\II_{\hF,\hG}}\left(\Fil^+L_\hF \otimes_{\II_{\hF}} \II_{\hF,\hG}, \dfrac{L_\hG}{\Fil^+L_\hG} \otimes_{\II_{\hG}} \II_{\hF,\hG}\right)  \otimes_{\II_{\hF,\hG}} \II_{\hF,\hG}[\![\Gamma_\Cyc]\!](\widetilde{\kappa}^{-1})$. Consequently, we can consider the following discrete $\II_{\hF,\hG}[\![\Gamma_\Cyc]\!]$-submodule of $\DDD_{\II_{\hF,\hG}[\![\Gamma_\Cyc]\!]}$ (which is $\Gal{\overline{\Q}_p}{\Q_p}$-stable):

\[\DDD_{\II_{\hF,\hG}[\![\Gamma_\Cyc]\!],\Fil_{\mathrm{sum}}^+} \coloneqq \SLLL_{\II_{\hF,\hG}[\![\Gamma_\Cyc]\!],\Fil_{\mathrm{sum}}^+}  \otimes_{\II_{\hF,\hG}[\![\Gamma_\Cyc]\!]} \widehat{\II_{\hF,\hG}[\![\Gamma_\Cyc]\!]}.\]
This, in turn, allows us to consider the following strict primitive cyclotomic Selmer group:
\begin{align*}
\Sel_{\mathrm{str},\mathrm{sum}}( \DDD_{\II_{\hF,\hG}[\![\Gamma_\Cyc]\!]}) \coloneqq \ker\bigg(& H^1_{\mathrm{cont}}\left(G_S, \DDD_{\II_{\hF,\hG}[\![\Gamma_\Cyc]\!]} \right) \xlongrightarrow{\mathrm{Res}}  H^1_{\mathrm{cont}}\left(\Q_p, \ \dfrac{\DDD_{\II_{\hF,\hG}[\![\Gamma_\Cyc]\!]}}{\DDD_{\II_{\hF,\hG}[\![\Gamma_\Cyc]\!],\Fil_{\mathrm{sum}}^+}} \right) \times \prod_{\nu \mid N} \mathrm{Loc}(\Q_\nu, \DDD_{\II_{\hF,\hG}[\![\Gamma_\Cyc]\!]}) \bigg).	
\end{align*}

Let $i \in \{1,2\}$. Applying the Snake Lemma to the following commutative diagram, 
{\tiny \begin{align}
\xymatrix{
 & H^1_{\mathrm{cont}}\left(G_S, \DDD_{\II_{\hF,\hG}[\![\Gamma_\Cyc]\!]} \right) \ar[d]\ar[r]^{\cong}& H^1_{\mathrm{cont}}\left(G_S, \DDD_{\II_{\hF,\hG}[\![\Gamma_\Cyc]\!]} \right) \ar[d] \\
& H^1_{\mathrm{cont}}\left(\Q_p, \ \dfrac{\DDD_{\II_{\hF,\hG}[\![\Gamma_\Cyc]\!]}}{\DDD_{\II_{\hF,\hG}[\![\Gamma_\Cyc]\!],\Fil_{(1)}^+}} \right) \times \prod \limits_{\nu \mid N} \mathrm{Loc}(\Q_\nu, \DDD_{\II_{\hF,\hG}[\![\Gamma_\Cyc]\!]})
 \ar[r]& H^1_{\mathrm{cont}}\left(\Q_p, \ \dfrac{\DDD_{\II_{\hF,\hG}[\![\Gamma_\Cyc]\!]}}{\DDD_{\II_{\hF,\hG}[\![\Gamma_\Cyc]\!],\Fil_{\mathrm{sum}}^+}} \right) \times \prod \limits_{\nu \mid N} \mathrm{Loc}(\Q_\nu, \DDD_{\II_{\hF,\hG}[\![\Gamma_\Cyc]\!]})
}	
\end{align}
}
gives us the following exact sequence of $\II_{\hF,\hG}[\![\Gamma_\Cyc]\!]$-modules:
\begin{align}\label{eq:longexseq}
\notag \mathrm{coker}(\phi_{\Sel_{\mathrm{str},\mathrm{sum}}})^\vee \rightarrow \mathrm{coker}(\phi_{\Sel_{\mathrm{str},(i)}})^\vee & \rightarrow H^1_{\mathrm{cont}}\left(\Q_p, \ \dfrac{\DDD_{\II_{\hF,\hG}[\![\Gamma_\Cyc]\!],\Fil_{\mathrm{sum}}^+}}{\DDD_{\II_{\hF,\hG}[\![\Gamma_\Cyc]\!],\Fil_{(i)}^+}} \right) ^\vee \\ & \rightarrow \Sel_{\mathrm{str},\mathrm{sum}}( \DDD_{\II_{\hF,\hG}[\![\Gamma_\Cyc]\!]})^\vee \rightarrow 	\Sel_{\mathrm{str},(i)}( \DDD_{\II_{\hF,\hG}[\![\Gamma_\Cyc]\!]})^\vee \rightarrow 0.
\end{align}

To justify the claim that the Snake lemma provides the exact sequence in equation (\ref{eq:longexseq}), we need to show that the kernel of the bottom row is isomorphic to the $\II_{\hF,\hG}[\![\Gamma_\Cyc]\!]$-module $H^1_{\mathrm{cont}}\left(\Q_p, \ \dfrac{\DDD_{\II_{\hF,\hG}[\![\Gamma_\Cyc]\!],\Fil_{\mathrm{sum}}^+}}{\DDD_{\II_{\hF,\hG}[\![\Gamma_\Cyc]\!],\Fil_{(i)}^+}} \right) ^\vee$. This follows from an application of the long exact sequence in 
Galois cohomology. The hypothesis \ref{lab:pdist-gsp4} allows us to deduce that $H^0_{\mathrm{cont}}\left(\Q_p, \ \dfrac{\DDD_{\II_{\hF,\hG}[\![\Gamma_\Cyc]\!]}}{\DDD_{\II_{\hF,\hG}[\![\Gamma_\Cyc]\!],\Fil_{\mathrm{sum}}^+}} \right)$ equals zero. \\

\cite[Corollary 3.2.3]{MR2740696} and the proof of Claim \ref{claim:nopn} tell us that the global-to-local map defining the non-primitive Selmer group is surjective. Since the Galois representation $\rhopmb{4}{3}$ is a cyclotomic twist deformation, one can then conclude using \cite[Corollary 3.2.5]{MR2740696} that the cokernel $\mathrm{coker}(\phi_{\Sel_{\mathrm{str},(i)}})^\vee $, of the map defining the primitive Selmer group $\Sel_{\mathrm{str},(i)}( \DDD_{\II_{\hF,\hG}[\![\Gamma_\Cyc]\!]})$ is a co-torsion $\II_{\hF,\hG}[\![\Gamma_\Cyc]\!]$-module, for each $i \in \{1,2\}$. Equation (\ref{eq:longexseq}) then lets us deduce that the cokernel $\mathrm{coker}(\phi_{\Sel_{\mathrm{str},\mathrm{sum}}})$ of the map defining the Selmer group is also a co-torsion $\II_{\hF,\hG}[\![\Gamma_\Cyc]\!]$-module. The torsion-ness of the Pontryagin dual of the primitive cyclotomic Selmer group is implicit in hypothesis \ref{lab:gcd}. An application of both the global Euler--Poincar\'e formula \cite[Proposition 4.1]{MR2290593} and the local Euler--Poincar\'e formula \cite[Proposition 4.2]{MR2290593} along with the computation of the co-ranks of the discrete Galois cohomology groups in the proof of Claim \ref{claim:nopn} then provides us the following equality of ranks: 
\begin{align}\label{eq:rank1crit}
\mathrm{Rank}_{\II_{\hF,\hG}[\![\Gamma_\Cyc]\!]} \ \Sel_{\mathrm{str},\mathrm{sum}}( \DDD_{\II_{\hF,\hG}[\![\Gamma_\Cyc]\!]})^\vee= 1. 	
\end{align}

\begin{claim}\label{claim:notorsion}
	For each $i \in \{1,2\}$, the $\II_{\hF,\hG}[\![\Gamma_\Cyc]\!]$-module $H^1_{\mathrm{cont}}\left(\Q_p, \ \dfrac{\DDD_{\II_{\hF,\hG}[\![\Gamma_\Cyc]\!],\Fil_{\mathrm{sum}}^+}}{\DDD_{\II_{\hF,\hG}[\![\Gamma_\Cyc]\!],\Fil_{(i)}^+}} \right)^\vee$ has no non-zero torsion submodules. 
\end{claim}
\begin{proof}[Proof of Claim \ref{claim:notorsion}]
Consider the following free rank one $\II_{\hF,\hG}$ and $\II_{\hF,\hG}[\![\Gamma_\Cyc]\!]$-modules (which are $\Gal{\overline{\Q}_p}{\Q_p}$-stable):
\begin{align*}
 \SLLL_{\II_{\hF,\hG},\mathrm{quot}}  &\coloneqq \Hom_{\II_{\hF,\hG}}\left(\Fil^+L_\hF \otimes_{\II_{\hF}} \II_{\hF,\hG}, \Fil^+L_\hG \otimes_{\II_{\hG}} \II_{\hF,\hG}\right), \\  	\SLLL_{\II_{\hF,\hG}[\![\Gamma_\Cyc]\!],\mathrm{quot}}  &\coloneqq \SLLL_{\II_{\hF,\hG},\mathrm{quot}}  \otimes_{\II_{\hF,\hG}} \II_{\hF,\hG}[\![\Gamma_\Cyc]\!](\widetilde{\kappa}^{-1}).
\end{align*}
Analysing equations (\ref{eq:filchoicedom}) and (\ref{eq:sumlattice}) tells us that we have the following isomorphism of discrete $\II_{\hF,\hG}[\![\Gamma_\Cyc]\!]$-modules:
\begin{align*}
 \dfrac{\DDD_{\II_{\hF,\hG}[\![\Gamma_\Cyc]\!],\Fil_{\mathrm{sum}}^+}}{\DDD_{\II_{\hF,\hG}[\![\Gamma_\Cyc]\!],\Fil_{(i)}^+}}	\cong \SLLL_{\II_{\hF,\hG}[\![\Gamma_\Cyc]\!],\mathrm{quot}}  \otimes_{\II_{\hF,\hG}[\![\Gamma_\Cyc]\!]} \widehat{\II_{\hF,\hG}[\![\Gamma_\Cyc]\!]}.
	\end{align*}

We also consider the free rank one $\II_{\hF,\hG}$ and $\II_{\hF,\hG}[\![\Gamma_\Cyc]\!]$-modules, associated to the Tate-twist of $\SLLL_{\II_{\hF,\hG}[\![\Gamma_\Cyc]\!],\mathrm{quot}}$ (which are $\Gal{\overline{\Q}_p}{\Q_p}$-stable):
	\begin{align}\label{eq:defquotfiltate}
	\SLLL^*_{\II_{\hF,\hG},\mathrm{quot}} &\coloneqq \Hom_{\II_{\hF,\hG}} \left(\SLLL_{\II_{\hF,\hG},\mathrm{quot}} , \ \II_{\hF,\hG}(\cyc)\right) \notag 
	 \cong \Hom_{\II_{\hF,\hG}}\left(\Fil^+L_\hG \otimes_{\II_{\hG}} \II_{\hF,\hG}, \Fil^+L_\hF \otimes_{\II_{\hF}} \II_{\hF,\hG}\right)(\cyc), \\
	\SLLL^*_{\II_{\hF,\hG}[\![\Gamma_\Cyc]\!],\mathrm{quot}} &\coloneqq \Hom_{\II_{\hF,\hG}[\![\Gamma_\Cyc]\!]} \left(\SLLL_{\II_{\hF,\hG}[\![\Gamma_\Cyc]\!],\mathrm{quot}} , \ \II_{\hF,\hG}[\![\Gamma_\Cyc]\!](\cyc)\right) \cong \SLLL^*_{\II_{\hF,\hG},\mathrm{quot}} \otimes_{\II_{\hF,\hG}} \II_{\hF,\hG}[\![\Gamma_\Cyc]\!](\widetilde{\kappa}).  
\end{align}
We consider the associated discrete $\II_{\hF,\hG}$ and $\II_{\hF,\hG}[\![\Gamma_\Cyc]\!]$-modules:
\begin{align}\label{eq:iotaiso}
	\notag \DDD^*_{\II_{\hF,\hG},\mathrm{quot}} &\coloneqq \SLLL^*_{\II_{\hF,\hG},\mathrm{quot}} \otimes_{\Z_p[\![x_\hF,x_\hG]\!]} \widehat{\Z_p[\![x_\hF,x_\hG]\!]}, \\ \DDD^*_{\II_{\hF,\hG}[\![\Gamma_\Cyc]\!],\mathrm{quot}} &\coloneqq \SLLL^*_{\II_{\hF,\hG}[\![\Gamma_\Cyc]\!],\mathrm{quot}} \otimes_{\Z_p[\![x_\hF,x_\hG]\!][\![\Gamma_\Cyc]\!]} \widehat{\Z_p[\![x_\hF,x_\hG]\!][\![\Gamma_\Cyc]\!]}.
\end{align}

Since the Galois representation $\rhopmb{4}{3}$ is a cyclotomic twist deformation, one can use \cite[Proposition 4.1]{MR3919711} to conclude that as an $\II_{\hF,\hG}$-module, $H^0\left(\Q_p, \DDD^*_{\II_{\hF,\hG}[\![\Gamma_\Cyc]\!],\mathrm{quot}}\right)^\vee$ must be finitely generated; consequently, as an $\II_{\hF,\hG}[\![\Gamma_\Cyc]\!]$-module, $H^0\left(\Q_p, \DDD^*_{\II_{\hF,\hG}[\![\Gamma_\Cyc]\!],\mathrm{quot}}\right)^\vee$ is torsion. In fact, it is annihilated by a monic polynomial, say $\beta$, in $\II_{\hF,\hG}[\![x]\!]$, if we identify $\II_{\hF,\hG}[\![\Gamma_\Cyc]\!]$ with $\II_{\hF,\hG}[\![x]\!]$, by identifying a topological generator $\gamma_0$ of $\Gamma_\Cyc$ with $x+1$.  \\

The action of $G_{\Q_p}\coloneqq \Gal{\overline{\Q}_p}{\Q_p}$ on the rank one $\II_{\hF,\hG}$-module $\SLLL^*_{\II_{\hF,\hG},\mathrm{quot}}$, and hence on $\DDD^*_{\II_{\hF,\hG},\mathrm{quot}}$, is given by the character:
\begin{align}\label{eq:charquot}
	\chi_0 \cyc \dfrac{\kappa_\hF}{\epsilon_\hF} \left(\dfrac{\kappa_\hG}{\epsilon_\hG}\right)^{-1} : G_{\Q_p} \rightarrow \II_{\hF,\hG}^\times, 
\end{align}

 where
 \begin{itemize}
 \item 	 $\chi_0:G_{\Q_p} \rightarrow \OO^\times \hookrightarrow \II_{\hF,\hG}^\times$ is a finite-order character, 
 \item $\kappa_\hF:G_{\Q_p} \twoheadrightarrow \Gal{\Q_{p,\Cyc}}{\Q_p} \hookrightarrow \OO[\![x_\hF]\!]^\times \hookrightarrow \II_{\hF,\hG}^\times$ and $\kappa_\hG:G_{\Q_p} \twoheadrightarrow \Gal{\Q_{p,\Cyc}}{\Q_p} \hookrightarrow \OO[\![x_\hG]\!]^\times \hookrightarrow \II_{\hF,\hG}^\times$ are the tautological characters,
 \item $\epsilon_\hF$ and $\epsilon_\hG$ are unramified characters of $G_{\Q_p}$ characterized by the properties that $\epsilon_\hF(\Frob_p) = a_p(\hF)$ and $\epsilon_\hG(\Frob_p)=a_p(\hG)$.
 \end{itemize}
 
Since the $U_p$-eigenvalue of the classical specializations of Hida families varies with weight, one can conclude that $a_p(\hF)$ and $a_p(\hG)$ (respectively) are non-constant elements of $\II_{\hF}$ and $\II_{\hG}$ (respectively), and hence do not belong to $\OO$. Note also that the extension $\Q_p(\mu_{p^\infty})/\Q_p$ is totally ramified. Consequently, the restriction of the character in equation (\ref{eq:charquot}) to the subgroup $\Gal{\overline{\Q}_p}{\Q_p(\mu_{p^\infty})}$ equals $\psi_0 \epsilon_\hF^{-1}\epsilon_\hG$ for some finite order character $\psi_0$. Since $\II_{\hF} \cap \II_{\hG} = \OO$, the character $\psi_0 \epsilon_\hF^{-1}\epsilon_\hG$ of $\Gal{\overline{\Q}_p}{\Q_p(\mu_{p^\infty})}$ must be of infinite order, and hence non-trivial. This observation provides a non-zero element in $\II_{\hF,\hG}$ that annihilates $H^0\left(\Q_p, \DDD^*_{\II_{\hF,\hG}[\![\Gamma_\Cyc]\!],\mathrm{quot}}\right)^\vee$, and hence a non-zero annihilator $r$ of the $\II_{\hF,\hG}[\![\Gamma_\Cyc]\!]$-module in equation (\ref{eq:iotaiso}). Since $\mathrm{Ann}_{\II_{\hF,\hG}[\![\Gamma_\Cyc]\!]}  H^0\left(\Q_p, \DDD^*_{\II_{\hF,\hG}[\![\Gamma_\Cyc]\!],\mathrm{quot}}\right)^\vee \supseteq (\beta, r)$, 
the annihilator ideal in $\II_{\hF,\hG}[\![\Gamma_\Cyc]\!]$ must have height $\geq 2$. Thus, the $\II_{\hF,\hG}[\![\Gamma_\Cyc]\!]$-module $H^0\left(\Q_p, \DDD^*_{\II_{\hF,\hG}[\![\Gamma_\Cyc]\!],\mathrm{quot}}\right)^\vee$ must be pseudo-null. The claim now follows from a direct application of \cite[Proposition 5.7]{MR2290593}.

\end{proof}

We can now deduce the following short exact sequence from the long exact sequence in equation (\ref{eq:longexseq}), Claim \ref{claim:notorsion} and the fact that $\mathrm{coker}(\phi_{\Sel_{\mathrm{str},(i)}})^\vee $ is a co-torsion $\II_{\hF,\hG}[\![\Gamma_\Cyc]\!]$-module for each $i \in \{1,2\}$:
\begin{align}\label{eq:sesdom}
0 \rightarrow H^1_{\mathrm{cont}}\left(\Q_p, \ \dfrac{\DDD_{\II_{\hF,\hG}[\![\Gamma_\Cyc]\!],\Fil_{\mathrm{sum}}^+}}{\DDD_{\II_{\hF,\hG}[\![\Gamma_\Cyc]\!],\Fil_{(i)}^+}} \right) ^\vee \rightarrow \Sel_{\mathrm{str},\mathrm{sum}}( \DDD_{\II_{\hF,\hG}[\![\Gamma_\Cyc]\!]})^\vee \rightarrow 	\Sel_{\mathrm{str},(i)}( \DDD_{\II_{\hF,\hG}[\![\Gamma_\Cyc]\!]})^\vee \rightarrow 0.
\end{align}
This, in turn, allows us to deduce the following inclusions of torsion $\II_{\hF,\hG}[\![\Gamma_\Cyc]\!]$-modules:
\begin{align}
\left(\Sel_{\mathrm{str},\mathrm{sum}}( \DDD_{\II_{\hF,\hG}[\![\Gamma_\Cyc]\!]})^\vee\right)_{\mathrm{tor}} &\hookrightarrow \Sel_{\mathrm{str},(1)}( \DDD_{\II_{\hF,\hG}[\![\Gamma_\Cyc]\!]})^\vee, \\ \left(\Sel_{\mathrm{str},\mathrm{sum}}( \DDD_{\II_{\hF,\hG}[\![\Gamma_\Cyc]\!]})^\vee\right)_{\mathrm{tor}} &\hookrightarrow \Sel_{\mathrm{str},(2)}( \DDD_{\II_{\hF,\hG}[\![\Gamma_\Cyc]\!]})^\vee.
\end{align}
Hypothesis \ref{lab:gcd} now lets us conclude that the $\II_{\hF,\hG}[\![\Gamma_\Cyc]\!]$ torsion-submodule $\left(\Sel_{\mathrm{str},\mathrm{sum}}( \DDD_{\II_{\hF,\hG}[\![\Gamma_\Cyc]\!]})^\vee\right)_{\mathrm{tor}}$ is pseudo-null. Since the localization of the ring $\II_{\hF,\hG}[\![\Gamma_\Cyc]\!]$ at a height one prime ideal $\p$ is a DVR, one can use the structure theorem for DVRs. Equation (\ref{eq:rank1crit}) allows us to deduce the following isomorphism for all height one prime ideals $\p$ in $\II_{\hF,\hG}[\![\Gamma_\Cyc]\!]$:
\begin{align}\label{eq:rank1free}
\left(\Sel_{\mathrm{str},\mathrm{sum}}( \DDD_{\II_{\hF,\hG}[\![\Gamma_\Cyc]\!]})^\vee\right)_\p 	\cong \left(\II_{\hF,\hG}[\![\Gamma_\Cyc]\!]\right)_\p.  
\end{align}
The theorem now follows from equation (\ref{eq:rank1free}) and the surjections of $\II_{\hF,\hG}[\![\Gamma_\Cyc]\!]$-modules:
\[\Sel_{\mathrm{str},\mathrm{sum}}( \DDD_{\II_{\hF,\hG}[\![\Gamma_\Cyc]\!]})^\vee \twoheadrightarrow 	\Sel_{\mathrm{str},(i)}( \DDD_{\II_{\hF,\hG}[\![\Gamma_\Cyc]\!]})^\vee, \qquad  i \in \{1,2\},\] given in equation (\ref{eq:sesdom}).

\begin{remark}
The proof of Theorem \ref{thm:selmstructure} is motivated by \cite[Proposition 4.1]{MR3993809}, which is obtained under more restrictive hypotheses than ours. This necessitates adapting the proof of Claim \ref{claim:notorsion} in our setting since we are working with fewer hypotheses than \cite{MR3993809}.
\end{remark}

\section*{Acknowledgements}

We are grateful to Shih-Yu Chen, Shaunak Deo, Mladen Dimitrov, Radhika Ganapathy, Jaclyn Lang and  Abhishek Saha for helpful discussions at various stages of the project.  We are grateful to the anonymous referee for their thorough reading 
and detailed comments, which have led to numerous improvements 
throughout the manuscript. The research of M.-L. Hsieh is partially supported by NSTC grants 113-2628-M-002-017 and 114-2115-M-002 -006 -MY3. BP's research is partially supported by the Infosys Young Investigator Award from the Infosys Foundation Bangalore along with the SERB-MATRICS grant MTR/2022/000244 and DST FIST program 2021 [TPN - 700661].
 During the initial stages of this project, BP was a postdoctoral fellow at the National Center for Theoretical Sciences, Taiwan. NCTS also supported our collaboration by hosting BP in July 2023. BP would like to thank the faculty and staff affiliated with NCTS, especially Jungkai Chen, Yng-Ing Lee and Peggy Lee, for their constant support and encouragement. 

\appendix

\section{Miscellaneous results}

In this appendix, we collect proofs of a few miscellaneous results, whose proofs will be well-known to the experts. We have included them for the ease of the reader. 

Let $\RRR$ be a reduced ring that is a finite integral extension of $\Lambda$, fixing a congruence class $(a,b) \pmod{(p-1)\Z^2}$ (as in Section~\ref{sec:hidatheoryGSP4}).

\begin{lemma} \label{lem:density} 
 The set $\frakX^{\rm temp}_{\charcomp}\left(\RRR\right)$ is dense in $\Spec(\RRR)$. That is, for every non-zero element $\theta$ in $\RRR$, there exists a prime ideal $\p$ in $\frakX^{\rm temp}_{\charcomp}\left(\RRR\right)$ such that $\theta$ does not belong to $\p$. As a result, 
we have 

\begin{align}\label{eq:intersectionofallclassical}
\bigcap_{\mathfrak{p} \in \frakX^{\rm temp}_{\charcomp}\left(\RRR\right)}\mathfrak{p}  = \{0\}.
\end{align}	

\end{lemma}

\begin{proof}
Let $\theta$ be an element of $\RRR$. Since $\RRR$ is a reduced ring, there exists a minimal prime $\eta$ such that $\theta \notin \eta$. As a result, the image of $\theta$ in the integral domain $\RRR/\eta$ is non-zero.  By considering norms over the finite integral extension $\Lambda \hookrightarrow \RRR/\eta$,  we may assume without loss of generality that $\RRR$ equals $\Lambda$.  We will identify $\Lambda$ with $\Z_p\powerseries{X_1,X_2}$. Furthermore, we will also work with the natural identification of the power series ring $\Z_p\powerseries{X_1,X_2}$ with $\Z_p[\![X_2]\!][\![X_1]\!]$.  Let $\theta$ be a non-zero element in $\Lambda$. Let us write 
\[\theta = \sum_{n=0}^\infty a_n(X_2) X_1^n \in \Z_p[\![X_2]\!][\![X_1]\!].\]
	
Here, $a_n(X_2)$ belongs to $\Z_p[\![X_2]\!]$. One of the coefficients $a_n(X_2)$ must then be non-zero. Without loss of generality, assume that $a_0(X_2) \neq 0$. For each positive integer $k_2 > 3$, the assignment $X_2 \mapsto (1+p)^{k_2}-1$ determines a height one prime ideal in $\Z_p[\![X_2]\!]$. Since $a_0(X_2)$ is non-zero, for all but finitely many  such assignments with the weight $k_2 > 3$, the element $a_0(X_2)$ must map to a non-zero element in $\Z_p$. Choose one such ring homomorphism $\alpha_{j_0}: \Z_p[\![X_2]\!]\rightarrow \Z_p$, with weight $j_0 \gg 0$ and $j_0 \equiv b \ (\mathrm{mod} \ p-1)$ along with $\alpha_{j_0}(a_0(X_2)) \neq 0$, determined by this assignment. Applying $\alpha_{j_0}$ to each coefficient of $\theta$ gives us a non-zero element $\alpha_{j_0}(\theta)$ in $\Z_p[\![X_1]\!]$. Once again, we are able to choose a ring homomorphism with  $\beta_{k_0}:\Z_p[\![X_1]\!] \rightarrow \Z_p$, with weight $k_0 > j_0$ determined by the assignment $X_1 \mapsto (1+p)^{k_0}$, such that $k_0 \equiv a \ (\mathrm{mod} \ p-1)$  and $\beta_{k_0}\left(\alpha_{j_0}\left(\theta\right)\right) \neq 0$.  The kernel of the composite map given below must be $P_{(k_0,j_0)}$.   
\begin{align*}
\xymatrix@C=2pc @R=.4pc{
\Z_p[\![X_2]\!][\![X_1]\!] \ar[r]^{\alpha_{j_0}} & \Z_p[\![X_1]\!] \ar[r]^{\beta_{k_0}} & \Z_p, \\
\theta \ar[rr]&& \beta_{k_0}\left(\alpha_{j_0}\left(\theta\right)\right).
}
\end{align*}

This shows that $\theta$ does not belong to $P_{(k_0,j_0)}$. The lemma follows. 
\end{proof}

Let $\bbT=\hhh(N)_\m$ be local component of the Hecke algebra $\hhh(N)$ on the space of ordinary $\Lambda$-adic cuspidal Siegel modular forms.
\begin{theorem}\label{thm:pseudochar}
There exists a unique continuous pseudocharacter $\pcT: G_S \rightarrow \bbT$ such that for all primes $\ell$ not in $S$, we have $\pcT(\Frob_\ell) = T_\ell$. Here, $N$ is the tame level and $S$ is a finite set of primes of $\Q$ as in Section~\ref{sec:pseudocharactersGMA}.
\end{theorem}

\begin{proof}
The idea of the proof follows the use of pseudocharacters by Wiles \cite{MR0969243}. For each prime ideal $\p$ in $\frakX^{\rm cls}_{\charcomp}\left(\bbT\right)$, the quotient ring $\bbT/\p$ is a finitely generated module over $\Z_p$. The quotient topology on $\bbT/\p$ coincides with the $p$-adic topology. We may thus identify the fraction field of $\bbT/\p$ with a finite extension $E_\p$ of $\Q_p$. Observe that the quotient $\Lambda/(\p\cap\Lambda)$ is canonically identified with $\Z_p$. We endow the infinite direct product \[\prod \limits_{\p \in \frakX^{\rm cls}_{\charcomp}\left(\bbT\right)} \bbT/\p\] with the direct product topology. Using Corollary \ref{cor:densityhomega}, one can conclude that the intersection $\bigcap \limits_{\p \in \frakX^{\rm cls}_{\charcomp}\left(\bbT\right)} \p$ must equal $\{0\}$.  As a result, the diagonal embedding 
\begin{align*}
i: \bbT &\hookrightarrow \prod \limits_{\p \in \frakX^{\rm cls}_{\charcomp}\left(\bbT\right)} \bbT/\p
\end{align*}
must be a continuous injective map. Furthermore, since $\bbT$ is profinite (and hence compact), the image of the diagonal embedding must be closed in $\prod \limits_{\p \in  \frakX^{\rm cls}_{\charcomp}\left(\bbT\right)} \bbT/\p$. Under the map $i$, the image of $T_{\ell}$ corresponds to 
\[\left(\varphi_\p(T_{\ell})\right)_\p.\]
As indicated in Corollary \ref{cor:everyspecclassical}, the map $\varphi_\p:\bbT \rightarrow \bbT/\p \hookrightarrow E_\p$ corresponds to a   classical specialization determined (up to Galois conjugacy) by $\p$. For each  prime ideal $\p$ in $ \frakX^{\rm cls}_{\charcomp}\left(\bbT\right)$, by Theorem \ref{thm:GSP4weissauertaylorlaumon}, we have an associated continuous semi-simple Galois representation $\rho_\p: G_S \rightarrow \mathrm{GL}_4(E_\p)$. This Galois representation is unramified at all primes $\ell$ not belonging to $S$. Furthermore, $\mathrm{Trace}(\rho_\p)(\Frob_\ell)$ equals $\varphi_\p(T_{\ell})$. Since the $\Frob_\ell$'s generate a dense subset of $G_S$, the image of the associated trace function must lie in $\bbT/\p$, since $\bbT/\p$ is closed in $E_\p$. The associated trace function $\mathrm{Trace}(\rho_\p):G_S \rightarrow \bbT/\p$ is a continuous pseudocharacter of dimension $4$. Since the dimension of the pseudocharacter remains constant as we vary over the classical primes $\p$'s, the following function must be a continuous pseudocharacter: 
\begin{align}
\xymatrix@C+6pc{ G_S \ar[r]^{\prod \limits_{\p \in  \frakX^{\rm cls}_{\charcomp}\left(\bbT\right)}\mathrm{Trace}(\rho_\p)\quad} & \prod \limits_{\p \in  \frakX^{\rm cls}_{\charcomp}\left(\bbT\right)} \bbT/\p, \\
\Frob_\ell \ar@{|->}[r]&  \left(\varphi_\p(T_{\ell})\right)_\p. 
}
\end{align}

The images of the $\Frob_\ell$'s land inside the closed subset $\bbT$ of $\prod \limits_{\p \in  \frakX^{\rm cls}_{\charcomp}\left(\bbT\right)} \bbT/\p$. Since the $\Frob_\ell$'s generate a dense subset of $G_S$, the image of the entire continuous pseudocharacter lands in $\bbT$. We get an induced continuous pseudocharacter $\pcT:G_S \rightarrow \bbT$, such that $\pcT(\Frob_\ell)$ equals $T_{\ell}$. The uniqueness of the pseudocharacter follows since the $\Frob_\ell$'s generate a dense subset of $G_S$. 
\end{proof}

\bibliographystyle{amsalpha}
\bibliography{yoshida_bib}

\end{document}